\documentclass[11pt,reqno]{amsart}

\usepackage{amsmath, amsfonts, amsthm, amssymb}
\usepackage[usenames]{color}

\newtheorem{Theorem}{Theorem}[section]
\newtheorem{Lemma}{Lemma}[section]
\newtheorem{Proposition}{Proposition}[section]

\theoremstyle{definition}
\newtheorem{Definition}{Definition}[section]

\theoremstyle{remark}

\numberwithin{equation}{section}

\renewcommand{\u}{{\bf u}}
\renewcommand{\H}{{\bf H}}
\newcommand{\R}{{\mathbb R}}
\newcommand{\Dv}{{\rm div}}

\newcommand{\tr}{{\rm tr}}

\newcommand{\dl}{\delta}

\def\f{\frac}
\renewcommand{\O}{\Omega}

\def\D{\Delta }

\def\hf1{^\f{1}{1-\xi^2}}

\def\be{\begin{equation}}
\def\en{\end{equation}}
\def\bs{\begin{split}}
\def\es{\end{split}}
\newcommand{\PP}{\mathbb{P}}

\renewcommand{\v}{{\bf v}}

\newcommand{\supp}{{\rm supp}}

\author{Hengrong Du, \ Xianpeng Hu, \ Changyou Wang}
\address{Department of Mathematics, Purdue University, West Lafayette, IN, 47907, USA.} 
\email{du155@purdue.edu}
\address{Department of Mathematics, City University of Hong Kong, Hong Kong, PRC.} \email{xianpehu@cityu.edu.hk}
\address{Department of Mathematics, Purdue University, West Lafayette, IN, 47907, USA.} 
\email{wang2482@purdue.edu}

\title[Q-tensor]
{Suitable weak solutions for  the co-rotational Beris-Edwards system in dimension three}

\date{\today}

\begin{document}

\begin{abstract}
In this paper, we establish the global existence of a suitable weak solution
to the co-rotational Beris-Edwards $Q$-tensor system modeling the hydrodynamic motion of
nematic liquid crystals with either Landau-De Gennes bulk potential in $\mathbb R^3$ 
or Ball-Majumdar bulk potential in $\mathbb{T}^3$, a system coupling  the forced incompressible Navier-Stokes equation with
a dissipative, parabolic system of Q-tensor $Q$ in $\mathbb R^3$,
which is shown to be smooth away from a closed set $\Sigma$ 
whose $1$-dimensional parabolic Hausdorff measure is zero. 
\end{abstract}

\maketitle

\section{Introduction}

In this paper, we consider in dimension three the so-called Beris-Edwards system (\cite{BE} and \cite{DE})
that describes the hydrodynamic motion of nematic liquid crystals,
with either the Landau-De Gennes bulk potential function \cite{GP}
 or the Maire-Saupe (Ball-Majumdar) bulk potential function \cite{BM}. 
Roughly speaking, this is a system that couples a forced Navier-Stokes equation for the underlying
fluid velocity field $u$ with a dissipative parabolic system of $Q$-tensors modeling nematic liquid crystal director fields.
We are interested in establishing the existence of certain global weak solutions for such a Beris-Edwards system
that enjoys partial smoothness property, analogous to the celebrated works by Cafferalli-Kohn-Nirenberg \cite{CKN}
on the Navier-Stokes equation and Lin-Liu \cite{LL1} and \cite{LL2} on the simplified Ericksen-Leslie system modeling nematic liquid crystal flows with variable degree of orientations, which was proposed by Ericksen \cite{E1, E2} and Leslie \cite{Leslie} back in 1960's. 

We begin with the description of this system. Recall that the configuration space of $Q$-tensors is the set of traceless, symmetric $3\times 3$-matrices, defined by
$$\mathcal{S}_0^{(3)}=\Big\{Q\in \mathbb{R}^{3\times 3}: \ Q=Q^{\top}, \ {\rm{tr}} Q=0\Big\}.$$
For technical reasons, we will consider the one constant approximate form of the Landau-De Gennes energy functional of $Q$-tensors,  namely,
$$
E(Q)=\int_\Omega \big(\frac{L}2 |\nabla Q|^2+F_{\rm{bulk}}(Q)\big)\,dx,$$
on the Sobolev space $H^1(\Omega, \mathcal{S}_0^{(3)})$, where $\O$ is a three dimensional domain that is either
$\mathbb R^3$ or the torus $\mathbb{T}^3=\mathbb{R}^3/\mathbb{Z}^3$. Here $L>0$ denotes the elasticity constant, and
$F_{\rm{bulk}}(Q)$ denotes the bulk potential function that usually describes the phase transition among various phase states
including isotropic, uniaxial, or biaxial states. We refer the interested readers to Mottram-Newton 
\cite{MN} and Sonnet-Virga \cite{SV} for a more detailed discussion of general  Landau-De Gennes energy functionals involving multiple elasticity constants $L_i$'s.  In this paper, we will consider two
classes of bulk potential functions:
\begin{enumerate}
\item [(i)] (Landau-De Gennes bulk potential \cite{GP}). Here $F_{\rm{bulk}}(Q)=F_{\rm{LdG}}(Q)$, where 
\begin{equation}\label{de gennes1}
\displaystyle F_{\rm{LdG}}(Q)=\widehat{F}_{\rm{LdG}}(Q)-\min_{Q'\in\mathcal{S}_0^{(3)}}\widehat{F}_{\rm{LdG}}(Q'),
\end{equation}
and
\begin{equation}\label{de gennes2}
\widehat{F}_{\rm{LdG}}(Q)=\frac{a}{2} {\rm{tr}}(Q^2)-\frac{b}3 {\rm{tr}}(Q^3)+\frac{c}4 {\rm{tr}}^2(Q^2),
\end{equation}
where $a,b,c>0$ are temperature dependent material constants. It is a well known fact that if
$0<a<\frac{b^2}{27c}$, then
$\widehat{F}_{\rm{LdG}}$ reaches
the minimum when $Q=s_+(d\otimes d-\frac13I_3)$, where $s_+=\frac{b+\sqrt{b^2-24ac}}{4c}$ and $d\in\mathbb{S}^2$
is a unit vector field.
\item [(ii)] (Ball-Majumdar singular bulk potential \cite{BM}).  Here $F_{\rm{bulk}}(Q)=F_{\rm{BM}}(Q)$ is a modified Maire-Saupe bulk potential introduced by Ball-Majumdar \cite{BM}, which is defined as follows. $F_{\rm{BM}}(Q)=\nu G_{\rm{BM}}(Q)-\frac{\kappa}{2}|Q|^2$
for some $\nu>0$ and $\kappa>0$, and
 \begin{equation}\label{BMP}
 G_{\rm{BM}}(Q)\equiv
    \left\{
      \begin{array}{ll}\displaystyle
	\min_{\rho\in \mathcal{A}_{Q}}\int_{\mathbb{S}^{2}}\rho(p)\log \rho(p)\,d\sigma(p) & \text{if }-\frac{1}{3}<\lambda_j(Q)<\frac{2}{3}, \\
	\infty & \text{otherwise}, 
      \end{array}
      \right.
    \end{equation}
    $\lambda_j$ ($j=1,2,3$) denote the eigenvalues of $Q\in \mathcal{S}^{(3)}_0$, and 
    $$\mathcal{A}_Q\equiv\Big\{ \rho\in L^1(\mathbb{S}^{2},\mathbb{R}_+):\ \int_{\mathbb{S}^{2}}\rho(p)\,d\sigma(p)=1,\ 
    \int_{\mathbb{S}^{2}}\big( p\otimes p-\frac{1}{3}I_3 \big)\rho(p)\,d\sigma(p) =Q\Big\}.$$
    It was proven by \cite{BM} that $G_{\rm{BM}}$ is strictly convex and smooth in the interior of the convex set
    $$\mathcal{D}=\Big\{Q\in \mathcal{S}_0^{(3)}: \ -\frac13\le \lambda_i(Q)\le\frac23, \ i=1,2,3\Big\}.$$
\end{enumerate}

It is well-known that the first order variation of the Landau-De Gennes energy functional is given by
\begin{equation}\label{firstvar}
H=L\D Q-f_{\rm{bulk}}(Q), \  \ f_{\rm{bulk}}(Q)=\langle\nabla F_{\rm{bulk}}(Q)\rangle=\nabla F_{\rm{bulk}}(Q)
-\f{{\rm{tr}} (\nabla F_{\rm{bulk}}(Q))}{3}I_3.
\end{equation}
In particular, if $F_{\rm{bulk}}(Q)=F_{\rm{LdG}}(Q)$, then
$$
f_{\rm{bulk}}(Q)=\langle\nabla F_{\rm{LdG}}(Q)\rangle=aQ-b\big[Q^2-\f{\tr(Q^2)}{3}I_3\big]+cQ\tr(Q^2).$$

For $0<T\le\infty$, denote $Q_T=\Omega\times (0,T]$.  Let ${\bf u}: Q_T\mapsto\mathbb R^3$ denote the fluid velocity field and
$Q:Q_T\mapsto \mathcal{S}_0^{(3)}$ denote the director field. Define
\begin{equation*}
\begin{split}
S(\nabla\u, Q)=(\xi D+\omega)\big(Q+\f{1}{3}I_3\big)+\big(Q+\f{1}{3}I_3\big)(\xi D-\omega)-2\xi\big(Q+\f{1}{3}I_3\big)\tr(Q\nabla\u),
\end{split}
\end{equation*}
where $$D=\f12(\nabla\u+(\nabla\u)^\top) \ \ {\rm{and}}\ \ \omega=\f12(\nabla\u-(\nabla\u)^\top)$$
are the symmetric part and the antisymmetric part, respectively, of the velocity gradient tensor $\nabla\u$, 
and $\xi\in\mathbb{R}$ is a rotational parameter measuring the ratio between the aligning and tumbling effects to $Q$ by the fluid velocity field.

The Beris-Edwards $Q$-tensor system modeling the hydrodynamic motion of nematic liquid crystals reads \cite{GR1,PZ}
\begin{equation}\label{e}
\begin{cases}
\partial_tQ+\u\cdot\nabla Q-S(\nabla\u, Q)=\Gamma H\\
\partial_t\u+\u\cdot\nabla\u+\nabla P=\mu\D\u+\Dv(\tau+\sigma)\\
\Dv\u=0,
\end{cases}
\end{equation}
where $\Gamma>0$ is a relaxation time parameter, $\mu>0$ is the fluid viscosity constant, and $\tau$ is the symmetric part of the additional stress tensor given by
\begin{equation*}
\begin{split}
\tau_{\alpha\beta}&=-\xi\big(Q_{\alpha\gamma}+\f{\dl_{\alpha\gamma}}{3}\big)H_{\gamma\beta}
-\xi H_{\alpha\gamma}\big(Q_{\gamma\beta}+\f{\dl_{\gamma\beta}}{3}\big)\\
&\quad+2\xi\big(Q_{\alpha\beta}+\f{\dl_{\alpha\beta}}{3}\big)Q_{\gamma\dl}H_{\gamma\dl}-L\partial_\beta Q_{\gamma\dl}\partial_\alpha Q_{\gamma\dl},
\ 1\le\alpha,\beta\le 3, 
\end{split}
\end{equation*}
and $\sigma$ is the  antisymmetric part of the additional stress tensor:
$$\sigma_{\alpha\beta}=Q_{\alpha\gamma}H_{\gamma\beta}-H_{\alpha\gamma}Q_{\gamma\beta},
\ 1\le\alpha,\beta\le 3.$$

In this paper, we will focus on the co-rotational  Beris-Edwards system \eqref{e}, i.e., 
$$\fbox{$\xi=0$}$$
Since the exact values of $L, \Gamma, \mu$ don't play roles in our analysis, we will assume for simplicity
$$\fbox{$L=\Gamma=\mu=1$}$$
We will also assume the domain $\Omega$ to be 
$$\fbox{$\Omega=\begin{cases} \mathbb R^3 & if \ \  F_{\rm{bulk}}(Q)=F_{\rm{LdG}}(Q),\\
\mathbb{T}^3 & if \ \ F_{\rm{bulk}}(Q)=F_{\rm{BM}}(Q). 
\end{cases}
$}
$$
With these assumptions and the following identity:
$$\partial_\beta(\partial_\beta Q_{\gamma\dl}\partial_\alpha Q_{\gamma\dl} )
=\partial_\alpha Q_{\gamma\delta} \Delta Q_{\gamma\delta}+\partial_\alpha (\frac12|\nabla Q|^2),$$
the system \eqref{e} reduces to the following component-wise form (under the Einstein convention of summation).
\begin{equation}\label{e1}
\begin{cases}
&\partial_tQ_{\alpha\beta}+\u\cdot\nabla Q_{\alpha\beta}-\omega_{\alpha\gamma}Q_{\gamma\beta}+Q_{\alpha\gamma}\omega_{\gamma\beta}
=\D Q_{\alpha\beta}-f_{\rm{bulk}}(Q)_{\alpha\beta}\\
&\partial_t\u_{\alpha}+\u\cdot\nabla\u_{\alpha}+\partial_{\alpha} P=\D\u_{\alpha}-\partial_{\alpha}Q_{\xi\dl}\D Q_{\xi\dl}\\
&+\partial_\beta(Q_{\alpha\gamma}(\D Q_{\gamma\beta} -f_{\rm{bulk}}(Q)_{\gamma\beta})-
(\D Q_{\alpha\gamma}- f_{\rm{bulk}}(Q)_{\alpha\gamma}) Q_{\gamma\beta})\\
&\Dv\u=0
\end{cases} \ {\rm{in}}\  \Omega\times (0,\infty)
\end{equation}
subject to the initial condition
\begin{equation}\label{IC}
(\u, Q)|_{t=0}=({\u}_0,{Q}_0)(x)\quad\textrm{for}\quad x\in\Omega.
\end{equation}
A key feature of the Beris-Edwards system \eqref{e1} (or \eqref{e} in general) is the energy dissipation property, which plays a fundamental role in the analysis
of \eqref{e1}. More precisely,
if $(\u, Q):\Omega\times (0,\infty)\mapsto \mathbb R^3\times \mathcal{S}_0^{(3)}$ is a sufficiently regular solution
of \eqref{e}, then it satisfies the following energy inequality \cite{PZ, PZ1}:
\begin{equation}\label{genergy}
\begin{split}
\f{d}{dt}E(\u, Q)(t)=-\int_{\Omega}(|\nabla\u|^2+|H|^2)(x,t)\,dx
\end{split}
\end{equation}
where 
\begin{equation}\label{e1energy}
E(\u, Q)(t)=\int_{\Omega}(\frac12 |\u|^2+\frac12|\nabla Q|^2+{F}_{\rm{bulk}}(Q))(x,t)\,dx
\end{equation}
is the total energy of the complex fluid consisting of the elastic energy of the director field $Q$ and the kinetic energy of the underlying fluid $\u$.
While the right hand side of \eqref{genergy} denotes the dissipation rate of this system of complex fluid.

\bigskip
\noindent{\bf Some Notations.} For $Q\in \mathcal{S}_0^{(3)}$,  we use the Frobenius norm of $Q$, i.e.
$$|Q|=\sqrt{\tr (Q^2)}=\sqrt{Q_{\alpha\beta}Q_{\alpha\beta}},$$ 
and the Sobolev spaces  of $Q$-tensors, $W^{l,p}\big(\Omega, \mathcal{S}_0^{(3)}\big)$ ($l\in\mathbb{N}_+$ and $1\le p\le \infty$), are defined by
$$W^{l,p}\big(\Omega, \mathcal{S}_0^{(3)}\big)=\Big\{Q=(Q_{\alpha\beta}):\Omega\mapsto\mathcal{S}_0^{(3)}:
\ Q_{\alpha\beta}\in W^{l,p}(\Omega), \ \forall 1\le \alpha,\beta\le 3\Big\}.$$
When $p=2$, we denote $W^{l,2}\big(\Omega, \mathcal{S}_0^{(3)}\big)$ by $H^l(\Omega,\mathcal{S}_0^{(3)})$.
For $A, B\in  \R^{3\times 3}$, we denote 
$$A:B=A_{\alpha\beta}B_{\alpha\beta},\ A\cdot B=\tr(AB),\   |\nabla Q|^2=Q_{\alpha\beta,\gamma}Q_{\alpha\beta,\gamma}, 
\ |\D Q|^2=\D Q_{\alpha\beta}\D Q_{\alpha\beta},$$
and
$$ (\u\otimes\u)_{\alpha\beta}=\u_\alpha \u_\beta,\ \ (\nabla Q\otimes\nabla Q)_{\alpha\beta}=\nabla_\alpha Q_{\gamma\delta} \nabla_\beta Q_{\gamma\delta}.$$
Note that $A:B=A\cdot B$ for $A, B\in \mathcal{S}_0^{(3)}$.

Define
$${\H}={\rm{Closure \ of}}\ \Big\{\u\in C_0^\infty(\Omega,\mathbb R^3): \ \Dv\u=0\Big\} \ {\rm{in}}\ L^2(\Omega),$$
and
$${\bf V}={\rm{Closure \ of}}\ \Big\{\u\in C_0^\infty(\Omega,\mathbb R^3): \ \Dv\u=0\Big\} \ {\rm{in}}\ H^1(\Omega).$$

For $0\le k\le 5$, $\mathcal{P}^k$ denotes the  $k$-dimensional Hausdorff measure on $\mathbb R^3\times\mathbb R_+$
with respect to the parabolic distance:
$$\delta((x,t), (y,s))=\max\Big\{|x-y|,\sqrt{|t-s|}\Big\}, \ \forall (x,t), \ (y,s)\in \mathbb R^3\times\mathbb R_+.$$

Now we would like to recall the definition of weak solutions of \eqref{e1}. 

\begin{Definition}\label{weak solution} A pair of functions $(\u, Q): \Omega\times (0,\infty)\mapsto \mathbb R^3\times \mathcal{S}_0^{(3)}$
is a weak solution of \eqref{e1} and \eqref{IC}, if 
$\u\in L^\infty_tL^2_x\cap L^2_tH^1_x(\Omega\times (0,\infty))$ and $Q\in L^\infty_tH^1_x\cap L^2_tH^2_x(\Omega\times (0,\infty))$,
and for any $\phi\in C_0^\infty\big(\Omega\times [0,\infty),\mathcal{S}_0^{(3)}\big)$ 
and $\psi\in C_0^\infty\big(\Omega\times [0,\infty),\mathbb R^3\big)$, with $\Dv \psi=0$ in $\Omega\times [0,\infty)$, it holds
\begin{equation}\label{ws1}
\begin{split}
&\int_{\Omega\times (0,\infty)} \big[-Q\cdot\partial_t\phi -\Delta Q\cdot\phi-Q\cdot \u\otimes\nabla\phi +Q\omega\cdot \phi-\omega Q\cdot\phi\big]\,dxdt\\
&=-\int_{\Omega\times (0,\infty)}f_{\rm{bulk}}(Q)\cdot\phi\,dxdt+\int_\Omega Q_0(x)\cdot\phi(x,0)\,dx,
\end{split}
\end{equation}
and
\begin{eqnarray}
\label{ws2}
&&\int_{\Omega\times (0,\infty)} \big[-\u\cdot\partial_t\psi +\nabla \u\cdot\nabla\psi-\u\otimes\u:\nabla\psi\big]\,dxdt=\nonumber\\
&&\int_{\Omega\times (0,\infty)}\big[- \D Q (\psi\cdot\nabla) Q+\big((\Delta Q-f_{\rm{bulk}}(Q)) Q-Q(\Delta Q-f_{\rm{bulk}}(Q))\big)\cdot\nabla\psi\big]\,dxdt\nonumber\\
&&\ \ +\int_\Omega \u_0(x)\cdot\psi(x,0)\,dx,
\end{eqnarray}
\end{Definition}

Paicu-Zarnescu \cite{PZ} have obtained the existence of global weak solutions to \eqref{e1} and \eqref{IC} in $\mathbb R^3$,
and the existence of global strong solutions to \eqref{e1} and \eqref{IC} in $\mathbb R^2$, when the bulk potential function is $F_{\rm{LdG}}(Q)$. 
For non-corotational Beris-Edwards system (i.e. $\xi\not=0$), Paicu-Zarnescu \cite{PZ1} have obtained the existence of global weak solutions to \eqref{e1} and \eqref{IC} in $\mathbb R^3$ for sufficiently small $|\xi|>0$. Later, Cavaterra-Rocca-Wu-Xu \cite{CRWX} have removed the smallness condition on $\xi$ for
 \eqref{e1} and \eqref{IC} in $\mathbb R^2$. Wilkinson \cite{Wilkinson} has obtained the existence of global  weak solutions to  \eqref{e1} and \eqref{IC}
 in three dimensional torus $\mathbb{T}^3$, when the bulk potential function is the Ball-Majumdar potential  $F_{\rm{BM}}(Q)$. The situation of Beris-Edwards system \eqref{e1} for the De Gennes potential $F_{\rm{LdG}}(Q)$ on bounded domains, under the initial-boundary condition,  behaves slightly different from that on $\mathbb R^3$. In fact, Abels-Dolzmann-Liu
 \cite{ADL1, ADL2} have established the well-posedness of \eqref{e} for any arbitrary constant $\xi$. See also \cite{FSRZ}
 for related works on nonisothermal Beris-Edwards system. We also mention an interesting work on the dynamics of
 $Q$-tensor system by Wu-Xu-Zarnescu \cite{WXZ}.  Interested readers can refer to Wang-Zhang-Zhang \cite{WZZ}
for a rigorous derivation from Landau-De Gennes theory to Ericksen-Leslie theory. For related works on the existence of global weak solutions to the simplified Ericksen-Leslie system, see \cite{LLW, LW1, LW2, HLW}. 
 
These previous works mentioned above left the question open that if certain weak solutions of \eqref{e} pose either smoothness
or partial smoothness properties. This motivates us to study both the existence of suitable weak solutions of \eqref{e1}
and their partial regularities.  The notion of suitable weak solutions was first introduced by Caffarelli-Kohn-Nirenberg
\cite{CKN} and Scheffer \cite{S} for the Navier-Stokes equation, and later extended by  Lin-Liu \cite{LL1, LL2} for the simplified Ericksen-Leslie system with variable degree of orientations. Here we introduce the notion of suitable weak solutions 
to the Beris-Edwards system as follows.

\begin{Definition} A weak solution $(\u,P, Q)\in (L^\infty_tL^2_x\cap L^2_tH^1_x)(\Omega\times(0,\infty), \mathbb R^3)
\times L^\frac32(\Omega\times (0,\infty))\times (L^\infty_t H^1_x\cap L^2_tH^2_x)(\Omega\times(0,\infty), \mathcal{S}_0^{(3)})$  of \eqref{e1} and \eqref{IC} is a suitable weak solution of \eqref{e1},
if, in addition, $(\u,P, Q)$ satisfies the local energy inequality: $\forall\ 0\le \phi\in C_0^\infty(\Omega\times (0,t])$,
\begin{eqnarray}\label{lei}
&&\int_{\Omega}(|\u|^2+|\nabla Q|^2)\phi(x,t)\,dx+2\int_{Q_t}(|\nabla \u|^2+|\D Q|^2)\phi(x,s)\,dxds\nonumber\\
&&\le \int_{Q_t}(|\u|^2+|\nabla Q|^2)(\partial_t\phi+\Delta\phi)(x,s)\,dxds\nonumber\\
&&+\int_{Q_t}[(|\u|^2+2P)\u\cdot\nabla\phi+2\nabla Q\otimes\nabla Q: \u\otimes\nabla\phi](x,s)\,dxds\\
&& +2\int_{Q_t}(\nabla Q\otimes\nabla Q-|\nabla Q|^2I_3):
\nabla^2\phi (x,s)\,dxds\nonumber\\
&& -2\int_{Q_t} \big((Q\Delta Q-f_{\rm{bulk}}(Q))-(\Delta Q-f_{\rm{bulk}}(Q)) Q\big)\cdot\u\otimes\nabla\phi(x,s)\,dxds\nonumber\\
&& -2\int_{Q_t}\big[(\omega Q-Q\omega)\cdot \nabla Q \nabla\phi
+\nabla ( {f}_{\rm{bulk}}(Q))\cdot\nabla Q \phi\big](x,s)\,dxds.\nonumber
\end{eqnarray}
\end{Definition}

The notion of suitable weak solutions turns out to be a necessary condition for the smoothness of \eqref{e1}.
In fact, the local energy inequality \eqref{lei} automatically holds for sufficiently regular solution of \eqref{e},
which can be obtained  by 
multiplying \eqref{e}$_2$ by $\u\phi$, and taking spatial derivative of \eqref{e}$_1$ and multiplying
the resulting equation by $\nabla Q\phi$, and then applying integration by parts, see Lemma 2.2 below
for the details. We would like to point out that  in the process of derivation of \eqref{lei},
the following cancellation identity:
\begin{equation}\label{can2}
\begin{split}
&\int_\Omega (Q\omega-\omega Q):(\Delta Q-f_{\rm{bulk}}(Q))\phi\,dx\\
&=- \int_\Omega \big(Q(\Delta Q-f_{\rm{bulk}}(Q))-(\Delta Q-f_{\rm{bulk}}(Q)) Q\big):\nabla \u\phi\,dx
\end{split}
\end{equation}
play critical roles.

\medskip
Now we are ready to state our main theorem, which is valid for the Beris-Edwards system associate 
with both the Landau-De Gennes bulk potential $F_{\rm{LdG}}(Q)$ in $\mathbb R^3$ and
Ball-Majumdar bulk potential $F_{\rm{BM}}(Q)$ in $\mathbb T^3$. More precisely,  we have

\begin{Theorem}\label{main} For any ${\bf u}_0\in {\bf H}$, if either \\
$(i)$ $\Omega=\mathbb R^3$, $F_{\rm{bulk}}(\cdot)=F_{\rm{LdG}}(\cdot)$ with $c>0$,
and $Q_0\in H^1(\mathbb R^3, \mathcal{S}_0^{(3)})\cap L^\infty(\mathbb R^3, \mathcal{S}_0^{(3)} )$, or\\
$(ii)$ $\Omega=\mathbb T^3$, $F_{\rm{bulk}}(\cdot)=F_{\rm{BM}}(\cdot)$,  and $Q_0\in H^1(\mathbb T^3, \mathcal{S}_0^{(3)})$
satisfies $G_{\rm{bulk}}(Q_0)\in L^1(\mathbb T^3)$,\\
then there exists a global
suitable weak solution $({\bf u}, P, Q):\Omega\times \mathbb R_+\mapsto\mathbb{R}^3\times\mathbb R\times\mathcal{S}_0^{(3)}$
of the Beris-Edwards system \eqref{e1}, subject to the initial condition \eqref{IC}. Moreover, 
$$({\bf u}, Q)\in C^\infty(\Omega\times (0,\infty)\setminus\Sigma),$$
where $\Sigma\subset\Omega\times\mathbb R_+$ is  a closed subset 
with $\mathcal{P}^1(\Sigma)=0$.
\end{Theorem}

We would like to highlight some crucial steps of the proof for Theorem \ref{main}:
\begin{enumerate}
\item The existence of suitable weak solutions to \eqref{e1} and \eqref{IC} 
is obtained by modifying the retarded mollification technique, 
originally due to \cite{S} and \cite{CKN} in the construction of suitable weak solutions to the Navier-Stokes
equation. 
\item For the Landau-De Gennes potential $F_{\rm{LdG}}(Q)$, we establish a weak maximum principle
of $Q$ for suitable weak solutions $(\u, P, Q)$ of \eqref{e1} and \eqref{IC} that bounds the $L^\infty$-norm
of $Q$ in $\mathbb R^3\times (0,\infty)$ in terms of that of initial data $Q_0$, see also \cite{GR1}. 
In particular, $\nabla^l_Q f_{\rm{LdG}}(Q)$ is also bounded in $\mathbb R^3\times (0,\infty)$ for $l\ge 0$.
\item For the Ball-Majumdar potential $F_{\rm{BM}}(Q)$,  we follow the approximation scheme of $G_{\rm{BM}}$
by Wilkinson \cite{Wilkinson} and use the convexity property of $G_{\rm{BM}}(Q)$ to bound
$$\|G_{\rm{BM}}(Q)\|_{L^\infty(\mathbb T^3\times [\delta,T])}, \ \forall 0<\delta<T<\infty,$$
in terms of $\|F_{\rm{BM}}(Q_0)\|_{L^1(\mathbb T^3)}$,
$\delta$,  and $T$. This guarantees that $Q$ is strictly physical in $\mathbb T^3\times [\delta,T]$, i.e., 
there exists a small $\gamma>0$, depending on $\delta, T$, such that
$$-\frac13+\gamma\le \lambda_j(Q(x,t)) \le \frac23-\gamma, \ j=1, 2,3, \ \forall (x,t)\in\mathbb T^3\times [\delta, T].$$
In particular, both $Q(x,t)$  and $f_{\rm{BM}}(Q(x,t))$ are bounded in
$\mathbb T^3\times [\delta,T]$.
\item Based on the local energy inequality \eqref{lei}, (2), and (3), we perform a blowing up argument 
to obtain an $\varepsilon_0$-regularity criteria of any suitable weak solution $(\u, P, Q)$ of \eqref{e1}, which asserts
that if 
\begin{equation}
\begin{split}
&\Phi(z_0, r):=\\
&r^{-2}\int_{\mathbb P_r(x_0,t_0)} (|\u|^3+|\nabla Q|^3)\,dxdt + \big(r^{-2}\int_{\mathbb P_r(x_0,t_0)} |P|^\frac32\,dxdt\big)^2
\le\varepsilon_0^3,
\end{split}
\end{equation} 
then $(x_0,t_0)\in \Omega\times (0,\infty)$ is a smooth point of $(\u, Q)$. The idea is to show that $(\u, P, Q)$ is well approximated by a smooth solution to a linear coupling system in the parabolic neighborhood 
$\mathbb{P}_{\frac{r}2}(x_0,t_0)$ of $(x_0,t_0)$, which heavily relies on the local energy inequality \eqref{lei}
and interior $L^\frac32$-estimate of the pressure function $P$, which turns out to solve the following Poisson equation:
\begin{equation}\label{P-eqn0}
-\Delta P={\rm{div}}^2(\u\otimes \u+(\nabla Q\otimes\nabla Q-\frac12|\nabla Q|^2I_3)) \ {\rm{in}}\ B_r(x_0).
\end{equation}
Here the following simple identity plays a crucial role in the derivation of \eqref{P-eqn0}.
\begin{equation}
\label{free-div2} {\rm{div}}^2\big(Q_1(\Delta Q_2-f_{\rm{bulk}}(Q_2))-(\Delta Q_2-f_{\rm{bulk}}(Q))Q_1\big)=0 \ {\rm{in}}\ B_r(x_0), 
\end{equation}
for $Q_1, Q_2\in H^2(B_r(x_0), \mathcal{S}_0^{(3)})$, whose proof is given in \S 2.

This blowing up argument implies that for some $\theta\in (0,1)$, $\Phi_{(x_*,t_*)}(r) \leq C r^{3\theta}$
for $(x_*, t_*)$ near $(x_0,t_0)$ , which can be used to further show 
that $(\u, \nabla Q)$ are almost bounded near $(x_0,t_0)$ by an iterated Reisz potential estimates
in the parabolic Morrey spaces, see also Huang-Wang \cite{HW},
Hineman-Wang \cite{HLW1}, and Huang-Lin-Wang \cite{HLW}. 
Higher order regularity of $(\u, Q)$ near $(x_0,t_0)$ turns out to be more involved
than the usual situations, due to the special nonlinearities. Here we establish  it by performing higher order
energy estimates and utilizing the intrinsic cancellation property, see also \cite{HLW} for a similar argument
on general Ericksen-Leslie system in dimension two. 
It is well-known \cite{S} that this step is sufficient to show that $(\u, Q)$ is smooth away from a closed set
$\Sigma$ which has $\mathcal{P}^\frac53(\Sigma)=0$.

\item  To obtain $\mathcal{P}^1(\Sigma)=0$ from the previous step, we adapt the
argument by \cite{CKN} to show  if
\begin{equation}
\displaystyle{{\overline\lim}}_{r\to 0} r^{-1}\int_{\mathbb P_r(x_0,t_0)} (|\nabla\u|^2+|\nabla^2 Q|^2)\,dxdt<\varepsilon_1^2,
\end{equation}
then $(\u, Q)\in C^\infty(\mathbb P_{\frac{r}2}(x_0,t_0))$.
This will be established by extending the  so called A, B, C, D Lemmas in \cite{CKN} to system \eqref{e1}. 
\end{enumerate}

The paper is organized as follows. In \S2, we derive both the global and local energy inequality for
sufficiently regular solutions of \eqref{e1}. In \S3, we indicate the construction of suitable weak solutions
to \eqref{e1} and \eqref{IC} for both Landau-De Gennes potential and Ball-Majumdar potential. In \S4, we prove two 
weak maximum principles for suitable weak solutions
to \eqref{e1} and \eqref{IC}: one for $Q$ and the other for $G_{\rm{BM}}(Q)$. 
In \S5, we prove the first $\varepsilon_0$-regularity of suitable weak solutions
to \eqref{e1} and \eqref{IC} in terms of $\Phi(z_0, r)$. 
In \S6, we will prove the second $\varepsilon_0$-regularity of suitable weak solutions
to \eqref{e1} and \eqref{IC} in terms of (1.17).

\section{Global and local energy inequalities}

In this section,  we will present  proofs for both global energy inequality and local energy inequality 
for sufficiently regular solutions to the Beris-Edwards system \eqref{e1}.

\begin{Lemma}\label{L0} Let $({\bf u}, Q)\in C^\infty(\Omega\times(0,\infty),
\mathbb R^3\times \mathcal{S}_0^{(3)})$ be a
smooth solution of Beris-Edwards system \eqref{e1}. Then the global energy inequality 
\eqref{genergy} holds.
\end{Lemma}
\begin{proof} Multiplying the  equation \eqref{e1}$_1$ by $H$ and integrating over $\Omega$, we obtain
\begin{eqnarray}\label{L00}
&&\frac{d}{dt}\int_{\Omega} (\frac12|\nabla Q|^2+F_{\rm{bulk}}(Q))\,dx+\int_{\Omega} |H|^2\,dx\nonumber\\
&&=\int_{\Omega} ({\bf u}\cdot\nabla Q):\Delta Q\,dx+\int_{\Omega} (Q\omega -\omega Q)\cdot(\Delta Q-f_{\rm{bulk}}(Q)) \,dx
-\int_{\Omega} \u\cdot\nabla (F_{\rm{bulk}}(Q))\,dx\nonumber\\
&&=\int_{\Omega} ({\bf u}\cdot\nabla Q):\Delta Q\,dx+\int_{\O} (Q\omega -\omega Q)\cdot(\Delta Q-f_{\rm{bulk}}(Q))\,dx. 
\end{eqnarray}
Now we multiply the equation \eqref{e1}$_2$ by $\u$ and integrate over $\O$ to obtain that
\begin{eqnarray}\label{L01}
&&\frac{d}{dt}\int_{\O} \frac12|{\bf u}|^2\,dx+ \int_{\O}|\nabla{\bf u}|^2\,dx\nonumber\\
&&=\int_{\O} (-\Delta Q\cdot \nabla Q)\cdot {\bf u}+\big(Q(\Delta Q-f_{\rm{bulk}}(Q))-(\Delta Q-f_{\rm{bulk}}(Q)) Q\big)\cdot\nabla {\bf u}\,dx.
\end{eqnarray}
Note that direct calculations yield the following identity:
$$
\int_{\O} (Q\omega -\omega Q)\cdot(\Delta Q-f_{\rm{bulk}}(Q))\,dx=-\int_{\O} \big(Q(\Delta Q-f_{\rm{bulk}}(Q))-(\Delta Q-f_{\rm{bulk}}(Q)) Q\big)\cdot\nabla {\bf u}\,dx.
$$
Therefore, by adding \eqref{L00} and \eqref{L01}, we obtain \eqref{genergy}. This completes the proof.
\end{proof}
Next we are going to present a local energy inequality for sufficiently regular solutions to the system \eqref{e1}. 
\begin{Lemma}\label{L1} Assume $(\u, P, Q)\in C^\infty(\Omega\times (0,\infty),
\mathbb R^3\times \mathbb R\times \mathcal{S}_0^{(3)})$ is a smooth solution of \eqref{e1}.
Then for $t>0$ and any nonnegative $\phi\in C_0^\infty(\Omega\times (0,t])$, 
the following inequality holds on $Q_t=\Omega\times [0,t]$:
\begin{equation}\label{lei1}
\begin{split}
&\int_{\O}\big(|\u|^2+|\nabla Q|^2\big)\phi(x,t)\,dx+2\int_{Q_t}\Big(|\nabla\u|^2+|\D Q|^2\Big)\phi \,dxds\\
&=\int_{Q_t}\Big(|\u|^2+|\nabla Q|^2\Big)(\partial_t+\D)\phi\,dxds\\
&+\int_{Q_t}[(|\u|^2+2P)\u\cdot\nabla\phi+2(\nabla Q\otimes\nabla Q):\u\otimes\nabla\phi]\, dxds\\
&+2\int_{Q_t}[(\nabla Q\otimes\nabla Q-|\nabla Q|^2 I_3):\nabla^2\phi\,dxds\\
&-2\int_{Q_t} (Q(\D Q-f_{\rm{bulk}}(Q))-(\D Q-f_{\rm{bulk}}(Q)) Q):\u\otimes\nabla\phi]\, dxds\\
&-2\int_{Q_t}\Big[(\omega Q-Q\omega):\nabla Q\nabla\phi +\nabla(f_{\rm{bulk}}(Q))\cdot \nabla Q\phi\Big]\, dxds.
\end{split}
\end{equation}
\end{Lemma}
\begin{proof}
Using $\Dv\u=0$, we multiply the momentum equation \eqref{e1}$_2$ by $\u\phi$, integrate 
the resulting equation over $\O$,  and apply integration by parts to obtain
\begin{equation}\label{1}
\begin{split}
&\f12\f{d}{dt}\int_{\O}|\u|^2\phi \,dx+\int_{\O}|\nabla\u|^2\phi \,dx\\
&=\f12\int_{\O}|\u|^2(\partial_t\phi+\D\phi)dx+\f12\int_{\O}(|\u|^2+2P)\u\cdot\nabla\phi \,dx
-\int_{\O} (\u\cdot\nabla) Q\cdot \D Q\phi\,dx\\
&\quad-\int_{\O}(Q(\D Q-f_{\rm{bulk}}(Q))-(\D Q-f_{\rm{bulk}}(Q)) Q):\nabla\u\phi \,dx\\
&\quad-\int_{\O}(Q(\D Q-f_{\rm{bulk}}(Q))-(\D Q-f_{\rm{bulk}}(Q)) Q):\u\otimes\nabla\phi \,dx.
\end{split}
\end{equation}
Taking a spatial derivative of the equation of $Q$ \eqref{e1}$_1$ yields
\begin{equation*}
\partial_t\partial_{\alpha}Q+\u\cdot\nabla \partial_{\alpha}Q+\partial_{\alpha}\u\cdot\nabla Q+\partial_{\alpha}\big(Q\omega-\omega Q\big)
=\D \partial_{\alpha}Q-\partial_{\alpha}(f_{\rm{bulk}}(Q)).
\end{equation*}
Using again $\Dv\u=0$, we multiply the equation above by $\partial_{\alpha}Q\phi$, integrate 
the resulting equation over $\O$, apply integration by parts, and sum over $\alpha$ to obtain
\begin{equation}\label{2}
\begin{split}
&\f12\f{d}{dt}\int_{\O}|\nabla Q|^2\phi \,dx+\int_{\O}|\D Q|^2\phi \,dx\\
&=\f12\int_{\O}|\nabla Q|^2\partial_t\phi\,dx+\int_{\O} (\u\cdot\nabla) Q\cdot(\D Q\phi+\nabla Q\nabla\phi) \,dx\\
&\quad-\int_{\O}\big(\omega Q-Q \omega\big):(\D Q\phi+\nabla Q\nabla\phi) \,dx\\
&\quad-\int_{\O}\D Q\cdot \nabla Q \nabla\phi\, dx
-\int_{O}\nabla(f_{\rm{bulk}}(Q))\cdot \nabla Q\phi dx.
\end{split}
\end{equation}
By direct calculations, there hold
\begin{eqnarray}\label{cancel1}
&&-\int_{\O}\Delta Q\cdot\nabla Q\nabla\phi\,dx\nonumber\\
&&=\int_{\O}\frac12|\nabla Q|^2\Delta\phi\,dx
+\int_{\O} (\nabla Q\otimes\nabla Q-|\nabla Q|^2 I_3):\nabla^2\phi \,dx,
\end{eqnarray}
and
\begin{equation} \label{cancel2}
\begin{split}
&\int_{\O}\big(\omega Q -Q \omega\big)(\D Q-f_{\rm{bulk}}(Q))\phi \, dx\\
&=\int_{\O}\big(Q(\D Q-f_{\rm{bulk}}(Q)) -Q(\D Q-f_{\rm{bulk}}(Q)))\big):\nabla\u\phi \,dx.
\end{split}
\end{equation}
Hence, by adding \eqref{1} and \eqref{2} together and applying \eqref{cancel1} and \eqref{cancel2}, we have
\begin{equation*}
\begin{split}
&\f12\f{d}{dt}\int_{\O}\big(|\u|^2+|\nabla Q|^2\big)\phi \,dx
+\int_{\O}\big(|\nabla\u|^2+|\D Q|^2\big)\phi \,dx\\
&=\f12\int_{\O}\big(|\u|^2+|\nabla Q|^2\big)(\partial_t+\D)\phi\,dx
+\f12\int_{\O}(|\u|^2+2P)\u\cdot\nabla\phi \,dx\\
&+\int_{\O}(\u\cdot\nabla) Q\cdot \nabla Q\nabla\phi \,dx
-\int_{\O}(Q(\D Q-f_{\rm{bulk}}(Q))-(\D Q-f_{\rm{bulk}}(Q)) Q):\u\otimes\nabla\phi \,dx\\
&-\int_{\O}\big(\omega Q-Q\omega\big):\nabla Q \nabla \phi \,dx
-\int_{\O}\nabla(f_{\rm{bulk}}(Q))\cdot\nabla Q\phi \,dx\\
&+\int_{\O} (\nabla Q\otimes\nabla Q-|\nabla Q|^2 I_3):\nabla^2\phi \,dx.
\end{split}
\end{equation*}
This, after integrating over $[0,t]$, yields the  local energy inequality \eqref{lei1}.
\end{proof}

We close this section by giving a proof of the identity \eqref{free-div2}. More precisely, we have
\begin{Lemma}\label{freed2} For $\Omega=\mathbb R^3$ or $\mathbb T^3$, if $Q^1, Q^2\in H^2(\Omega, \mathcal{S}_0^{(3)})$, then 
\begin{equation}\label{free-div3} 
{\rm{div}}^2(Q^1(\Delta Q^2-f_{\rm{bulk}}(Q^2))-(\Delta Q^2-f_{\rm{bulk}}(Q^2))Q^1)=0\ \ {\rm{in}}\ \ \Omega, 
\end{equation}
in the sense of distributions.
\end{Lemma} 
\begin{proof} For any $\phi\in C_0^\infty(\Omega)$, we see that
\begin{equation*}
\begin{split}
&{\rm{div}}^2(Q^1(\Delta Q^2-f_{\rm{bulk}}(Q^2))-(\Delta Q^2-f_{\rm{bulk}}(Q^2))Q^1)(\phi)\\
&=\int_{\Omega} \big(Q^1_{\alpha\gamma}(\Delta Q^2_{\gamma\beta}-f_{\rm{bulk}}(Q^2)_{\gamma\beta})
-(\Delta Q^2_{\alpha\gamma}-f_{\rm{bulk}}(Q^2)_{\alpha\gamma}) Q^1_{\gamma\beta}\big)\frac{\partial^2 \phi}{\partial x_\alpha\partial x_\beta} \,dx.
\end{split}
\end{equation*}
Set 
$$A_{\alpha\beta}=Q^1_{\alpha\gamma}(\Delta Q^2_{\gamma\beta}-f_{\rm{bulk}}(Q^2)_{\gamma\beta})- (\Delta Q^2_{\alpha\gamma}-f_{\rm{bulk}}(Q^2)_{\alpha\gamma})Q^1_{\gamma\beta},\  \ \forall 1\le\alpha,\beta\le 3,$$
and
$$B_{\alpha\beta}=\frac{\partial^2 \phi}{\partial x_\alpha\partial x_\beta}, \ \forall 1\le\alpha,\beta\le 3.$$
Since $Q^1$ and $Q^2$ are symmetric, it is easy to check that
$$A_{\alpha\beta}=-A_{\beta\alpha}, \ B_{\alpha\beta}=B_{\beta\alpha}, \ \forall 1\le\alpha,\beta\le 3.$$
Hence \eqref{free-div3} follows.
\end{proof}

\section{Global existence of suitable weak solutions}

This section is devoted to the construction of suitable weak solutions to the Beris-Edwards system \eqref{e1}. The idea
is motived by the ``retarded mollification technique" originally due to \cite{S} and \cite{CKN} in the context of
Navier-Stokes equations. Since the procedure for Ball-Majumdar potential $F_{\rm{BM}}(Q)$ is
somewhat different from that for Landau-De Gennes potential $F_{\rm{LdG}}(Q)$, we will describe
them in two separate subsections.

We explain the construction of suitable weak solutions in the spirit of \cite{CKN}. 
For $f:\mathbb R^4\to\mathbb R$ and $0<\theta<1$, define the ``retarded mollifier" $\Psi_{\theta}(f)$ of $f$
by
$$\Psi_{\theta}[f](x,t)=\f{1}{\theta^4}\int_{\R^4}\eta\left(\f{y}{\theta}, \f{\tau}{\theta}\right)\tilde{f}(x-y, t-\tau)\,dyd\tau,$$
where
$$\tilde{f}(x,t)=\begin{cases} f(x,t) & t \ge 0,\\
0 & t<0,
\end{cases}
$$
and the mollifying function $\eta\in C^\infty_0(\mathbb R^4)$ satisfies
$$\begin{cases}\displaystyle\eta\ge0\quad\textrm{and}\quad \int_{\R^4}\eta\, dxdt=1,\\
\supp\ \eta\subset \Big\{(x,t): |x|^2<t,\quad 1<t<2\Big\}.
\end{cases}
$$ 
It follows from Lemma A.8 in \cite{CKN} that for $\theta\in (0,1]$ and $0<T\le\infty$,
\begin{subequations}\label{3}
\begin{align}
\Dv\Psi_{\theta}[\u]&=0\quad\textrm{if}\quad \Dv\u=0,\nonumber\\
\sup_{0\le t\le T}\int_{\R^3}|\Psi_{\theta}[\u]|^2(x,t)\,dx&\le C\sup_{0\le t\le T}\int_{\R^3}|\u|^2(x,t)\,dx\nonumber\\
\int_{\R^3\times [0,T]}|\nabla\Psi_{\theta}[\u]|^2(x,t)\,dxdt&\le C\int_{\R^3\times [0,T]}|\nabla\u|^2(x,t)\,dxdt.\nonumber
\end{align}
\end{subequations}

Now we proceed to find the existence of suitable weak solutions of \eqref{e1} and \eqref{IC} as follows.
\subsection{The Landau-De Gennes potential \fbox{$F_{\rm{bulk}}(Q)=F_{\rm{LdG}}(Q)$ and $\Omega=\mathbb R^3$}}

With the mollifier $\Psi_{\theta}[\u]\in C^\infty(\R^4)$, 
we introduce an approximate version of the Beris-Edwards system \eqref{e1}, namely,
\begin{equation}\label{e2}
\begin{cases}
&\partial_tQ^\theta+\u^\theta\cdot\nabla \Psi_\theta[Q^\theta]-\omega^\theta\Psi_{\theta}[Q^\theta]+\Psi_{\theta}[Q^\theta]\omega^\theta
=\D Q^\theta-f_{\rm{LdG}}(Q^\theta),\\
&\partial_t\u^\theta+\Psi_{\theta}[\u^\theta]\cdot\nabla\u^\theta+\nabla P^\theta\\
&=\D\u^\theta-\nabla (\Psi_\theta[Q^\theta])\cdot\big(\D Q^\theta-f_{\rm{LdG}}(Q^\theta)\big)\\
&\ \ \ +\Dv\Big(\Psi_{\theta}[Q^\theta](\D Q^\theta-f_{\rm{LdG}}(Q^\theta))-(\D Q^\theta - f_{\rm{LdG}}(Q^\theta))\Psi_{\theta}[Q^\theta]\Big),\\
&\Dv\u^\theta=0.
\end{cases} \ {\rm{in}}\ Q_T
\end{equation}
subject to the initial condition \eqref{IC}.  Here $\omega^\theta=\omega(\u^\theta)=\frac{\nabla \u^\theta-(\nabla\u^\theta)^\top}{2}$.

The idea behind the construction of suitable weak solutions to \eqref{e2} is as follows. For a fixed large $N\ge 1$, 
set $\theta=\frac{T}{N}\in (0, 1]$,
we want to find $\u=\u^\theta, P=P^\theta$, and $Q=Q^\theta$ solving \eqref{e2} and \eqref{IC}. 
Since $\Psi_\theta[\u]$ and $\Psi_\theta[Q]$ are smooth, and their values at time $t$ depend only
on the values of $\u$ and $Q$ at times prior to $t-\theta$, solving \eqref{e2}  and \eqref{IC}
involves iteratively solving \eqref{e2} in the interval $[m\theta, (m+1)\theta]$,
subject to the initial condition
$$(\u, Q)\big|_{t=m\theta}=(u^\theta, Q^\theta) (\cdot, m\theta)\ {\rm{in}}\  \mathbb R^3,$$
for $0\le m\le N-1$. This amounts to solving a system that 
couples a semi-linear parabolic-like
equation for $Q$ and a Stokes-like equation for $\u$, in which all the coefficient functions are given 
smooth functions. 

We can verify, by the classical Faedo-Garlekin method,  the existence of $(\u^\theta, Q^\theta, P^\theta)$ inductively on each time interval $(m\theta, (m+1)\theta)$ for all $0\le m\le N-1$.
Indeed for $m=0$, according to the definition of $\Psi_\theta$, $\Psi_\theta(\u^\theta)=\Psi_\theta(Q^\theta)=0$, 
and the system \eqref{e2} reduces to a linear system
\begin{equation}\label{e0}
\begin{cases}
\partial_tQ^\theta=\D Q^\theta-f_{\rm{LdG}}(Q^\theta)\\
\partial_t\u^\theta+\nabla P^\theta=\D\u^\theta\\
\Dv\u^\theta=0\\
(\u^\theta, Q^\theta)|_{t=0}=(\u_0, Q_0)
\end{cases}
\end{equation}
in $\R^3\times [0,\theta]$. For the system \eqref{e0}, $Q^\theta$ and $\u^\theta$ are decouple,
and  $\u^\theta$ can be found according to the standard theory of Stokes equations, 
while the equation of $Q^\theta$ is a semi-linear parabolic equation which can be solved by the standard method
for parabolic equations. 

Suppose now that the system \eqref{e2} has been solved for some $0\le k< N-1$. We are going to solve the system \eqref{e2}
\begin{equation}\label{e3.0}
\begin{cases}
&\partial_tQ_{\alpha\beta}+\u\cdot\nabla \tilde{Q}_{\alpha\beta}-\omega_{\alpha\gamma}\tilde{Q}_{\gamma\beta}+\tilde{Q}_{\alpha\gamma}\omega_{\gamma\beta}=\D Q_{\alpha\beta}-f_{LdG}(Q)_{\alpha\beta}\\
&\partial_t\u_{\alpha}+\tilde{\u}\cdot\nabla\u_{\alpha}+\partial_{\alpha} P=\D\u_{\alpha}-\partial_\alpha \tilde{Q}_{\beta\gamma}(\D Q-f_{LdG}(Q))_{\beta\gamma}\\
&\quad+\partial_\beta\Big(\tilde{Q}_{\alpha\gamma}(\D Q-f_{LdG}(Q))_{\gamma\beta}-(\D Q-f_{LdG}(Q))_{\alpha\gamma}\tilde{Q}_{\gamma\beta}\Big)\\
&\Dv\u=0.
\end{cases}
\end{equation}
in the time interval $[k\theta, (k+1)\theta]$ with the initial data 
\begin{equation}\label{IC1}
(\u, Q)|_{t=k\theta}=(\u^\theta, Q^\theta)(\cdot,k\theta) \quad\textrm{in}\quad \R^3,
\end{equation}
and
$$\tilde{Q}=\Psi_\theta[Q^\theta]\quad\textrm{and}\quad \tilde{\u}=\Psi_{\theta}[\u^\theta].$$ 
Note that $\tilde{\u}$ and $\tilde{Q}$ are smooth functions in $[k\theta, (k+1)\theta]\times \R^3$.

The existence of $(\u, Q)$ in \eqref{e3.0} may be solved by using the Faedo-Galerkin method. 
Indeed for a pair of smooth test functions $(\psi, \phi)\in H^2(\R^3, \mathcal{S}_0^{(3)})
\times {\bf V}$, the system \eqref{e3.0} turns to be
\begin{equation}\label{e4a}
\begin{split}
&\f{d}{dt}\int_{\R^3}(\nabla Q, \nabla\psi)\,dx-\int_{\R^3}(\u\cdot\nabla\tilde{Q}, \D\psi)\,dx-\int_{\R^3}(-\omega_{\alpha\gamma}\tilde{Q}_{\gamma\beta}+\tilde{Q}_{\alpha\gamma}\omega_{\gamma\beta}, \D\psi_{\alpha\beta})\,dx\\
&\quad=-\int_{\R^3}(\D Q_{\alpha\beta}-f_{LdG}(Q)_{\alpha\beta}, \D\psi_{\alpha\beta})\,dx,
\end{split}
\end{equation}
and
\begin{equation}\label{e4b}
\begin{split}
&\f{d}{dt}\int_{\R^3}(\u, \phi)\,dx+\int_{\R^3}(\tilde{\u}\cdot\nabla\u, \phi)\,dx+\int_{\R^3}(\nabla\u,\nabla\phi)\,dx\\
&\quad=-\int_{\R^3}\Big(\partial_\alpha \tilde{Q}_{\beta\gamma}(\D Q-f_{LdG}(Q))_{\beta\gamma},\phi_\alpha\Big)\,dx\\
&\qquad-\int_{\R^3}\left(\Big(\tilde{Q}_{\alpha\gamma}(\D Q-f_{LdG}(Q))_{\gamma\beta}-(\D Q-f_{LdG}(Q))_{\alpha\gamma}\tilde{Q}_{\gamma\beta}\Big), \partial_\beta\phi_\alpha\right)\,dx,
\end{split}
\end{equation}
in the sense of distributions.
The system of first order ODE equations \eqref{e4a}-\eqref{e4b} can be solved when the test function $(\psi, \phi)$ are taken to be the basis of 
$H^2(\R^3,\mathcal{S}_0^{(3)})\times {\bf V}$ up to a short time interval $[k\theta, k\theta+T_0]$. 
Performing the energy estimate for \eqref{e3.0} as for the original system, we get that for $k\theta\le t\le T_0$,
\begin{equation*}
\begin{split}
&\sup_{t\ge k\theta}\int_{\R^3}\left(|\u^\theta|^2+|\nabla Q^\theta|^2+F_{\rm{LdG}}(Q^\theta)\right)\,dx
+\int_{k\theta}^t\int_{\R^3}\Big(|\nabla\u^\theta|^2+|\D Q-f_{\rm{LdG}}(Q^\theta)|^2\Big)\,dxds\\
&\quad\le \int_{\R^3}\left(|\u^\theta|^2+|\nabla Q^\theta|^2+F_{\rm{LdG}}(Q^\theta)\right)(x, k\theta)\,dx.
\end{split}
\end{equation*}
Hence $T_0$ can be extended up to $\theta$.

Let $(\u^\theta, P^\theta, Q^\theta)$ be the global weak solution
of  \eqref{e2} and \eqref{IC}  in $Q_T$. Then
$$\u^\theta\in L^\infty_tL^2_x\cap L^2_tH^1_x(Q_T), \ Q^\theta\in L^\infty_tH^1_x\cap L^2_tH^2_x(Q_T),
\ P^\theta\in L^2(Q_T).$$
Observe that
\begin{equation*}
\begin{split}
&\big(\omega^\theta\Psi_{\theta}[Q^\theta]-\Psi_{\theta}[Q^\theta]\omega^\theta\big): (\D Q^\theta-f_{\rm{LdG}}(Q^\theta))\\
&=-\big(\Psi_{\theta}[Q^\theta](\D Q^\theta-f_{\rm{LdG}}(Q^\theta))-(\D Q^\theta-f_{\rm{LdG}}(Q^\theta))\Psi_{\theta}[Q^\theta]\big):\nabla\u^\theta.
\end{split}
\end{equation*}
Hence, by calculations similar to Lemma \ref{L0}, we deduce that $(\u^\theta, Q^\theta)$  satisfies
the global energy inequality: for $0\le t\le T$, 
\begin{eqnarray}\label{4}
&&E(\u^\theta,Q^\theta)(t)
+\int_{\R^3\times [0,t]}\big(|\nabla\u^\theta|^2+|\D Q^\theta-f_{\rm{LdG}}(Q^\theta)|^2\big)\,dxdt\nonumber\\
&&\le E(\u^\theta,Q^\theta)(0)=\int_{\R^3}(\frac12|\u_0|^2+\frac12|\nabla Q_0|^2+F_{\rm{LdG}}(Q_0))(x,t)\,dx.
\end{eqnarray} 
Direct calculations show that
\begin{equation*}
\begin{split}
&\int_{\R^3}\Delta Q^\theta\cdot f_{\rm{LdG}}(Q^\theta)\,dx\\
&=-a\int_{\R^3}|\nabla Q^\theta|^2\,dx-c\int_{\R^3}\big(|\nabla Q^\theta|^2|Q^\theta|^2+\f12|\nabla\tr((Q^\theta)^2)|^2\big)\,dx\\
&\quad+b\int_{\R^3}\nabla\big((Q^\theta)^2-\f{\tr((Q^\theta)^2)}{3}I_3\big)\cdot\nabla Q^\theta\,dx\\
&\le -\f{c}{4}\int_{\R^3}\big(|\nabla Q^\theta|^2|Q^\theta|^2+\f12|\nabla\tr((Q^\theta)^2)|^2\big)\,dx+C(a,b,c)\int_{\R^3}|\nabla Q^\theta|^2\,dx.
\end{split}
\end{equation*}
This, combined with the assumption $c>0$ and estimate \eqref{4}, gives
\begin{equation}\label{4a}
\begin{split}
&\f{d}{dt}\int_{\R^3}(|\u^\theta|^2+|\nabla Q^\theta|^2+F_{\rm{LdG}}(Q^\theta))(x,t)\,dx+2\int_{\R^3}\big(|\nabla\u^\theta|^2+|\D Q^\theta|^2\big)\,dx\\
&\ +c\int_{\R^3}\big(|\nabla Q^\theta|^2|Q^\theta|^2+\f12|\nabla\tr((Q^\theta)^2)|^2\big)\,dx\\
&\le C(a,b,c)\int_{\R^3}|\nabla Q^\theta|^2\,dx.
\end{split}
\end{equation} 
Therefore we deduce from \eqref{4a} and Gronwall's inequality that
\begin{equation}\label{5}
\begin{split}
&\sup_{0\le t\le T}\int_{\R^3}(|\u^\theta|^2+|\nabla Q^\theta|^2+F_{\rm{LdG}}(Q^\theta))(x,t)\,dx\\
&\qquad+\int_{\R^3\times [0,T]}\big(|\nabla\u^\theta|^2+|\D Q^\theta|^2\big)\,dxdt\\
&\le C(a,b,c,T) \big(\|{\u}_0\|_{L^2(\R^3)}^2+\|Q_0\|_{H^1(\R^3)}^2\big).
\end{split}
\end{equation}
From \eqref{de gennes1}, we know that there exists a $M_0>0$, depending on $a,b,c$, such that
$$F_{\rm{LdG}}(Q)\ge \frac{c}2 |Q|^4, \ \forall Q\in \mathcal{S}_0^{(3)} \ {\rm{with}}\ |Q|\ge M_0.$$
This, combined with \eqref{5} and $F_{\rm{LdG}}(Q)\ge 0$, implies that
\begin{equation}\label{5.1}
\begin{split}
&\sup_{0\le t\le T}\int_{\{x\in\R^3: \ |Q^\theta(x,t)|\ge M_0\}}  |Q^\theta(x,t)|^4\,dx\\
&\le \frac{2}{c}\sup_{0\le t\le T}\int_{\R^3} F_{\rm{LdG}}(Q^\theta)(x,t)\,dx\\
&\le C(a,b,c,T) \big(\|{\u}_0\|_{L^2(\R^3)}^2+\|Q_0\|_{H^1(\R^3)}^2\big).
\end{split}
\end{equation}
From \eqref{5.1},  we can conclude that for any compact set $K\subset\R^3$,
\begin{equation}\label{5.2}
\begin{split}
&\sup_{0\le t\le T}\int_{K}  |Q^\theta(x,t)|^4\,dx\\
&\le \sup_{0\le t\le T}\Big\{\int_{\{x\in K:\ |Q^\theta(x,t)|\le M_0\}}  |Q^\theta(x,t)|^4\,dx
+\int_{\{x\in K:\ |Q^\theta(x,t)|> M_0\}}  |Q^\theta(x,t)|^4\,dx\Big\}\\
&\le |K|M_0^4+C(a,b,c,T) \big(\|{\u}_0\|_{L^2(\R^3)}^2+\|Q_0\|_{H^1(\R^3)}^2\big).
\end{split}
\end{equation} 
From \eqref{5} and \eqref{5.2}, we have that $\u^\theta$ is uniformly bounded in $L^2_tH^1_x(\mathbb R^3\times [0,T])$,
$Q^\theta$ is uniformly bounded in $L^2_tH^2_x(K\times [0,T])$ for any compact set $K\subset\R^3$, and
$\nabla Q^\theta$ is uniformly bounded in $L^2_tH^1_x(\R^3\times [0,T])$.
Therefore, after passing to a subsequence, we may assume that as $\theta\to 0$ (or equivalently $N\to \infty$), there exist
$\u\in L^\infty_tL^2_x\cap L^2_tH^1_x(\R^3\times [0,T])$, $Q\in \displaystyle\cap_{R>0} L^\infty_t L^4_x(B_R\times [0,T])$,
with $\nabla Q\in L^\infty_tL^2_x\cap L^2_tH^1_x(\R^3\times [0,T])$, such that 
\begin{equation}\label{6}
\begin{cases}
Q^\theta\rightharpoonup Q & {\rm{in}}\ \ \ L^2 ([0,T], L^2(\R^3)),\\
\nabla Q^\theta\rightharpoonup \nabla Q & \textrm{in}\quad L^2([0,T], H^1(\R^3)),\\
\u^\theta\rightharpoonup \u & \textrm{in}\quad L^2([0,T], H^1(\R^3)).
\end{cases}
\end{equation}
Hence by the lower semicontinuity and \eqref{4} we have that
\begin{eqnarray}\label{6.1}
&&E(\u,Q)(t)
+\int_{\R^3\times [0,t]}\big(|\nabla\u|^2+|\D Q-f_{\rm{LdG}}(Q)|^2\big)\,dxdt\nonumber\\
&&\le E(\u,Q)(0)=\int_{\R^3}(\frac12|\u_0|^2+\frac12|\nabla Q_0|^2+F_{\rm{LdG}}(Q_0))(x,t)\,dx
\end{eqnarray} 
holds for $0\le t\le T$.

Now we want to estimate the pressure function $P^\theta$. 
Taking divergence of \eqref{e2}$_2$ gives
\begin{equation}\label{pressure}
\begin{split}
-\D P^\theta&=\Dv^2(\Psi_{\theta}[\u^\theta]\otimes\u^\theta)+\Dv\big(\nabla(\Psi_{\theta}[Q^\theta])\cdot(\D Q^\theta-f_{\rm{LdG}}(Q^\theta))\big)   \\
&\quad-\Dv^2\big(\Psi_{\theta}[Q^\theta](\D Q^\theta -f_{\rm{LdG}}(Q^\theta) )-(\D Q^\theta-f_{\rm{LdG}}(Q^\theta)) \Psi_{\theta}[Q^\theta]\big)\\
&=\Dv^2(\Psi_{\theta}[\u^\theta]\otimes\u^\theta)+\Dv\big(\nabla(\Psi_{\theta}[Q^\theta])\cdot(\D Q^\theta-f_{\rm{LdG}}(Q^\theta))\big)  \ \ {\rm{in}}\ \R^3.
\end{split}
\end{equation}
Here we have used in the last step the fact that
$$\Dv^2\big(\Psi_{\theta}[Q^\theta](\D Q^\theta-f_{\rm{LdG}}(Q^\theta))-(\D Q^\theta-f_{\rm{LdG}}(Q^\theta)) \Psi_{\theta}[Q^\theta]\big)=0\ \ {\rm{in}}\ \ \R^3,$$
which follows from \eqref{free-div2}.

For $P^\theta$, we claim that $P^\theta$ in $L^{\f53}(\R^3\times [0,T])$ and
\begin{equation}\label{pestimate}
\big\|P^\theta\big\|_{L^{\f53}(\R^3\times [0,T])}\le C\big(a,b,c,T, \|{\u}_0\|_{L^2(\R^3)}, \|Q_0\|_{H^1(\R^3)}\big),
\ \forall \theta\in (0,1].
\end{equation}
To see this, first observe that \eqref{5}  implies 
$\nabla(\Psi_\theta[Q^\theta])\in L^\infty_tL^2_x\cap L^2_tH^1_x(\R^3\times [0,T])$.
Hence by the Sobolev interpolation inequality we have that 
\begin{equation*}
\begin{split}
\big\|\nabla(\Psi_\theta[Q^\theta])\big\|_{L^{10}_tL^{\frac{30}{13}}_x(\R^3\times [0,T])}
&\leq C\big\|\nabla(\Psi_\theta[Q^\theta])\big\|_{L^\infty_tL^2_x\cap L^2_tH^1_x(\R^3\times [0,T])}\\
&\leq C\big(a,b,c,T, \|{\u}_0\|_{L^2(\R^3)}, \|Q_0\|_{H^1(\R^3)}\big).
\end{split}
\end{equation*}
By H\"older's inequality we then have that
\begin{equation}\label{6.3}
\begin{split}
&\big\|\nabla(\Psi_\theta[Q^\theta])\cdot(\D Q^\theta-f_{\rm{LdG}}(Q^\theta))\big\|_{L^{\frac53}_tL^{\frac{15}{14}}_x(\R^3\times [0,T])}\\
&\le \big\|\nabla(\Psi_\theta[Q^\theta])\big\|_{L^{10}_tL^{\frac{30}{13}}_x(\R^3\times [0,T])}\big\|\D Q^\theta-f_{\rm{LdG}}(Q^\theta)\big\|_{L^2(\R^3\times [0,T])}\\
&\leq C\big(a,b,c,T, \|{\u}_0\|_{L^2(\R^3)}, \|Q_0\|_{H^1(\R^3)}\big).
\end{split}
\end{equation}
By Calderon-Zygmund's $L^p$-estimate \cite{Stein} \cite{Evans}, we conclude that
$P^\theta \in L^{\f53}([0,T]\times\R^3)$, and 
\begin{equation*}
\begin{split}
&\big\|P^\theta\big\|_{L^{\f53}([0,T]\times\R^3)}\\
&\le C \Big[\big\|\Psi_{\theta}[\u^\theta]\otimes\u^\theta\big\|_{L^\frac53(\R^3\times [0,T])}
+\big\|\nabla(\Psi_\theta[Q^\theta])\cdot(\D Q^\theta-f_{\rm{LdG}}(Q^\theta))\big\|_{L^{\frac53}_tL^{\frac{15}{14}}_x(\R^3\times [0,T])}\Big]\\
&\le C \Big[\big\|\u^\theta\big\|_{L^\frac13(\R^3\times [0,T])}^2
+\big\|\nabla(\Psi_\theta[Q^\theta])\cdot (\D Q^\theta-f_{\rm{LdG}}(Q^\theta))\big\|_{L^{\frac53}_tL^{\frac{15}{14}}_x(\R^3\times [0,T])}\Big]\\
&\le C\big(a,b,c,T, \|{\u}_0\|_{L^2(\R^3)}, \|Q_0\|_{H^1(\R^3)}\big).
\end{split}
\end{equation*}
It follows from \eqref{pestimate} that we may assume that  
there exists  $P\in L^\frac53(\R^3\times [0,T])$ such that as $\theta\to 0$, 
\begin{equation}\label{pconv1}
P^\theta\rightharpoonup P\ \ {\rm{in}}\ \ L^\frac53(\R^3\times [0,T]). 
\end{equation}

From \eqref{e2}$_2$ and the bounds \eqref{5} and \eqref{5.1}, we have that
\begin{equation*}
\begin{split}
&\partial_t \u^\theta=-\Psi_\theta[\u^\theta]\cdot\nabla\u^\theta-\nabla P^\theta+\Delta \u^\theta-\nabla(\Psi_\theta[Q^\theta])
\cdot(\D Q^\theta-f_{\rm{LdG}}(Q^\theta))\\
&\ \ \ \ \ \ \ \ \ -\Dv(\Psi_\theta[Q^\theta](\Delta Q^\theta-f_{\rm{LdG}}(Q^\theta))-(\Delta Q^\theta-f_{\rm{LdG}}(Q^\theta))\Psi_\theta[Q^\theta])\\
&\in L^\frac54(\R^3\times [0,T])+L^\frac53([0,T], W^{-1,\frac53}(\R^3))+\bigcap_{R>0}L^2([0,T], W^{-1, \frac43}(B_R)),
\end{split}
\end{equation*}
and for any $0<R<\infty$,
\begin{equation}\label{ut-estimate1}
\begin{split}
&\Big\|\partial_t \u^\theta\Big\|_{L^\frac54(\R^3\times [0,T])+L^\frac53([0,T], W^{-1,\frac53}(\R^3))+L^2([0,T], W^{-1, \frac43}(B_R))}\\
&\leq C\big(a, b, c, R, T, \|\u_0\|_{L^2(\R^3)}, \|Q_0\|_{H^1(\R^3)}\big), \ \forall \theta\in (0,1].
\end{split}
\end{equation}
Similarly, it follows from \eqref{e2}$_1$ and the bounds \eqref{5} and \eqref{5.1} that 
$\partial_t Q^\theta\in L^\frac53(\R^3\times [0,T])+\bigcap_{R>0}L^2([0,T], L^\frac43(B_R))$, and
\begin{equation}\label{dt-estimate1}
\Big\|\partial_t Q^\theta\Big\|_{L^\frac53(\R^3\times [0,T])+L^2([0,T], L^\frac43(B_R))}
\le C\big(a, b, c, R, T, \|\u_0\|_{L^2(\R^3)}, \|Q_0\|_{H^1(\R^3)}\big), 
\end{equation} 
for all $0<R<\infty$ and $\theta\in (0,1]$.

By \eqref{5}, \eqref{5.1}, \eqref{ut-estimate1}, and \eqref{dt-estimate1}, we can apply Aubin-Lions' compactness Lemma 
(\cite{Temam}) to conclude that
for any $0<R<\infty$,
\begin{equation}\label{l2-conv1}
\big(\u^\theta, Q^\theta, \nabla Q^\theta\big)\to \big(\u, Q, \nabla Q\big)\ \ \ {\rm{in}}\ \ \  L^3(B_R\times [0,T]),  \ \ {\rm{as}}\ \ \theta\to 0.
\end{equation}
On the other hand, it follows from $F_{\rm{LdG}}(Q^\theta )\ge 0$ in $\R^3\times [0,T]$  and  \eqref{5} that 
$$\sup_{0\le t\le T} \int_{\R^3} |\nabla Q^\theta|^2(x,t)\,dx\le C\big(a,b,c,T, \|{\u}_0\|_{L^2(\R^3)}, \|Q_0\|_{H^1(\R^3)}\big).$$
Hence by \eqref{l2-conv1} we also have that for any $1<p_1<6$ and $1<q_1<\frac{10}{3}$, 
\begin{equation}\label{l2-conv1.0}
Q^\theta\to Q \  {\rm{in}}\  L^{p_1} (B_R\times [0,T]); \ \u^\theta\to \u \  \  {\rm{in}}\  L^{p_2} (B_R\times [0,T]) \ \ {\rm{as}}\ \ \theta\to 0.
\end{equation} 

With the convergences \eqref{6}, \eqref{pconv1}, and \eqref{l2-conv1}, it is not hard to show that 
the limit $(\u, P, Q)$ is a weak solution of \eqref{e1} and \eqref{IC}, i.e., it satisfies the system \eqref{e1} and \eqref{IC}
in the sense of distributions (see also \cite{PZ} Proposition 3). We leave the details to interested readers, besides 
pointing out that in the sense of distributions, as $\theta\to 0$, 
$$\nabla P^\theta-\nabla(\Psi_\theta[Q^\theta])\cdot  f_{\rm{LdG}}(Q^\theta)\to\nabla P- \nabla Q\cdot f_{\rm{LdG}}(Q)=\nabla(P-F_{\rm{LdG}}(Q)).$$

To show that $(\u, P, Q)$ is a suitable weak solution of \eqref{e1}, observe that,  as in Lemma \ref{L1}, we can test equations of $\u^\theta$ in \eqref{e2} by $\u^\theta\phi$,  and taking a spatial derivative of the equation of $Q^\theta$  in \eqref{e2} and then testing it by $\nabla Q^\theta\phi$ for any nonnegative
$\phi\in C_0^\infty(\R^3\times (0,t])$,  to obtain the following local energy inequality
\begin{eqnarray}\label{7}
&&\int_{\R^3}\big(|\u^\theta|^2+|\nabla Q^\theta|^2\big)\phi(x,t)\,dx+2\int_0^t\int_{\R^3}\big(|\nabla\u^\theta|^2+|\D Q^\theta|^2\big)\phi \,dxds\nonumber\\
&&=\int_0^t\int_{\R^3}\Big[\big(|\u^\theta|^2+|\nabla Q^\theta|^2\big)(\partial_t\phi+\D\phi)
+2 \nabla \Psi_\theta[Q^\theta]\otimes\nabla Q^\theta:
\u^\theta\otimes \nabla\phi\Big]\,dxds\nonumber\\
&&+\int_0^t\int_{\R^3}(|\u^\theta|^2\Psi_\theta[\u^\theta]\cdot\nabla\phi+2P^\theta\u^\theta\cdot\nabla\phi+2\nabla(\Psi_\theta[Q^\theta])\cdot
f_{\rm{LdG}}(Q^\theta) \u^\theta\phi) \,dxds\nonumber\\
&&+2\int_0^t\int_{\R^3}\big(\Psi_\theta[Q^\theta] f_{\rm{LdG}}(Q^\theta)-f_{\rm{LdG}}(Q^\theta)\Psi_\theta[Q^\theta]\big): \nabla\u^\theta\phi \,dxds\nonumber\\
&&+2\int_0^t\int_{\R^3}\big(\nabla Q^\theta\otimes \nabla Q^\theta-|\nabla Q^\theta|^2 I_3\big)):\nabla^2\phi \,dxds\nonumber\\
&&-2\int_0^t\int_{\R^3}(\Psi_\theta[Q^\theta](\D Q^\theta-f_{\rm{LdG}}(Q^\theta))-(\D Q^\theta-f_{\rm{LdG}}(Q^\theta))\Psi_\theta[Q^\theta]):\u^\theta\otimes\nabla\phi \,dxds\nonumber\\
&&-2\int_0^t\int_{\R^3}\Big(\omega^\theta\Psi_\theta[Q^\theta]-\Psi_\theta[Q^\theta]\omega^\theta\Big):\nabla Q^\theta\nabla\phi \,dxds\nonumber\\
&&-2\int_0^t\int_{\R^3}\nabla(f_{\rm{LdG}}(Q^\theta))\cdot \nabla Q^\theta\phi \,dxds.
\end{eqnarray}
Taking the limit in \eqref{7} as $\theta\rightarrow 0$, we see by the lower semicontinuity that it holds
\begin{equation*}
\begin{split}
&\int_{\R^3}\big(|\u|^2+|\nabla Q|^2\big)\phi(x,t)\,dx+2\int_0^t\int_{\R^3}\big(|\nabla\u|^2+|\D Q|^2\big)\phi \,dxds\\
&\le\liminf_{\theta\to 0}\Big[\int_{\R^3}\big(|\u^\theta|^2+|\nabla Q^\theta|^2\big)\phi(x,t)\,dx\\
&\qquad\qquad\ +2\int_0^t\int_{\R^3}\big(|\nabla\u^\theta|^2+|\D Q^\theta|^2\big)\phi \,dxds\Big].
\end{split}
\end{equation*}
While it follows from \eqref{l2-conv1} and \eqref{l2-conv1.0} that 
\begin{equation*}
\begin{split}
&\lim_{\theta\to0}\ {\rm{RHS\ of }}\ \eqref{7}\\
&=\int_0^t\int_{\R^3}\Big(|\u|^2+|\nabla Q|^2\Big)(\partial_t\phi+\D\phi)\,dxdt\\
&+\int_0^t\int_{\R^3}(|\u|^2+|\nabla Q|^2+2(P-F_{\rm{LdG}}(Q)))\u\cdot\nabla\phi)+2\nabla Q\otimes\nabla Q:\u\otimes\nabla\phi\,dxds\\
&+2\int_0^t\int_{\R^3}\big[\nabla Q\otimes\nabla Q-|\nabla Q|^2 I_3\big]:\nabla^2\phi\,dxds\\
&-2\int_0^t\int_{\R^3} (Q(\D Q-f_{\rm{LdG}}(Q))-(\D Q-f_{\rm{LdG}}(Q)) Q): \u\otimes\nabla\phi \,dxds\\
&-2\int_0^t\int_{\R^3}\big(\omega Q-Q\omega\big):\nabla Q\nabla\phi \,dxds
-2\int_0^t\int_{\R^3}\nabla(f_{\rm{LdG}}(Q))\cdot \nabla Q\phi \,dxds.
\end{split}
\end{equation*}
Here we have used the following convergence result
\begin{equation}\label{new-p1}
\begin{split}
\int_0^t\int_{\R^3} \nabla(\Psi_\theta[Q^\theta])\cdot
f_{\rm{LdG}}(Q^\theta) \u^\theta\phi \,dxds
&\to \int_0^t\int_{\R^3} \nabla Q\cdot
f_{\rm{LdG}}(Q) \u\phi \,dxds\\
&=\int_0^t\int_{\R^3} \nabla (F_{\rm{LdG}}(Q)) \u\phi \,dxds\\
&=-\int_0^t\int_{\R^3} F_{\rm{LdG}}(Q) \u\nabla\phi \,dxds.
\end{split}
\end{equation}
Putting these together yields the desired local energy inequality \eqref{lei} for $(\u, P, Q)$. This completes the proof
of the existence of suitable weak solution in the first case. \qed

\medskip
In the next subsection, we will indicate how to construct a suitable weak solution of \eqref{e2} for the Ball-Majumdar potential function.
\subsection{The Ball-Majumdar potential \fbox{$F_{\rm{bulk}}(Q)=F_{\rm{BM}}(Q)$ and $\Omega=\mathbb T^3$}}  Since
$G_{\rm{BM}}$, given by \eqref{BMP}, is singular outside the physical domain
$$\mathcal{D}=\Big\{Q\in \mathcal{S}_0^{(3)}: \ -\frac13<\lambda_i(Q)<\frac23, \ i=1,2,3\Big\},$$
we need to regularize it. For this part, we follow the scheme by Wilkinson \cite{Wilkinson} (Section 3)  very closely. First
we regularize it by using the Yosida-Moreau regularization of convex analysis \cite{Y} \cite{ET}: For $m\in\mathbb N^+$, define
  $$\widetilde{G}^m_{\rm{BM}}(Q):=\inf_{A\in \mathcal{S}_0^{(3)}}\Big\{ m|A-Q|^2+G_{\rm{BM}}(A) \Big\},\ \forall Q\in\mathcal{S}_0^{(3)}.$$
Then smoothly mollify $\widetilde{G}_{\rm{BM}}^m$ through the standard mollifications:
  $${G}_{\rm{BM}}^m(Q):=\int_{\mathcal{S}_0^{(3)}}\widetilde{G}_{\rm{BM}}^m(Q-R)\Phi_m(R)\,dR,$$
  where $\Phi_m(R)={m^5}\Phi\left(m R \right)$, and 
  $\Phi\in C_0^\infty(\mathcal{S}_0^{(3)})$ is nonnegative and satisfies
  $${\rm{supp}}\ \Phi\subset \Big\{Q\in\mathcal{S}_0^{(3)}: \ |Q|<1\Big\},  
\ \  \int_{\mathcal{S}_0^{(3)}}\Phi(R)\,dR=1.$$
As in \cite{Wilkinson} Proposition 3.1, $G_{\rm{BM}}^m$ satisfies the following properties:{\it
\begin{enumerate}
\item[(G0)] $G_{\rm{BM}}^m$ is an isotropic function of $Q$.
\item [(G1)] $G_{\rm{BM}}^m\in C^\infty(\mathcal{S}_0^{(3)})$ is convex on $\mathcal{S}_0^{(3)}$.
\item [(G2)] There exists a constant $g_0>0$, independent of $m$, such that for any $m\in \mathbb N^+$,
$G_{\rm{BM}}^m(Q)\ge -g_0$ holds for all $Q\in \mathcal{S}_0^{(3)}$.
\item [(G3)] $G_{\rm{BM}}^m(Q)\le G_{\rm{BM}}^{m+1}(Q)\le G_{\rm{BM}}(Q)$  on $\mathcal{S}_0^{(3)}$ for all $m\ge 1$.
\item [(G4)] $G_{\rm{BM}}^m\to G_{\rm{BM}}$ and $\nabla_QG_{\rm{BM}}^m\to \nabla_Q G_{\rm{BM}}$ in $L^\infty_{\rm{loc}}(\mathcal{D})$, as $m\to\infty$.
\item [(G5)] There exist  $\alpha(m), \beta(m),\gamma(m)>0$ such that
$$\alpha(m)|Q|-\beta(m)\le \big|\langle\nabla_Q G_{\rm{BM}}^m(Q)\rangle\big|\le \gamma(m)(1+|Q|), \ \forall Q\in \mathcal{S}_0^{(3)}.$$
\item [(G6)] For $k\ge 2$, there exists $C(m,k)>0$ such that
$$\big|\langle\nabla_Q^k G_{\rm{BM}}^m(Q)\rangle\big|\le C(m,k)(1+|Q|^2), \ \forall Q\in \mathcal{S}_0^{(3)}.$$
\end{enumerate}}
For our purpose in this paper, we also need the following estimate on $G_{\rm{BM}}^m$.
\begin{Lemma}\label{lowerG1} For any $m\in \mathbb N^+$, $G_{\rm{BM}}^m$ satisfies
\begin{equation}\label{lowerG2}
G_{\rm{BM}}^m(Q) \ge \frac{m}4 |Q|^2-g_0, \ \forall Q\in \mathcal{S}_0^{(3)} \ {\rm{with}}\ |Q|\ge 11,
\end{equation}
where $g_0>0$ is the same constant given by ${\rm{(G2)}}$.  
\end{Lemma} 
\begin{proof} Since $G_{\rm{BM}}(Q)=\infty$ for $Q\not\in \mathcal{D}$, it follows from the definition of  $\widetilde{G}^m_{\rm{BM}}$ and (G2) that 
\begin{equation*}
\begin{split}
\widetilde{G}^m_{\rm{BM}}(Q)&= \inf_{A\in \mathcal{D}}\Big\{ m|A-Q|^2+G_{\rm{BM}}(A) \Big\}\\
&\ge \inf_{A\in \mathcal{D}}\Big\{ m|A-Q|^2\Big\}-g_0\\
&=m {\rm{dist}}^2\big(Q, \overline{\mathcal{D}}\big)-g_0.
\end{split}
\end{equation*}
Thus for any $Q\in \mathcal{S}_0^{(3)}$ with $|Q|\ge 10$, we have
$$
\widetilde{G}_{\rm{BM}}^m(Q)\ge m (|Q|-\frac{2}{\sqrt{3}})^2-g_0\ge m \big(\frac{|Q|}{\sqrt{2}}\big)^2-g_0=\frac{m}2 |Q|^2-g_0.
$$
It is not hard to see that this estimate, along with the definition of ${G}_{\rm{BM}}^m$, yields \eqref{lowerG2}.  The proof is now complete.
\end{proof} 

Now we set 
$$F_{\rm{BM}}^m (Q)= \nu G_{\rm{BM}}^m (Q)-\frac{\kappa}2 |Q|^2, \ \forall Q\in \mathcal{S}_0^{(3)}, $$
and 
$$ f_{\rm{BM}}^m (Q)= \nu \big\langle\nabla_QG_{\rm{BM}}^m (Q)\big\rangle-{\kappa} Q,
\ \forall Q\in \mathcal{S}_0^{(3)}.$$
Observe that the convexity of $G^m_{\rm{BM}}$ on $\mathcal{S}_0^{(3)}$ yields that
\begin{equation}\label{convex}
{\rm{tr}}\nabla_Qf_{\rm{BM}}^m (Q)(\nabla Q,\nabla Q)={\rm{tr}}\nabla^2_QF_{\rm{BM}}^m (Q)(\nabla Q,\nabla Q) \ge -\kappa |\nabla Q|^2, 
\end{equation}
for all $Q\in H^1(\Omega, \mathcal{S}_0^{(3)}$.

Note that if we view a function on $\mathbb T^3$ as a $\mathbb Z^3$- periodic function on $\R^3$, then the ``retarded" mollification procedure
given in the previous subsection can be directly performed on functions defined in $\mathbb T^3$.

Similar to the subsection 3.1, we can introduce an approximate system of \eqref{e2} for the Ball-Majumdar potential as follows.
For $T>0$ and a fixed large $N\in\mathbb N^+$, let $\theta=\frac{T}{N}\in (0,1]$. Then we seek $(\u^{\theta,m}, P^{\theta,m}, Q^{\theta,m})$ that solves
\begin{equation}\label{e3}
  \left\{
    \begin{array}{l}
     \partial_t Q^{\theta,m}+\u^{\theta, m}\cdot \nabla\Psi_\theta[Q^{\theta,m}]-\omega^{\theta,m}\Psi_\theta[Q^{\theta,m}]
     +\Psi_\theta[Q^{\theta,m}] \omega^{\theta,m}\\
     =\Delta Q^{\theta,m}- f^m_{\rm{BM}}(Q^{\theta,m}),\\
      \partial_t \u^{\theta,m}+\Psi_\theta[\u^{\theta,m}]\cdot \nabla \u^{\theta,m}+\nabla P^{\theta,m}\\
      =\Delta \u^{\theta,m}-\nabla (\Psi_\theta[Q^{\theta,m}])\cdot\big(\Delta Q^{\theta,m}- f^m_{\rm{BM}}(Q^{\theta,m})\big)\\
  -\Dv\left( \Psi_\theta[Q^{\theta,m}](\Delta Q^{\theta,m}- f^m_{\rm{BM}}(Q^{\theta,m}))
  -(\Delta Q^{\theta,m}- f^m_{\rm{BM}}(Q^{\theta,m}))\Psi_\theta[Q^{\theta,m}] \right), \\
      \Dv \u^{\theta,m}=0,
         \end{array}
    \right.
  \end{equation}
 in $\mathbb T^3\times [0,T]$,  subject to the initial condition \eqref{IC}.  Here $\omega^{\theta,m}=\omega(\u^{\theta,m})=\frac{\nabla \u^{\theta,m}-(\nabla\u^{\theta,m})^\top}{2}$.
  
Since the system \eqref{e3} is simply the system \eqref{e2} with
$f_{\rm{LdG}}$ replaced by $f^m_{\rm{BM}}$,
we can argue  as in the subsection 3.1 to find a global weak solution $(\u^{\theta, m}, P^{\theta,m}, Q^{\theta,m})$
of  \eqref{e3} and \eqref{IC}  in $Q_T=\mathbb T^3\times [0,T]$ such that
$$\u^{\theta,m}\in L^\infty_tL^2_x\cap L^2_tH^1_x(Q_T), \ Q^{\theta,m}\in L^\infty_tH^1_x\cap L^2_tH^2_x(Q_T),
\ P^{\theta,m}\in L^2(Q_T).$$
Moreover, by calculations similar to Lemma \ref{L0}, we deduce that $(\u^{\theta,m}, Q^{\theta,m})$  satisfies
the global energy inequality: for $0\le t\le T$, 
\begin{eqnarray}\label{3.2.1}
&&E(\u^{\theta,m},Q^{\theta,m})(t)
+\int_{\mathbb T^3\times [0,t]}\big(|\nabla\u^{\theta,m}|^2+|\D Q^{\theta,m}-f_{\rm{BM}}^m(Q^{\theta,m})|^2\big)\,dxdt\nonumber\\
&&= E(\u^{\theta,m},Q^{\theta,m})(0)\le\int_{\mathbb T^3}(\frac12|\u_0|^2+\frac12|\nabla Q_0|^2+F_{\rm{BM}}(Q_0))(x)\,dx.
\end{eqnarray} 
It follows from \eqref{3.2.1} and \eqref{convex} that
\begin{equation*}
\begin{split}
&\int_{\mathbb T^3\times [0,t]}|\D Q^{\theta,m}-f_{\rm{BM}}^m(Q^{\theta,m})|^2\,dxdt\\
&=\int_{\mathbb T^3\times [0,t]}\big(|\D Q^{\theta,m}|^2+|f_{\rm{BM}}^m(Q^{\theta,m})|^2-2\D Q^{\theta,m}\cdot f_{\rm{BM}}^m(Q^{\theta,m})\big)\,dxdt\\
&=\int_{\mathbb T^3\times [0,t]}\big(|\D Q^{\theta,m}|^2+|f_{\rm{BM}}^m(Q^{\theta,m})|^2+2{\rm{tr}}\nabla_Qf_{\rm{BM}}^m(Q^{\theta,m})
(\nabla Q^{\theta, m}, \nabla Q^{\theta, m}) \big)\,dxdt\\
&\ge \int_{\mathbb T^3\times [0,t]}\big(|\D Q^{\theta,m}|^2+|f_{\rm{BM}}^m(Q^{\theta,m})|^2-\kappa |\nabla Q^{\theta, m}|^2)\,dxdt.
\end{split}
\end{equation*}
Substituting this into \eqref{3.2.1} and applying Gronwall's inequality, we obtain that for any $0\le t\le T$, 
\begin{eqnarray}\label{3.2.10}
&&E(\u^{\theta,m},Q^{\theta,m})(t)
+\int_{\mathbb T^3\times [0,t]}\big(|\nabla\u^{\theta,m}|^2+|\D Q^{\theta,m}|^2+|f_{\rm{BM}}^m(Q^{\theta,m})|^2\big)\,dxdt\nonumber\\
&&\le e^{CT}\int_{\mathbb T^3}(\frac12|\u_0|^2+\frac12|\nabla Q_0|^2+F_{\rm{BM}}(Q_0))(x)\,dx.
\end{eqnarray} 
It follows from \eqref{3.2.1} that 
$$\sup_{0\le t\le T} \int_{\mathbb T^3}F_{\rm{BM}}^m(Q^{\theta,m})(x,t)\,dx
\le \int_{\mathbb T^3}(\frac12|\u_0|^2+\frac12|\nabla Q_0|^2+F_{\rm{BM}}(Q_0)\big)(x)\,dx.
$$
This, combined with (G2) and \eqref{lowerG2}, implies that there exists a sufficiently large $m_0=m_0(\nu, \kappa, g_0)\in \mathbb N^+$ such that for all
$m\ge m_0$, 
\begin{equation*}  
\begin{split}
&\big(\frac{m\nu}{8}-\frac{\kappa}2\big)\int_{\{x\in\mathbb T^3:\ |Q^{\theta, m}(x,t)|\ge 11\}} |Q^{\theta, m}|^2(x,t)\,dx \\
&\le\int_{\{x\in\mathbb T^3:\ |Q^{\theta, m}(x,t)|\ge 11\}}\Big[\nu(\frac{m}4 |Q^{\theta, m}|^2-g_0) -\frac{\kappa}2 |Q^{\theta, m}|^2\Big](x,t)\,dx \\
&\le \int_{\{x\in\mathbb T^3:\ |Q^{\theta, m}(x,t)|\ge 11\}} F^m_{\rm{BM}}(Q^{\theta, m})(x,t)\,dx \\
&=\int_{\mathbb T^3} F^m_{\rm{BM}}(Q^{\theta, m})(x,t)\,dx-\int_{\{x\in\mathbb T^3:\ |Q^{\theta, m}(x,t)|\le 11\}} F^m_{\rm{BM}}(Q^{\theta, m})(x,t)\,dx\\
&=\int_{\mathbb T^3} F^m_{\rm{BM}}(Q^{\theta, m})(x,t)\,dx\\
&-\int_{\{x\in\mathbb T^3:\ |Q^{\theta, m}(x,t)|\le 11\}} \Big[\nu \big(G^m_{\rm{BM}}(Q^{\theta, m})+g_0\big)-\frac{\kappa}2 |Q^{\theta, m}|^2-\nu g_0\Big](x,t)\,dx\\
&\le \int_{\mathbb T^3} F^m_{\rm{BM}}(Q^{\theta, m})(x,t)\,dx
+ \int_{\{x\in\mathbb T^3:\ |Q^{\theta, m}(x,t)|\le 11\}} (\nu g_0+\frac{\kappa}2 |Q^{\theta, m}|^2(x,t))\,dx\\
&\le \int_{\mathbb T^3}(\frac12|\u_0|^2+\frac12|\nabla Q_0|^2+F_{\rm{BM}}(Q_0)\big)(x)\,dx+(\nu g_0+\frac{121\kappa}2) |\mathbb T^3|
\end{split}
\end{equation*}
holds for any $0\le t\le T$. 
Therefore we conclude that for $m\ge m_0$, it holds that
\begin{equation}\label{l2est-Q}
\begin{split}
&\sup_{0\le t\le T} \int_{\mathbb T^3} |Q^{\theta, m}|^2(x,t)\,dx\\
&\le C\Big(\|\u_0\|_{L^2(\mathbb T^3)}, \|Q_0\|_{H^1(\mathbb T^3)}, \|F_{\rm{BM}}(Q_0)\|_{L^1(\mathbb T^3)}, \nu, g_0, \kappa\Big).
\end{split}
\end{equation}
As in subsection 3.1, the pressure function $P^{\theta, m}$ solves
\begin{equation}\label{p-equation2}
\begin{split}
&-\D P^{\theta,m}\\
&=\Dv^2\big(\Psi_{\theta}[\u^{\theta,m}]\otimes\u^{\theta,m}\big)
+\Dv\big(\nabla(\Psi_{\theta}[Q^{\theta,m}])\cdot(\D Q^{\theta,m}-f_{\rm{BM}}^m(Q^{\theta,m}))\big)
\end{split}  \ \ {\rm{in}}\ \mathbb T^3.
\end{equation}
We can apply the same argument as in the previous subsection to conclude that $P^{\theta, m}\in L^\frac53(\mathbb T^3\times [0,T])$, and 
\begin{equation}\label{pestimate1}
\big\|P^{\theta, m}\big\|_{L^\frac53(\mathbb T^3\times [0,T])}
\le C\Big(\|\u_0\|_{L^2(\mathbb T^3)}, \|Q_0\|_{H^1(\mathbb T^3)}, \|F_{\rm{BM}}(Q_0)\|_{L^1(\mathbb T^3)}, \nu, g_0, \kappa\Big).
\end{equation}

With estimates \eqref{pestimate1} and \eqref{3.2.10}, we can utilize the system \eqref{e3} to obtain that
\begin{equation}\label{ut-estimate2}
\begin{split}
&\Big\|\partial_t \u^{\theta,m}\Big\|_{L^2([0,T], W^{-1,4}(\mathbb T^3))}\\
&\leq C\Big(\|\u_0\|_{L^2(\R^3)}, \|Q_0\|_{H^1(\R^3)}, \|F_{\rm{BM}}(Q_0)\|_{L^1(\mathbb T^3)}, \nu, g_0, \kappa\Big),
\end{split}
\end{equation}
\begin{equation}\label{dt-estimate2}
\Big\|\partial_t Q^{\theta, m}\Big\|_{L^2([0,T], L^\frac32(\mathbb T^3))}
\le C\Big(\|\u_0\|_{L^2(\R^3)}, \|Q_0\|_{H^1(\R^3)}, \|F_{\rm{BM}}(Q_0)\|_{L^1(\mathbb T^3)}, \nu, g_0, \kappa\Big),
\end{equation} 
uniformly for $\theta\in (0, 1]$ and $m\ge m_0$.

For each fixed $m\ge m_0$, we can assume without loss of generality that there exists 
$$(\u^m, P^m, Q^m)\in L^\infty_tL^2_x\cap L^2_tH^1_x(Q_T)\times L^\frac53(Q_T)
\times L^\infty_tH^1_x(Q_T) $$  
such that as $\theta\to 0$, 
\begin{equation*}
\begin{cases}
&\u^{\theta,m}\rightharpoonup \u^m \ \ \ {\rm{in}}\ \ \ L^2_t H^1_x(Q_T), \\
&\u^{\theta,m}\rightarrow \u^m \ \ \ {\rm{in}}\ \ \ L^p(Q_T)\ \ \forall 1<p<\frac{10}3,\\
& P^{\theta, m} \rightharpoonup P^m \ \ {\rm{in}}\ \ L^\frac53(Q_T),\\
& Q^{\theta, m} \rightharpoonup Q^m \ \ {\rm{in}}\ \ L^2_t H^2_x(Q_T),  \\
& Q^{\theta, m} \rightarrow Q^m \ \ {\rm{in}}\ \ L^r_t L^s_x(Q_T), \ \forall 1<r,s<\infty,\\
&\D Q^{\theta, m} -f^m_{\rm{BM}}(Q^{\theta, m})\rightharpoonup\D Q^m -f^m_{\rm{BM}}(Q^{m})
\ \ \ {\rm{in}}\ \ \ L^2(Q_T),\\
&F^m_{\rm{BM}}(Q^{\theta, m})\rightarrow F^m_{\rm{BM}}(Q^{m})
\ \ \ {\rm{in}}\ \ \ L^1(Q_T).
\end{cases}
\end{equation*}
As in subsection 3.1, we can now verify that $(\u^m, P^m, Q^m)$ is a weak solution of 
\begin{equation}\label{e4}
  \left\{
    \begin{array}{l}
     \partial_t Q^{m}+\u^{m}\cdot \nabla Q^{m}-\omega^{m}Q^{m}
     +Q^{m} \omega^{m}=\Delta Q^{m}- f^m_{\rm{BM}}(Q^{m}),\\
      \partial_t \u^{m}+\u^{m}\cdot \nabla \u^{m}+\nabla (P^{m}-F^m_{\rm{BM}}(Q))
      =\Delta \u^{m}-\nabla Q^{m}\cdot\Delta Q^{m}\\
       -\Dv\left[Q^m(\Delta Q^{m}-f_{\rm{BM}}(Q^m))-
      (\Delta Q^{m}-f_{\rm{BM}}(Q^m))Q^{m} \right], \\
      \Dv \u^{m}=0,
         \end{array}
    \right.
  \end{equation}
 in $\mathbb T^3\times [0,T]$,  subject to the initial condition \eqref{IC}. 
 
 By the lower semicontinuity the following global
 energy inequality holds: for $0\le t\le T$, 
\begin{eqnarray}\label{3.2.2}
&&\int_{\mathbb T^3}(\frac12|\u^m|^2+\frac12|\nabla Q^m|^2+F^m_{\rm{BM}}(Q^m))(x,t)\,dx\nonumber\\
&&\ +\int_{\mathbb T^3\times [0,t]}\big(|\nabla\u^{m}|^2+|\D Q^{m}-f_{\rm{BM}}^m(Q^{m})|^2\big)\,dxdt\nonumber\\
&&\le \int_{\mathbb T^3}(\frac12|\u_0|^2+\frac12|\nabla Q_0|^2+F_{\rm{BM}}(Q_0))(x)\,dx,
\end{eqnarray} 
and
\begin{eqnarray}\label{3.2.100}
&&E(\u^{m},Q^{m})(t)
+\int_{\mathbb T^3\times [0,t]}\big(|\nabla\u^{m}|^2+|\D Q^{m}|^2+|f_{\rm{BM}}^m(Q^{m})|^2\big)\,dxdt\nonumber\\
&&\le e^{CT}\int_{\mathbb T^3}(\frac12|\u_0|^2+\frac12|\nabla Q_0|^2+F_{\rm{BM}}(Q_0))(x)\,dx, \ \forall t\in [0,T].
\end{eqnarray} 
Also it follows from (\ref{l2est-Q}), \eqref{pestimate1}, \eqref{ut-estimate2}, and \eqref{3.2.100} that 
\begin{eqnarray}\label{l2est-Q1}
&&\max\Big\{\big\|Q^{m}\big\|_{L^\infty_tL^2(Q_T)}, \big\|P^m\big\|_{L^\frac53(Q_T)}, \big\|\partial_t \u^{m}\big\|_{L^2_tW_x^{-1,4}(Q_T)},
\big\|\partial_t Q^{m}\big\|_{L^2_tL^\frac32_x(Q_T)}\Big\}\nonumber\\
&&\le C\Big(\|\u_0\|_{L^2(\mathbb T^3)}, \|Q_0\|_{H^1(\mathbb T^3)}, \|F_{\rm{BM}}(Q_0)\|_{L^1(\mathbb T^3)}, \nu, g_0, \kappa\Big).
\end{eqnarray}
Furthermore, we can check that $(\u^m, P^m, Q^m)$ is a suitable weak solution of \eqref{e4} by verifying that it
satisfies the local inequality \eqref{lei} with $f_{\rm{bulk}}$ replaced by
$f^m_{\rm{BM}}$.

To show that  as $m\to \infty$, $(\u^m, P^m, Q^m)$ gives rise to a suitable weak solution of \eqref{e2}, we need to first bound 
$Q^m$ in a strictly physical subdomain of the physical  domain $\mathcal{D}$, since $G_{\rm{BM}}(Q)$ blows up as $Q\in\mathcal{D}$
tends to $\partial\mathcal{D}$. This amounts to establishing an $L^\infty$-estimate of $G_{\rm{BM}}(Q)$ in terms of the $L^1$-norm of
$G_{\rm{BM}}(Q_0)$, which was previously shown by Wilkinson \cite{Wilkinson} in a slightly different setting. 

More precisely, we need the following version of a generalized maximum principle.
\begin{Lemma}\label{max0} There exist $m_0\in\mathbb N^+$ and a positive constant $C_0$, independent of $m$, such that
for all $m\ge m_0$, 
\begin{equation}\label{max00}
\big\|G^m_{\rm{BM}}(Q^m)(\cdot, t)\big\|_{L^\infty(\mathbb T^3)}
\le C_0t^{-\frac{5}2} \big\|G_{\rm{BM}}(Q_0)\big\|_{L^1(\mathbb T^3)} + C_0, \forall 0<t<T.
\end{equation}
\end{Lemma}

For now we assume Lemma \ref{max0}, which will be proved in \S4 below.  We may assume without loss of generality
that there exists 
$$(\u, P, Q)\in L^\infty_tL^2_x\cap L^2_tH^1_x(Q_T)\times  L^\frac53(Q_T)
\times L^\infty_tH^1_x\cap L^2_t H^2_x(Q_T)$$ 
such that
\begin{equation*}
\begin{cases}
&\u^{m}\rightharpoonup \u \ \ \ {\rm{in}}\ \  L^2_t H^1_x(Q_T), \\
&\u^{m}\rightarrow \u \ \ \ {\rm{in}}\ \  L^p(Q_T), \ \forall 1<p<\frac{10}3,\\
& P^{m} \rightharpoonup P \ \ {\rm{in}}\ \ L^\frac53(Q_T),\\
& Q^{m} \rightharpoonup Q \ \ {\rm{in}}\ \ L^2_t H^2_x(Q_T),  \\
& Q^{m} \rightarrow Q \ \ {\rm{in}}\ \ L^r_t L^s_x(Q_T), \ \forall 1<r,s<\infty.
\end{cases}
\end{equation*}
From \eqref{max00}, we can also deduce that for any $0<\delta<T$, 
\begin{equation}\label{max01}
\big\|G_{\rm{BM}}(Q)\big\|_{L^\infty(\mathbb T^3\times [\delta, T])}
\le (C\delta^{-\frac{5}2} +e^T) \big\|G_{\rm{BM}}(Q_0)\big\|_{L^1(\mathbb T^3)} + \kappa^2 e^{T}.
\end{equation}
By the logarithmic divergence of $G_{\rm{BM}}$ as $Q\in \mathcal{D}\rightarrow \partial\mathcal{D}$ and \eqref{max01}, we conclude that
for any $\delta>0$, there exists $\varepsilon_0=\varepsilon_0(\delta, T)>0$ such that 
\begin{equation}\label{strict1}
Q(x,t)\in \mathcal{D}_{\varepsilon_0}, \ \forall (x,t)\in \mathbb T^3\times [\delta, T],
\end{equation} 
where
\begin{equation}\label{strict-phy}
\mathcal{D}_{\varepsilon_0}:=\Big\{Q\in \mathcal{D}: \ -\frac13+\varepsilon_0\le \lambda_i(Q(x,t))
\le \frac23-\varepsilon_0, \  i=1, 2, 3\Big\}.
\end{equation}

From \eqref{max00} and the quadratic growth property of $G^m_{\rm{BM}}$, we also see that there exists $C_0>0$, independent of $m$, such that 
for $m\ge m_0$,
\begin{equation}\label{strict2}
|Q_m(x,t)|\le C_0, \ (x,t)\in \mathbb T^3\times [\delta, T].
\end{equation}
We now claim that 
\begin{equation}\label{2convergence}
f^m_{\rm{BM}}(Q^m)\rightharpoonup f_{\rm{BM}}(Q) \ {\rm{in}}\ L^2(\mathbb T^3\times [\delta, T]), \ {\rm{as}}\ m\to\infty. 
\end{equation}
To see this, first observe that \eqref{3.2.100} yields that
$f_{\rm{BM}}^m(Q^m)$ is uniformly bounded in $L^2(\mathbb T^3\times [0,T])$. Thus there exists  a function $\bar f\in L^2(\mathbb T^3\times [0, T])$ such that
$$
f_{\rm{BM}}^m(Q^m)\rightharpoonup   \bar f\in L^2(\mathbb T^3\times [0, T]).
$$
Now we want to identify $\bar f$. It follows from $Q^m\to Q$ in $L^2(\mathbb T^3\times [0,T])$ that there exists $E_m\subset\mathbb T^3\times [0,T]$,
with $|E_m|\rightarrow 0$, such that 
$$Q^m\rightarrow Q, \ {\rm{uniformly\ in}}\ \mathbb T^3\times [0,T]\setminus E_m,$$
which, combined with $Q(\mathbb T^3\times [\delta,T])\subset\mathcal{D}_{\varepsilon_0}$,  yields that for sufficiently large $m$,
$$ Q^m(\mathbb T^3\times [\delta,T]\setminus E_m)\subset\mathcal{D}_{\frac{\varepsilon_0}2}.$$
Since $f^m_{\rm{BM}}\rightarrow f_{\rm{BM}}$ in $W^{1,\infty}(\mathcal{D}_{\frac{\varepsilon_0}2})$, we conclude that
$$f^m_{\rm{BM}}(Q^m)\rightarrow f_{\rm{BM}}(Q), \ {\rm{uniformly\ in}}\ \mathbb T^3\times [\delta,T]\setminus E_m.$$
Therefore $\bar{f}= f_{\rm{BM}}(Q) $ for a.e. $(x,t)\in \mathbb T^3\times [0,T]$, and \eqref{2convergence} holds.

From \eqref{2convergence} and $
\Delta Q^m\rightharpoonup \Delta Q \ {\rm{in}}\ L^2(\mathbb T^3\times [0, T]) , \ {\rm{as}} \ m\to\infty,$
we see that
$$
\Delta Q^m-f_{\rm{BM}}^m(Q^m) \rightharpoonup \Delta Q-f_{\rm{BM}}(Q) \ {\rm{in}}\ L^2(\mathbb T^3\times [0, T]) , \ {\rm{as}} \ m\to\infty,
$$

With all the estimates at hand, it is rather standard to show that passing to the limit in 
\eqref{e4}, as $m\to\infty$ first and $\delta\to 0$ second, yields that $(\u, P, Q)$ is a weak solution of \eqref{e2}. While passing to the limit in 
the local inequality for $(\u^m, P^m, Q^m)$, as $m\to\infty$ first and then $\delta\to 0$, we can also verify that 
$(\u, P, Q)$ satisfies the local energy inequality \eqref{lei}
with $f_{\rm{bulk}}(Q)$ replaced by $f_{\rm{BM}}(Q)$.  \qed

\section{Maximum principles}

In this section, we will show the maximum principles  for any weak solution $({\bf u}, Q)$ of
\eqref{e1} and \eqref{IC} in $\R^3$ with the Landau-De Gennes potential function $F_{\rm{LdG}}(Q)$, see also
\cite{GR1, GR2}, and
in $\mathbb T^3$ with the Ball-Majumdar potential function $F_{\rm{BM}}(Q)$, see also \cite{Wilkinson}. 
These will play important roles in the proof of partial regularity of
suitable weak solutions to \eqref{e1} in the sections 5 and 6 below.

\begin{Lemma}\label{max1} For $({\bf u}_0, Q_0)\in {\bf H}\times H^1(\mathbb R^3, \mathcal{S}_0^{(3)})$,
let $({\bf u}, Q)\in L^2_tH^1_x(\mathbb R^3\times\mathbb R_+,\mathbb R^3)\times
L^2_tH^2_x(\mathbb R^3\times\mathbb R_+, \mathcal{S}_0^{(3)})$ be a weak solution of \eqref{e1}-\eqref{IC}.
If, in addition, $Q_0\in L^\infty(\mathbb R^3,\mathbb{S}_0^{(3)})$ and $c>0$, then there exists a constant
$C>0$, depending on $\|Q_0\|_{L^\infty(\mathbb R^3)}$ and $a, b, c$, such that
\begin{equation}\label{max2}
|Q(x,t)|\le C, \ \forall (x,t)\in \mathbb R^3\times\mathbb R_+.
\end{equation}
\end{Lemma}
\begin{proof} Multiplying \eqref{e1}$_1$ by $Q_{\alpha\beta}$ and taking summation over $1\le\alpha,\beta\le 3$,
we obtain
\begin{align}\label{Q1}
&\partial_t |Q|^2+{\bf u}\cdot\nabla |Q|^2+(Q\omega-\omega Q):Q\nonumber\\
&=\Delta |Q|^2-|\nabla Q|^2-2[a|Q|^2-b{\rm{tr}}(Q^3)+c{\rm{tr}}^2(Q^2)].
\end{align}
Since $\omega$ is skew-symmetric and $Q$ is symmetric, it is easy to see that
$$(Q\omega-\omega Q):Q=0, \ \forall(x,t)\in\mathbb R^3\times\mathbb R_+.$$
For $C>0$, to be determined later, define $\phi=(|Q|^2-C^2)_+$. It follows 
from \eqref{Q1} that $\phi$ satisfies
\begin{align}\label{Q2}
\partial_t\phi^2+{\bf u}\cdot\nabla \phi^2
=2(\Delta \phi)\phi -2|\nabla Q|^2\phi-4(a|Q|^2-b{\rm{tr}}(Q^3)+c{\rm{tr}}^2(Q^2))\phi
\end{align}
Integrating \eqref{Q2} over $\mathbb R^3$ and using $\nabla\cdot{\bf u}=0$, we obtain
\begin{align}\label{Q3}
&\frac{d}{dt}\int_{\mathbb R^3}\phi^2+4\int_{\mathbb R^3} (a|Q|^2-b{\rm{tr}}(Q^3)+c{\rm{tr}}^2(Q^2))\phi
\nonumber\\
&=-2\int_{\mathbb R^3}(|\nabla \phi|^2+|\nabla Q|^2\phi)\le 0.
\end{align}
Now we want to estimate the second term in the left hand side of \eqref{Q3} as follows.
It is easy to see, by Young's inequality,  that
$$b{\rm{tr}}(Q^3)\le \frac{c}{2} {\rm{tr}}^2(Q^2)+\frac{b^2}{2c} |Q|^2$$
so that
\begin{align*}
a|Q|^2-b{\rm{tr}}(Q^3)+c{\rm{tr}}^2(Q^2)&\ge \frac{c}2|Q|^4+\left(a-\frac{b^2}{2c}\right) |Q|^2\\
&\ge \frac{c}2|Q|^2\left[|Q|^2-\big(\frac{b^2}{c^2}-\frac{2a}{c}\big)\right].
\end{align*}
If we choose 
$$\displaystyle
C=\max\Big\{\|Q_0\|_{L^\infty(\mathbb R^3)}, \sqrt{(\frac{b^2}{c^2}-\frac{2a}{c})_+}\Big\}>0,$$
then we would have that
$$\int_{\mathbb R^3} (a|Q|^2-b{\rm{tr}}(Q^3)+c{\rm{tr}}^2(Q^2))\phi\ge 0,$$
and hence
$$\frac{d}{dt}\int_{\mathbb R^3}\phi^2\le 0.$$
Since $\phi(\cdot, 0)\equiv 0$ in $\mathbb R^3$, it follows that
$$\int_{\mathbb R^3}\phi^2(x,t)\,dx=0, \ \forall t\ge 0.$$
This implies that $|Q(x,t)|\le C$ for all $(x,t)\in\mathbb R^3\times\mathbb R_+$.
\end{proof}

Next we will give  a proof of Lemma \ref{max0}, which guarantees that 
$Q$ lies inside a strictly physical subdomain $\mathcal{D}_{\varepsilon_0}$ so that
$F_{\rm{BM}}(Q)$ becomes regular and hence $f_{\rm{BM}}(Q)$ is bounded. 

\medskip
\noindent{\it Proof of Lemma \ref{max0}}. It follows from the chain rule and the equation \eqref{e4}$_1$ 
that $G_{\rm{BM}}^m(Q^m)$ satisfies in the weak sense
\begin{equation}\label{weakpara1}
\begin{split}
&\partial_t (G_{\rm{BM}}^m(Q^m)) +\u^m\cdot\nabla (G_{\rm{BM}}^m(Q^m))\\
&=\D(G_{\rm{BM}}^m(Q^m))-{\rm{tr}}\nabla^2_QG_{\rm{BM}}^m(Q^m)(\nabla Q^m, \nabla Q^m)-f_{\rm{BM}}^m(Q^m)\langle\nabla_Q G_{\rm{BM}}^m(Q^m)\rangle,\\
&\le \D(G_{\rm{BM}}^m(Q^m))-(\nu \langle\nabla_Q G_{\rm{BM}}^m(Q^m)-\kappa Q^m) \langle\nabla_Q G_{\rm{BM}}^m(Q^m)\\
&\le \D(G_{\rm{BM}}^m(Q^m))+\frac{\kappa^2}{2\nu}|Q^m|^2,
\end{split}
\end{equation}
in $\mathbb T^3\times (0, T]$.
Indeed, this can be obtained by multiplying \eqref{e4}$_1$ by $\langle\nabla_Q G_{\rm{BM}}^m(Q^m)\rangle$
and using the fact $G_{\rm{BM}}^m$ is a smooth convex function.
Therefore $G_{\rm{BM}}^m(Q^m)\in L^\infty_tH^1_x(\mathbb T^3\times [0,T])$ satisfies in the weak sense
\begin{equation}\label{weakpara2}
\begin{split}
\partial_t (G_{\rm{BM}}^m(Q^m)) +\u^m\cdot\nabla (G_{\rm{BM}}^m(Q^m))\le \D(G_{\rm{BM}}^m(Q^m))+\frac{\kappa^2}{2\nu} |Q^m|^2,
\end{split}
\ \ {\rm{in}}\ \  \mathbb T^3\times (0, T].
\end{equation}
It follows from \eqref{3.2.2} and \eqref{l2est-Q1} that $Q^m\in L^2_tH^2_x(\mathbb T^3\times [0,T])$. In particular, by Sobolev's embedding theorem, we have that 
\begin{equation}\label{uniform bound}
\big\|Q^m\big\|_{L^2_tL^\infty_x(\mathbb T^3\times [0,T])}
\le C\Big(\|\u_0\|_{L^2(\mathbb T^3)}, \|Q_0\|_{H^1(\mathbb T^3)}, \|F_{\rm{BM}}(Q_0)\|_{L^1(\mathbb T^3)}, \nu, g_0, \kappa\Big).
\end{equation}

Since the drifting coefficient $\u^m$ in \eqref{weakpara2} is not smooth and $Q^m$ is not bounded in $\mathbb T^3\times [0,T]$, 
we can not directly apply the argument of  \S8 in \cite{Wilkinson} 
to prove \ref{max00}.  Here we proceed it by first considering an auxiliary equation with mollifying $\u^m$ as the drifting coefficient. 
More precisely, let $\u^m_\epsilon$ be a standard $\epsilon$-mollification on $\mathbb T^3\times [0,T]$ for $0<\epsilon<1$. Then
$\u^m_\epsilon\in C^\infty(\mathbb T^3\times [0,T])$ satisfies $\Dv \u^m_\epsilon=0$ and
$$\u^m_\epsilon\to \u^m \ {\rm{in}}\ L^2_tH^1_x(\mathbb T^3\times [0,T]), \ {\rm{as}}\ \epsilon\to 0.$$
Also let $g_\epsilon^m$  be $\epsilon$-mollifications of $|Q^m|^2$ in $\mathbb T^3\times [0,T]$,
and $h_\epsilon^m$ be $\epsilon$-mollifications of $G_{\rm{BM}}^m(Q_0)$ in $\mathbb T^3$. 
Then it follows from \eqref{uniform bound} that for all $m\ge m_0$, 
$$\big\|g^m\big\|_{L^2_tL^\infty_x(\mathbb T^3\times [0,T])}\le \big\|Q^m\big\|_{L^2_tL^\infty_x(\mathbb T^3\times [0,T])}^2,$$
$$\big\|h_\epsilon^m\big\|_{L^1(\mathbb T^3)}\le \big\|G_{\rm{BM}}(Q_0)\big\|_{L^1(\mathbb T^3)},$$
and
$$g^m_\epsilon\to |Q^m|^2 \ {\rm{in}}\ L^3(\mathbb T^3\times [0,T]),\  h_\epsilon^m\to G_{\rm{BM}}^m(Q_0) \ {\rm{in}}\ L^1(\mathbb T^3),
\ {\rm{as}}\ \epsilon\to 0.$$

Now let $v_\epsilon^m\in C^\infty(\mathbb T^3\times [0,T])$ be the unique solution of
\begin{equation}\label{auxi1}
\begin{cases}
\displaystyle\partial_t v_\epsilon^m+\u^m_\epsilon\cdot\nabla v_\epsilon^m=\D v_\epsilon^m+\frac{\kappa^2}{2\nu} g_\epsilon^m & \ {\rm{in}}\ \mathbb T^3\times [0,T],\\
\qquad\qquad\qquad v_\epsilon^m=h_\epsilon^m & \ {\rm{on}} \ \mathbb T^3\times\{0\}.
\end{cases}
\end{equation}
For $v_\epsilon^m$, we will modify the argument as illustrated in \cite{Wilkinson}, \S8, to achieve
that for $0<t<T$, 
\begin{equation}\label{uniform bound1}
\big\|v^m_\epsilon(\cdot,t)\big\|_{L^\infty(\mathbb T^3)} \le Ct^{-\frac{5}2} \big\|G_{\rm{BM}}(Q_0)\big\|_{L^1(\mathbb T^3)}+C_0.
\end{equation}
To show \eqref{uniform bound1}, decompose $v_\epsilon^m=v_1+v_2$, where $v_1$ solves
\begin{equation}\label{auxi2}
\begin{cases}
\displaystyle\partial_t v_1+\u^m_\epsilon\cdot\nabla v_1=\D v_1, & \ {\rm{in}}\ \ \ \mathbb T^3\times [0,T],\\
\displaystyle v_1=h_\epsilon^m-\frac{1}{|\mathbb T^3|}\int_{\mathbb T^3} h_\epsilon^m, & \ {\rm{on}} \ \ \ \mathbb T^3\times\{0\},
\end{cases} 
\end{equation}
and $v_2$ solves
\begin{equation}\label{auxi3}
\begin{cases}
\displaystyle\partial_t v_2+\u^m_\epsilon\cdot\nabla v_2=\D v_2 +\frac{\kappa^2}{2\nu} g_\epsilon^m, & \ {\rm{in}}\ \ \ \mathbb T^3\times [0,T],\\
\displaystyle v_2=\frac{1}{|\mathbb T^3|}\int_{\mathbb T^3} h_\epsilon^m, & \ {\rm{on}} \ \ \ \mathbb T^3\times\{0\}.
\end{cases} 
\end{equation}
For $v_1$, we can apply \cite{CKRZ} as in Lemma 8.1 of \cite{Wilkinson} to conclude that
\begin{equation}\label{v1-bound}
\big\|v_1(\cdot, t)\big\|_{L^\infty(\mathbb T^3)}\le Ct^{-\frac52} \big\|h_\epsilon^m-\frac{1}{|\mathbb T^3|}\int_{\mathbb T^3} h_\epsilon^m\big\|_{L^1(\mathbb T^3)}
\le Ct^{-\frac52}\big\|G_{\rm{BM}}(Q_0)\big\|_{L^1(\mathbb T^3)},
\end{equation}
for $0<t<T$.

While for $v_2$, we can multiply \eqref{auxi3}$_1$ by $|v_2|^{p-2}v_2$, $p>2$, and integrate the resulting equation over $\mathbb T^3$ to get
\begin{eqnarray*}
\frac{1}{p}\frac{d}{dt}\big\|v_2(t)\big\|_{L^p(\mathbb T^3)}^p &\le& \frac{\kappa^2}{2\nu} \big\|g_\epsilon^m(t)\big\|_{L^p(\mathbb T^3)}\big\|v_2(t)\big\|_{L^p(\mathbb T^3)}^{p-1}\\
&\le & \frac{\kappa^2}{2\nu} \big\|g_\epsilon^m(t)\big\|_{L^\infty(\mathbb T^3)}\big\|v_2(t)\big\|_{L^p(\mathbb T^3)}^{p-1},
\end{eqnarray*} 
so that
\begin{eqnarray*}
\frac{d}{dt}\big\|v_2(t)\big\|_{L^p(\mathbb T^3)}&\le & \frac{\kappa^2}{2\nu} \big\|g_\epsilon^m(t)\big\|_{L^\infty(\mathbb T^3)},
\end{eqnarray*} 
and hence
$$
\big\|v_2(t)\big\|_{L^p(\mathbb T^3)}\le \big\|v_2(0)\big\|_{L^p(\mathbb T^3)}+\frac{\kappa^2}{2\nu}\int_0^T \big\|g_\epsilon^m(t)\big\|_{L^\infty(\mathbb T^3)}\,dt,
\ \forall 0<t\le T.
$$
Sending $p\to\infty$ and applying \eqref{uniform bound},  we obtain that for $0<t<T$, 
\begin{equation}\label{v2-bound}
\begin{split}
&\big\|v_2(t)\big\|_{L^\infty(\mathbb T^3)}\\
&\le C\|h_\epsilon^m\|_{L^1(\mathbb T^3)}+\frac{\kappa^2}{2\nu}\int_0^T \big\|Q^m(t)\big\|_{L^\infty(\mathbb T^3)}^2\,dt\\
&\le \big\|G_{\rm{BM}}(Q_0)\big\|_{L^1(\mathbb T^3)}+
C\Big(\|\u_0\|_{L^2(\mathbb T^3)}, \|Q_0\|_{H^1(\mathbb T^3)}, \|F_{\rm{BM}}(Q_0)\|_{L^1(\mathbb T^3)}, \nu, g_0, \kappa\Big).
\end{split}
\end{equation}
Putting \eqref{v1-bound} and \eqref{v2-bound} together yields \eqref{uniform bound1}.

It is not hard to see that as $\epsilon\to 0$, there exists $v^m\in L^\infty_tL^2_x\cap L^2_tH^1_x(\mathbb T^3\times [0,T])$
such that $v_\epsilon^m\rightarrow v^m$ in $L^2(\mathbb T^3\times [0,T])$. Passing to the limit in the equation \eqref{auxi1}, we see that
$v^m$ is a weak solution of
\begin{equation}\label{auxi4}
\begin{cases}
\displaystyle\partial_t v^m+\u^m\cdot\nabla v^m=\D v^m+\frac{\kappa^2}{2\nu} |Q^m|^2 & \ {\rm{in}}\ \mathbb T^3\times [0,T],\\
\qquad\qquad\qquad v^m=G_{\rm{BM}}^m(Q_0) & \ {\rm{on}} \ \mathbb T^3\times\{0\}.
\end{cases}
\end{equation}
Moreover, passing to the limit of \eqref{uniform bound1}, we have that for any $0<t<T$, 
\begin{equation}\label{uniform bound2}
\big\|v^m(\cdot,t)\big\|_{L^\infty(\mathbb T^3)} \le Ct^{-\frac{5}2} \big\|G_{\rm{BM}}(Q_0)\big\|_{L^1(\mathbb T^3)}+C_0.
\end{equation}
Now observe that by the comparison principle on \eqref{weakpara2}, we know that for $m\ge m_0$, it holds.
$$G^m_{\rm{BM}}(Q^m)(x,t) \le v^m(\cdot,t)\le  Ct^{-\frac{5}2} \big\|G_{\rm{BM}}(Q_0)\big\|_{L^1(\mathbb T^3)}+C_0,$$
for all $(x,t)\in\mathbb T^3\times [0,T].$
This, combined with (G2), yields  \eqref{max00}. 
\qed

Note that passing to the limit in \eqref{max00}, the suitable weak solution $(\u, P, Q)$ to \eqref{e2}, constructed in \S3.2, satisfies that for
any $0<\delta<T$, 
\begin{equation}\label{uniform bound3}
\big\|G_{\rm{BM}}(Q)\big\|_{L^\infty(\mathbb T^3\times [\delta, T])}
\le C_0\delta^{-\frac{5}2}\big\|G_{\rm{BM}}(Q_0)\big\|_{L^1(\mathbb T^3)} +C_0.
\end{equation}
This completes the proof of Lemma \ref{max0}. \qed

\section{Partial regularity, Part I}

This section is devoted to establishing an $\epsilon_0$-regularity for suitable weak solutions
$({\bf u}, Q)$ of \eqref{e1} in $\O\times (0,\infty)$ in terms of renormalized $L^3$-norm of $(\u, Q)$.
The argument we will present is based on a blowing up argument, motivated by that of Lin \cite{LIN} 
on the Navier-Stokes equation, which works equally well for both the Landau-De Gennes potential $F_{\rm{LdG}}$
and the Ball-Majumdar potential $F_{\rm{BM}}$. More precisely, we want to establish the following property.

\begin{Lemma}\label{small-reg} For any $M>0$, there exist $\varepsilon_0>0$, $0<\tau_0<\frac12$, and $C_0>0$, depending on $M$,
such that if $({\bf u},Q,P)$ is a suitable weak solution of \eqref{e1}  in $\O\times (0,\infty)$, which satisfies,
for $z_0=(x_0,t_0)\in \O\times (r^2, \infty)$ and $r>0$,
\begin{equation}\label{bound}
\begin{cases} |Q|\le M & {\rm{if}}\ F_{\rm{bulk}}=F_{\rm{LdG}} \ {\rm{and}}\ \Omega=\mathbb R^3,\\
|G_{\rm{BM}}(Q)|\le M &  {\rm{if}}\ F_{\rm{bulk}}=F_{\rm{BM}}\ {\rm{and}}\ \Omega=\mathbb T^3,
\end{cases} \ \ {\rm{in}}\ \mathbb P_{r}(z_0),
\end{equation}
and 
\begin{equation}\label{small-cond}
r^{-2}\int_{\mathbb P_{r}(z_0)} (|{\bf u}|^3+|\nabla Q|^3)\,dxdt +\Big(r^{-2}\int_{\mathbb P_r(z_0)} |P|^\frac32\,dxdt\Big)^2\le\varepsilon_0^3,
\end{equation}
then
\begin{eqnarray}\label{small-decay}
&&(\tau_0r)^{-2}\int_{\mathbb P_{\tau_0 r}(z_0)} (|{\bf u}|^3+|\nabla Q|^3)\,dxdt +\Big((\tau_0r)^{-2}\int_{\mathbb P_{\tau_0 r}(z_0)}|P|^\frac32\,dxdt\Big)^2\nonumber\\
&\le&\frac12\max\Big\{r^{-2}\int_{\mathbb P_{r}(z_0)} (|{\bf u}|^3+|\nabla Q|^3)\,dxdt+\Big(r^{-2}\int_{\mathbb P_r(z_0)} |P|^\frac32\,dxdt\Big)^2, C_0r^3\Big\}.
\end{eqnarray}
\end{Lemma}
\begin{proof} We prove it by contradiction. Suppose that the conclusion were false. Then there exists $M_0>0$ such that
for any $\tau\in (0,\frac12)$, we can find $\varepsilon_i\to 0$, $C_i\rightarrow\infty$, and
$r_i>0$, and $z_i=(x_i, t_i)\in \mathbb R^3\times (r_i^2,\infty)$
such that
{
\begin{equation}\label{bound0}
\begin{cases} |Q|\le M_0 & {\rm{if}}\ F_{\rm{bulk}}=F_{\rm{LdG}},\\
|G_{\rm{BM}}(Q)|\le M_0 &  {\rm{if}}\ F_{\rm{bulk}}=F_{\rm{BM}},
\end{cases} \ \ {\rm{in}}\ \mathbb P_{r_i}(z_i),
\end{equation}
and
\begin{equation}\label{small-i}
 \ \ r_i^{-2}\int_{\mathbb P_{r_i}(z_i)} (|{\bf u}|^3+|\nabla Q|^3)\,dxdt+\big(r_i^{-2}\int_{\mathbb P_{r_i}(z_i)}|P|^\frac32\,dxdt\big)^2=\varepsilon_i^3,
\end{equation}
}
but
{
\begin{equation}\label{no-decay}
\begin{split}
&(\tau r_i)^{-2}\int_{\mathbb P_{\tau r_i}(z_i)} (|{\bf u}|^3+|\nabla Q|^3)\,dxdt+\big((\tau r_i)^{-2}\int_{\mathbb P_{\tau r_i}(z_i)}|P|^\frac32\,dxdt\big)^2\\
&>\frac12\max\big\{\varepsilon_i^3, C_i r_i^3\big\}.
\end{split}
\end{equation}
}
From \eqref{no-decay}, we see that
{
\begin{eqnarray*}
C_ir_i^3&\le& 2(\tau r_i)^{-2}\int_{\mathbb P_{\tau r_i}(z_i)} (|{\bf u}|^3+|\nabla Q|^3)\,dxdt+2\big((\tau r_i)^{-2}\int_{\mathbb P_{\tau r_i}(z_i)}|P|^\frac32\,dxdt\big)^2\\
&\le& 2\tau^{-4}\left\{ r_i^{-2} \int_{\mathbb P_{r_i}(z_i)} (|{\bf u}|^3+|\nabla Q|^3)\,dxdt+\big(r_i^{-2}\int_{\mathbb P_{ r_i}(z_i)}|P|^\frac32\,dxdt\big)^2\right\}\\
&=&2\tau^{-4}\varepsilon_i^3
\end{eqnarray*}
}
so that
$$r_i\le \big(\frac{2\varepsilon_i^3}{C_i\tau^4}\big)^\frac13\to 0.$$
Also from \eqref{bound0}, we know that there exist $C_0>0$ and $\delta_0>0$  such that in the case $F_{\rm{bulk}}=F_{\rm{BM}}$, 
\begin{equation}\label{bound00}
Q(z)\in \mathcal{D}_{\delta_0}\ {\rm{and}} \ |f_{\rm{BM}}(Q(z))|+|\nabla_Q f_{\rm{BM}}(Q(z))|\le C_0, \ \forall z\in \mathbb P_{r_i}(z_i).
\end{equation}

Define a rescaled sequence of maps
$$({\bf u}_i, Q_i, P_i)(x,t)=\big(r_i{\bf u}, {Q}, r_i^2{P}\big)(x_i+r_i x, t_i+r_i^2 t), \ \forall x\in\mathbb R^3,\ t>-1.$$
Then $({\bf u}_i, Q_i, P_i)$ is a weak solution of the scaled 
Beris-Edwards system:
\begin{equation}\label{scaled-BE}
\begin{cases}
\partial_t {Q}_i+{\bf u}_i\cdot\nabla {Q}_i-\Omega({\bf u}_i) Q_i+Q_i\Omega({\bf u}_i)=\Delta Q_i-r_i^2f_{\rm{bulk}}(Q_i),\\
\partial_t{\bf u}_i+{\bf u}_i\cdot\nabla {\bf u}_i+\nabla {P}_i
=\Delta{\bf u}_i-\nabla Q_i\D Q_i\\
-\big[(\Delta Q_i -r_i^2 f_{\rm{bulk}}(Q_i))Q_i-Q_i (\Delta Q_i-r_i^2 f_{\rm{bulk}}(Q_i))\big],\\
\Dv{\bf u}_i=0,
\end{cases}
\end{equation}
where
$$
\O({\bf u}_i)=\frac{\nabla{\bf u}_i-(\nabla{\bf u}_i)^T}{2}.
$$
Moreover, $({\bf u}_i, Q_i, P_i)$ satisfies 
{
\begin{equation}\label{small-i0}
\int_{\mathbb P_{1}(0)} (|{\bf u}_i|^3+|\nabla Q_i|^3)\,dxdt+\big(\int_{\mathbb P_{1}(0)}|P_i|^\frac32\,dxdt\big)^2=\varepsilon_i^3,
\end{equation}
}
and
\begin{equation}\label{no-decayi0}
\tau ^{-2}\int_{\mathbb P_{\tau}(0)} (|{\bf u}_i|^3+|\nabla Q_i|^3)\,dxdt+\big(\tau ^{-2}\int_{\mathbb P_{\tau}(0)}|P_i|^\frac32\,dxdt\big)^2
>\frac12\max\big\{\varepsilon_i^3, C_i r_i^3\big\}.
\end{equation}

Define the blowing-up sequence 
$(\widehat{\bf u}_i, \widehat{Q}_i, \widehat{P}_i): \mathbb P_1(0)\mapsto \mathbb R^3\times
\mathcal{S}_0^3\times \mathbb R$, of $({\bf u}_i, Q_i, P_i)$, by letting 
$$
(\widehat{\bf u}_i, \widehat{Q}_i, \widehat{P}_i)(z)
=\big(\frac{{\bf u}_i}{\epsilon_i}, \frac{Q_i-\overline{Q}_i}{\epsilon_i}, \frac{P_i}{\epsilon_i}\big)(z), \ \forall z=(x,t)\in \mathbb P_1(0),
$$
where 
$$\overline{Q}_i=\frac{1}{|\mathbb P_1(0)|}\int_{\mathbb P_1(0)} Q_i$$
denotes the average of $Q_i$ over $\mathbb P_1(0)$.
Then $(\widehat{\bf u}_i, \widehat{Q}_i, \widehat{P}_i)$ satisfies
{
\begin{equation}\label{scaled-bound}
\begin{cases}\displaystyle
\int_{\mathbb P_1(0)}\widehat{Q}_i=0,\\
\displaystyle\int_{\mathbb P_1(0)}(|\widehat{\bf u}_i|^3+|\nabla \widehat{Q}_i|^3)\,dxdt+\big(\int_{\mathbb P_1(0)}|\widehat{P}_i|^\frac32\,dxdt\big)^2=1,\\
\displaystyle\tau^{-2}\int_{\mathbb P_\tau(0)}(|\widehat{\bf u}_i|^3+|\nabla \widehat{Q}_i|^3)\,dxdt
+\big(\tau^{-2}\int_{\mathbb P_\tau(0)}|\widehat{P}_i|^\frac32\,dxdt\big)^2
>\frac12\max\big\{1, C_i\frac{r_i^3}{\varepsilon_i^3}\big\},
\end{cases}
\end{equation}
}
and  $(\widehat{\bf u}_i, \widehat{Q}_i, \widehat{P}_i)$ is a suitable weak solution of
the following scaled Beris-Edwards equation:
\begin{equation}\label{scaled-BE1}
\begin{cases}
\partial_t \widehat{Q}_i+\varepsilon_i\widehat{\bf u}_i\cdot\nabla \widehat{Q}_i-\Omega(\widehat{\bf u}_i) {Q}_i+{Q}_i\Omega(\widehat{\bf u}_i)
=\Delta \widehat{Q}_i-\frac{r_i^2}{\varepsilon_i}f_{\rm{bulk}}(Q_i),\\
\partial_t\widehat{\bf u}_i+\varepsilon_i\widehat{\bf u}_i\cdot\nabla \widehat{\bf u}_i+\nabla \widehat{P}_i
=\Delta\widehat{\bf u}_i-\varepsilon_i\nabla \widehat{Q}_i\D \widehat{Q}_i\\
-\big[Q_i (\Delta\widehat{Q}_i-\frac{r_i^2}{\varepsilon_i} f_{\rm{bulk}}(Q_i))-(\Delta \widehat{Q}_i  -\frac{r_i^2}{\varepsilon_i} f_{\rm{bulk}}(Q_i))Q_i\big]\\
\Dv\widehat{\bf u}_i=0,
\end{cases}
\end{equation}
From \eqref{scaled-bound}, we assume that there exists
$$(\widehat{\bf u}, \widehat{Q},\widehat{P})\in L^3(\mathbb P_1(0))\times L^3_tW^{1,3}_x(\mathbb P_1(0))\times L^\frac32(\mathbb P_1(0))$$ 
such that, after passing to a subsequence,
$$(\widehat{\bf u}_i, \widehat{Q}_i,\widehat{P}_i)\rightharpoonup (\widehat{\bf u}, \widehat{Q},\widehat{P})
\ {\rm{in}}\ L^3(\mathbb P_1(0))\times L^3_tW^{1,3}_x(\mathbb P_1(0))\times L^\frac32(\mathbb P_1(0)).$$
It follows from \eqref{scaled-bound} and the lower semicontinuity that 
{
\begin{equation}\label{limit-bound}
\int_{\mathbb P_1(0)}(|\widehat{\bf u}|^3+|\nabla \widehat{Q}|^3)+\big(\int_{\mathbb P_1(0)}|\widehat{P}|^\frac32\big)^2\le 1.
\end{equation}
}
Moreover, we claim that
\begin{equation}\label{scaled-bound10}
\big\|\widehat{\bf u}_i\big\|_{L^\infty_tL^2_x(\mathbb P_\frac12(0))\cap L^2_tH^1_x(\mathbb P_\frac12(0))}
+\big\|\nabla\widehat{Q}_i\big\|_{L^\infty_tL^2_x(\mathbb P_\frac12(0))\cap L^2_tH^1_x(\mathbb P_\frac12(0))}
\le C<\infty.
\end{equation}
To show \eqref{scaled-bound10},  choose a cut-off function $\phi\in C_0^\infty(\mathbb P_1(0))$ such that 
$$0\le \phi\le 1,\ \phi\equiv 1 \ {\rm{on}}\  \mathbb P_{\frac12}(0), \ {\rm{and}}\ |\partial_t\phi|+|\nabla\phi|+|\nabla^2\phi|\le C.$$
Define 
$$\phi_i(x,t)=\phi\big(\frac{x-x_i}{r_i}, \frac{t-t_i}{r_i^2}\big),  \ \forall (x,t)\in \mathbb R^3\times (0,\infty).$$
Applying Lemma 2.2 with $\phi$ replaced by $\phi_i^2$ and applying H\"older's inequality, we would arrive at
\begin{eqnarray*}
&&\sup_{t_i-\frac{r_i^2}{4}\le t\le t_i}\int_{B_{{r_i}}(x_i)}(|{\bf u}|^2+|\D Q|^2)\phi_i^2\,dx
+\int_{\mathbb P_{{r_i}}(z_i)}(|\nabla{\bf u}|^2+|\nabla^2 Q|^2)\phi_i^2\,dxdt\\
&&\leq C\big[\int_{\mathbb P_{r_i}(z_i)}(|{\bf u}|^2+|\nabla Q|^2)|(\partial_t+\D)\phi_i^2|\,dxdt\\
&&\qquad+\int_{\mathbb P_{r_i}(z_i)}(|{\bf u}|^2+|\nabla Q|^2+|P|)|{\bf u}||\nabla\phi_i^2|\,dxdt+\int_{\mathbb P_{r_i}(z_i)}|\nabla Q|^2||\nabla^2(\phi_i^2)|\\
&&\qquad+\int_{\mathbb P_{r_i}(z_i)} (|\D {Q}| + |f_{\rm{bulk}}(Q)|)|\u||\nabla\phi_i^2|+|\nabla_Q f_{\rm{bulk}}(Q)||\nabla Q|^2\phi_i^2\,dxdt\big].
\end{eqnarray*}
Observe that
$$\int_{\mathbb P_{r_i}(z_i)} |\D {Q}| |\u||\nabla\phi_i^2|\,dxdt \le \frac12 \int_{\mathbb P_{r_i}(z_i)} |\D {Q}|^2\phi_i^2\,dxdt
+C\int_{\mathbb P_{r_i}(z_i)} |\u|^2|\nabla\phi_i|^2\,dxdt.
$$
Substituting this into the above inequality and performing rescaling, we obtain that
\begin{eqnarray}\label{renormalize}
&&\sup_{-\frac{1}{4}\le t\le 0}\int_{B_{\frac{1}2}(0)}(|\widehat{\bf u}_i|^2+|\D \widehat{Q}_i|^2)\,dx
+\int_{\mathbb P_{\frac12}(0)}(|\nabla\widehat{\bf u}_i|^2+|\nabla^2 \widehat{Q}_i|^2)\,dxdt\nonumber\\
&&\leq C\Big[\int_{\mathbb P_{1}(0)}(|\widehat{\bf u}_i|^2+|\nabla \widehat{Q}_i|^2)
+ (\varepsilon_i|\widehat{\bf u}_i|^2+\varepsilon_i|\nabla \widehat{Q}_i|^2+|\widehat{P}_i|)|\widehat{\bf u}_i|\,dxdt\Big]\nonumber\\
&& \quad +C\Big[\int_{\mathbb P_{1}(0)} \frac{r_i^2}{\varepsilon_i}| \widehat{\u}_i|\,dxdt
+ r_i^2\int_{\mathbb P_{1}(0)}|\nabla \widehat{Q}_i|^2\,dxdt\Big]\nonumber\\
&&\leq C(1+\frac{r_i^2}{\varepsilon_i}+r_i^2)\le C.
\end{eqnarray}
This yields \eqref{scaled-bound10}. From \eqref{scaled-bound10}, we may also assume that
\begin{equation}
\label{22convergence}
(\widehat{\bf u}_i, \widehat{Q}_i)\rightharpoonup (\widehat{\bf u}, \widehat{Q})
\ {\rm{in}}\ L^2_tH^1_x(\mathbb P_\frac12(0))\times L^2_tH^2_x(\mathbb P_\frac12(0)).
\end{equation}
Since $r_i\le\varepsilon_i$ and by \eqref{bound00} $|Q_i|\le M_0$ and $|f_{\rm{bulk}}(Q_i)|+|\nabla_Qf_{\rm{bulk}}(Q_i)|\le C_0$ in $\mathbb P_1(0)$, 
there exists a constant  $\overline{Q}\in \mathcal{S}_0^{(3)}$, with $|\overline{Q}|\le M_0$, such that, after passing to a subsequence,
$${Q_i\to \overline{Q} \ \ \ {\rm{in}}\ \ \ L^3(\mathbb P_\frac12(0)),}$$
and
$$
\frac{r_i^2}{\varepsilon_i}f_{\rm{bulk}}(Q_i)\to 0
 \ \ \ {\rm{in}}\ \ \ L^\infty(\mathbb P_\frac12(0)).
$$
Hence $(\widehat{\bf u},\widehat{Q},\widehat{P}): \mathbb P_\frac12(0)\mapsto \mathbb R^3\times\mathcal{S}_0^{(3)}\times \mathbb R$
solves the linear system:
\begin{equation}\label{limit-eqn}
\begin{cases}
\partial_t \widehat{Q}-\Delta \widehat{Q}=\Omega(\widehat{\bf u}) \overline{Q}-\overline{Q}\Omega(\widehat{\bf u}),\\
\partial_t\widehat{\bf u}-\Delta\widehat{\bf u}+\nabla \widehat{P}
=-\Dv(\overline{Q} \Delta\widehat{Q}-\Delta \widehat{Q} \overline{Q}),\\
\Dv\widehat{\bf u}=0,
\end{cases}
\end{equation}

Applying Lemma \ref{reg-limit-eqn} and \eqref{limit-bound}, we know that 
$$(\widehat{\bf u}, \widehat{Q})\in C^\infty(\mathbb P_\frac14), \ \widehat{P}\in L^\infty([-(\frac1{4})^2, 0], C^\infty(B_\frac14(0)))$$
satisfies
\begin{align}\label{decay0}
&\tau^{-2}\int_{\mathbb P_\tau(0)} (|\widehat{\bf u}|^3+|\nabla\widehat{Q}|^3)\,dxdt + \big(\tau^{-2}\int_{\mathbb P_\tau(0)}|\widehat{P}|^\frac32\,dxdt\big)^2\nonumber\\
&\le C\tau^3 \int_{\mathbb P_\frac12(0)}(|\widehat{\bf u}|^3+|\nabla\widehat{Q}|^3)\,dxdt +\big(\int_{\mathbb P_1(0)}|\widehat{P}|^\frac32\big)^2\nonumber\\
&\le C\tau^3, \ \forall\ \tau\in (0,\frac18).
\end{align}

We now claim that 
\begin{equation}\label{L3-conv1.1}
(\widehat{\bf u}_i,\nabla\widehat{Q}_i)\rightarrow (\widehat{\bf u},\nabla\widehat{Q})
\ {\rm{in}}\ L^3(\mathbb P_\frac38(0)).
\end{equation}
To prove \eqref{L3-conv1.1},  first observe that \eqref{renormalize} and the equation \eqref{scaled-BE1} imply that
$$\partial_t \widehat{\bf u}_i\in \big(L^2_tH^{-1}+L^2_tL^\frac65_x+L^\frac32_tW^{-1,\frac32}_x\big)\big(\mathbb P_\frac38(0)\big);
\ \partial_t \widehat{Q}_i\in L^\frac32_tL^\frac32_x(\mathbb P_\frac38(0)),$$ 
enjoy the following uniform bounds:
\begin{eqnarray*}
&&\big\|\partial_t \widehat{\bf u}_i\big\|_{\big(L^2_tH^{-1}_x+L^2_tL^\frac65_x+L^\frac32_tW^{-1,\frac32}_x\big)(\mathbb P_\frac38(0))}\\
&&\le C\big[\|\widehat{\bf u}_i\|_{L^\infty_tL^2_x(\mathbb P_\frac12(0))}
+\|\nabla\widehat{\bf u}_i\|_{L^2_tL^2_x(\mathbb P_\frac12(0))}+\|\nabla\widehat{Q}_i\|_{L^3(\mathbb P_\frac12(0))}^2
+\|\nabla^2\widehat{Q}_i\|_{L^2(\mathbb P_\frac12(0))}\big]\\
&&\le C,
\end{eqnarray*}
and
\begin{eqnarray*}
&&\big\|\partial_t \widehat{Q}_i\big\|_{L^\frac32(\mathbb P_\frac38(0))}\\
&&\le C\big[\|\widehat{Q}_i\|_{L^2_t H^1_x(\mathbb P_\frac12(0))}+\|\nabla\widehat{\bf u}_i\|_{L^2(\mathbb P_\frac12(0))}
+\|\nabla\widehat{Q}_i\|_{L^3(\mathbb P_\frac12(0))}+\|\widehat{\bf u}_i\|_{L^3(\mathbb P_\frac12(0))}\big]\\
&&\le C.
\end{eqnarray*}
Thus we can apply Aubin-Lions' compactness Lemma to conclude the $L^3$-strong convergence as in \eqref{L3-conv1.1}.

It follows from the $L^3$-strong convergence property \eqref{L3-conv1.1} that for any  $\tau\in (0,\frac18)$, 
\begin{equation}\label{no-gap}
\displaystyle\tau^{-2}\int_{\mathbb P_\tau(0)}(|\widehat{\bf u}_i|^3+|\nabla \widehat{Q}_i|^3)
=\displaystyle\tau^{-2}\int_{\mathbb P_\tau(0)}(|\widehat{\bf u}|^3+|\nabla \widehat{Q}|^3)+\tau^{-2} o(1)
\le C\tau^3+\tau^{-2} o(1),
\end{equation} 
where $o(1)$ stands for a quantity such that $\displaystyle\lim_{i\rightarrow\infty}o(1)=0$.

Now we need to estimate the pressure $\widehat{P}_i$. First, by taking divergence of the second equation 
\eqref{scaled-BE}$_2$, we see that $\widehat{P}_i$ solves
\begin{equation}\label{P-eqn}
\Delta \widehat{P}_i=-\epsilon_i \Dv^2\big[\widehat{\bf u}_i\otimes \widehat{\bf u}_i+(\nabla \widehat{Q}_i\otimes \nabla\widehat{Q}_i-\frac12|\nabla \widehat{Q}_i|^2I_3)\big]
\ {\rm{in}} \ B_1,
\end{equation}
where we have applied Lemma \ref{freed2} to guarantee
$$\Dv^2\big[Q_i(\Delta\widehat{Q}_i-\frac{r_i^2}{\varepsilon_i}f_{\rm{bulk}}(Q_i))
-(\Delta\widehat{Q}_i - \frac{r_i^2}{\varepsilon_i}f_{\rm{bulk}}(Q_i))Q_i\big] =0\  \ {\rm{in}}\ \ B_1.
$$
We need to show that
\begin{equation}\label{P-est}
\tau^{-2}\int_{\mathbb P_{\tau}(0)}|\widehat{P}_i|^\frac32 \,dxdt \le C\tau^{-2}(\varepsilon_i+o(1))+C\tau, \ \forall i\ge 1.
\end{equation}

To prove \eqref{P-est}, let $\eta\in C_0^\infty(B_1(0))$ be a cut-off function such that
$\eta\equiv 1$ in $B_\frac38(0)$, $0\le\eta\le 1$. For any $-(\frac38)^2\le t\le 0$, define $\widehat{P}_i^{(1)}(\cdot, t): \mathbb R^3\to\mathbb R$ by letting
\begin{eqnarray}\label{auxil1}
\widehat{P}_i^{(1)}(x,t)
=\int_{\mathbb R^3} \nabla^2_x G(x-y) \eta(y)\varepsilon_i[\widehat{\bf u}_i\otimes\widehat{\bf u}_i
+(\nabla\widehat{Q}_i\otimes \nabla\widehat{Q}_i-\frac12|\nabla \widehat{Q}_i|^2I_3)](y, t)\,dy,
\end{eqnarray}
where $G(\cdot)$ is the fundamental solution of $-\Delta$ in $\mathbb R^3$. 
Then it is easy to check that $\widehat{P}_i^{(2)}(\cdot, t)=(\widehat{P}_i- \widehat{P}_i^{(1)})(\cdot, t)$ satisfies
\begin{equation}\label{auxi20}
-\Delta \widehat{P}_i^{(2)}(\cdot, t)=0 \ \ {\rm{in}}\ \ B_\frac38(0).
\end{equation}
For $\widehat{P}_i^{(1)}$, we can apply the Calderon-Zygmund theory to show that 
\begin{eqnarray}\label{P-est1}
\big\|\widehat{P}_i^{(1)}\big\|_{L^\frac32(\mathbb R^3)}\le C\epsilon_i\big[\|\widehat{\bf u}_i\|_{L^3(B_1(0))}^2+\|\nabla\widehat{Q}_i\|_{L^3(B_1(0))}^2\big]
\end{eqnarray}
so that
\begin{eqnarray}\label{P-est2}
\big\|\widehat{P}_i^{(1)}\big\|_{L^\frac32(\mathbb P_\frac13(0))}
&\le& C\varepsilon_i(\|\widehat{\bf u}_i\|_{L^3(\mathbb P_1(0))}^2
+\|\nabla\widehat{Q}_i\|_{L^3(\mathbb P_1(0))}^2)\nonumber\\
&\le&C(\varepsilon_i+o(1)).
\end{eqnarray}
From the standard theory on harmonic functions,  $\widehat{P}_i^{(2)}(\cdot, t)\in C^\infty(B_\frac12(0))$
satisfies: for any $0<\tau<\frac14$,
\begin{eqnarray}\label{P-est3}
\tau^{-2}\int_{\mathbb P_\tau(0)} |\widehat{P}_i^{(2)}|^\frac32
&\le& C\tau \int_{\mathbb P_\frac13(0)}|\widehat{P}_i^{(2)}|^\frac32
\le C\tau \Big[\int_{\mathbb P_\frac13(0)}\big(|\widehat{P}_i|^\frac32+|\widehat{P}_i^{(1)}|^\frac32\big)\nonumber\\
&\le& C\tau(1+\varepsilon_i+o(1)).
\end{eqnarray}
Putting \eqref{P-est2} and \eqref{P-est3} together, we obtain \eqref{P-est}.

It follows from 
\eqref{no-gap} and \eqref{P-est} that there exist sufficiently small $\tau_0\in (0,\frac14)$ and sufficiently large $i_0$,
depending on $\tau_0$, such that for any $i\ge i_0$, it holds that
$$\displaystyle\tau_0^{-2}\int_{\mathbb P_{\tau_0}(0)}(|\widehat{\bf u}_i|^3+|\nabla \widehat{Q}_i|^3)\,dxdt
+\displaystyle\big(\tau_0^{-2}\int_{\mathbb P_{\tau_0}(0)}|\widehat{P}_i|^\frac32\,dxdt)^2
\le \frac14.$$
This contradicts to \eqref{scaled-bound}.  The proof of Lemma \ref{small-reg} is completed.
\end{proof}
We now need to establish the smoothness of the limit equation \eqref{limit-eqn}, namely,

\begin{Lemma}\label{reg-limit-eqn} Assume that $(\widehat{\bf u}, \widehat{Q})\in 
(L^\infty_tL^2_x\cap L^2_tH^1_x)(\mathbb P_\frac12)\times (L^\infty_tH^1_x\cap L^2_tH^2_x)(\mathbb P_\frac12)$
and $\widehat{P}\in L^\frac32(\mathbb P_\frac12)$ is a weak solution of the linear system
\eqref{limit-eqn}, then $(\widehat{\bf u}, \widehat{Q})\in C^\infty(\mathbb P_\frac14)$, and
the following estimate
\begin{equation}\label{decay}
\theta^{-2}\int_{\mathbb P_\theta} (|\widehat{\bf u}|^3+|\nabla\widehat{Q}|^3+|\widehat{P}|^\frac32)
\le C\theta^3 \int_{\mathbb P_\frac12}(|\widehat{\bf u}|^3+|\nabla\widehat{Q}|^3+|\widehat{P}|^\frac32)
\end{equation}
holds for any $\theta\in (0,\frac18)$.
\end{Lemma}
\begin{proof} The regularity of the limit equation \eqref{limit-eqn} doesn't follow from the standard theory
of linear parabolic equations in \cite{LSN}, since the source term ${\rm{div}}(\overline{Q}\Delta\widehat{Q}-\Delta\widehat{Q}\overline{Q})$
in the second equation of \eqref{limit-eqn} depends on third order derivatives of $\widehat{Q}$. It is based
on higher order energy methods, for which the cancellation property, as in the derivation of
local energy inequality for suitable  weak solutions of \eqref{e1},  plays a critical role. 

For nonnegative multiple indices $\alpha$, $\beta$, and $\gamma$ such that
$\alpha=\beta+\gamma$ and $\gamma$ is of order $1$, it is easy to see that 
$(\nabla^\alpha \widehat{Q}, \nabla^{\beta}\widehat{\bf u}, \nabla^{\beta}\widehat{P})$ satisfies
\begin{equation}\label{limit-eqn1}
\begin{cases}
\partial_t (\nabla^\alpha\widehat{Q})-\Delta (\nabla^\alpha\widehat{Q})
=\Omega(\nabla^\alpha\widehat{\bf u}) \overline{Q}-\overline{Q}\Omega(\nabla^\alpha\widehat{\bf u}),\\
\partial_t(\nabla^\beta\widehat{\bf u})-\Delta(\nabla^\beta\widehat{\bf u})+\nabla (\nabla^\beta\widehat{P})
=-\Dv[\overline{Q} \Delta(\nabla^\beta\widehat{Q})-\Delta (\nabla^\beta\widehat{Q}) \overline{Q}],\\
\Dv(\nabla^\beta\widehat{\bf u})=0,
\end{cases}
\end{equation}
Now we want to derive an arbitrarily higher order local energy inequality for \eqref{limit-eqn1}. 
For any given $\phi\in C_0^\infty(\mathbb P_\frac12(0))$, multiplying the first equation of
\eqref{limit-eqn1} by $\nabla^\alpha\widehat{Q}\phi^2$ and integrating over $\mathbb R^3$,
we obtain that by summing over all $\gamma$,
\begin{eqnarray}\label{ID3}
&&\frac{d}{dt}\int_{\mathbb R^3} \frac12{|\nabla(\nabla^\beta\widehat{Q})|^2}\phi^2 +\int_{\mathbb R^3} |\nabla^2(\nabla^\beta\widehat{Q})|^2\phi^2\nonumber\\
&&=\int_{\mathbb R^3} \frac12{|\nabla(\nabla^\beta\widehat{Q})|^2}(\partial_t+\Delta)\phi^2\nonumber\\
&&\ \ +\int_{\mathbb R^3} (\overline{Q}\Omega(\nabla^\beta\widehat{\bf u})-\Omega(\nabla^\beta\widehat{\bf u})\overline{Q}):(\Delta(\nabla^\beta\widehat{Q})\phi^2
+\nabla(\nabla^\beta\widehat{Q})\cdot\nabla\phi^2).
\end{eqnarray}
While, by multiplying the second equation of  \eqref{limit-eqn} by $\nabla^\beta\widehat{\bf u}\phi^2$ and integrating over $\mathbb{R}^3$, we obtain that
\begin{eqnarray}\label{ID4}
&&\frac{d}{dt}\int_{\mathbb R^3} \frac12{|\nabla^\beta\widehat{\bf u}|^2}\phi^2 
+\int_{\mathbb R^3} |\nabla(\nabla^\beta\widehat{\bf u})|^2\phi^2\nonumber\\
&&=\int_{\mathbb R^3} \frac12{|\nabla^\beta\widehat{\bf u}|^2}(\partial_t+\Delta)\phi^2
+\int_{\mathbb R^3} \nabla^\beta\widehat{P} \nabla^\beta\widehat{\bf u}\cdot\nabla \phi^2\nonumber\\
&&+\int_{\mathbb R^3} (\overline{Q}\Delta(\nabla^\beta\widehat{Q})-\Delta(\nabla^\beta\widehat{Q})\overline{Q}):
(\nabla(\nabla^\beta\widehat{\bf u})\phi^2
+\nabla^\beta\widehat{\bf u}\otimes\nabla\phi^2).
\end{eqnarray}
As in above, we observe that
$$
\int_{\mathbb R^3} [(\overline{Q}\Omega(\nabla^\beta\widehat{\bf u})-\Omega(\nabla^\beta\widehat{\bf u})\overline{Q}):\Delta(\nabla^\beta\widehat{Q})\phi^2
+(\overline{Q}\Delta(\nabla^\beta\widehat{Q})-\Delta(\nabla^\beta\widehat{Q})\overline{Q}):
\nabla(\nabla^\beta\widehat{\bf u})\phi^2]=0.
$$
Also, if we decompose $\beta=\beta_1+\beta_2$, where $\beta_2$ is of order $1$,
then by integration by parts we have that 
$$
\int_{\mathbb R^3} \nabla^\beta\widehat{P} \nabla^\beta\widehat{\bf u}\cdot\nabla \phi^2
=-\int_{\mathbb R^3} \nabla^{\beta_1}\widehat{P} (\nabla^{\beta+\beta_2}\widehat{\bf u}\cdot\nabla \phi^2
+\nabla^{\beta}\widehat{\bf u}\cdot \nabla(\nabla^{\beta_2}(\phi^2))
$$
so that
$$
\big|\int_{\mathbb R^3} \nabla^\beta\widehat{P} \nabla^\beta\widehat{\bf u}\cdot\nabla \phi^2\big|
\le C\int_{\mathbb R^3} |\nabla^{|\beta|-1}\widehat{P}| (|\nabla^{|\beta|+1}\widehat{\bf u}||\nabla (\phi^2)|
+|\nabla^{|\beta|}\widehat{\bf u}||\nabla^2(\phi^2)|).
$$
Hence, by adding \eqref{ID3} and \eqref{ID4} together and then taking summation over all $\beta$'s with $|\beta|=k\ge 0$,
we obtain that
\begin{eqnarray*}
&&\frac{d}{dt}\int_{\mathbb R^3} \frac12(|\nabla^k\widehat{\bf u}|^2+{|\nabla^{k+1}\widehat{Q}|^2})\phi^2 
+\int_{\mathbb R^3} (|\nabla^{k+1}\widehat{\bf u}|^2+|\nabla^{k+2}\widehat{Q}|^2)\phi^2\nonumber\\
&&\le\int_{\mathbb R^3} \frac12({|\nabla^k\widehat{\bf u}|^2}+|\nabla^{k+1}\widehat{Q}|^2)(|\partial_t(\phi^2)|+|\nabla^2(\phi^2)|)
\nonumber\\
&&+C\int_{\mathbb R^3} |\nabla^{k-1}\widehat{P}| (|\nabla^{k+1}\widehat{\bf u}||\nabla (\phi^2)|
+|\nabla^{k}\widehat{\bf u}||\nabla^2(\phi^2)|)\nonumber\\
&&\ + C\int_{\mathbb R^3}\big(|\nabla^{k+1}\widehat{\bf u}||\nabla^{k+1}\widehat{Q}|
+|\nabla^{k}\widehat{\bf u}||\nabla^{k+2}\widehat{Q}|\big)|\nabla\phi^2|\nonumber\\
&&\le\int_{\mathbb R^3} \frac12({|\nabla^k\widehat{\bf u}|^2}+|\nabla^{k+1}\widehat{Q}|^2)(|\partial_t(\phi^2)|+
|\nabla^2(\phi^2)|)\nonumber\\
&&+C\int_{\mathbb R^3} |\nabla^{k-1}\widehat{P}| (|\nabla^{k+1}\widehat{\bf u}||\nabla (\phi^2)|
+|\nabla^{k}\widehat{\bf u}||\nabla^2(\phi^2)|)\nonumber\\
&&+\frac12\int_{\mathbb R^3} (|\nabla^{k+1}\widehat{\bf u}|^2+|\nabla^{k+2}\widehat{Q}|^2)\phi^2
+C\int_{\mathbb R^3}\big(|\nabla^{k}\widehat{\bf u}|^2+|\nabla^{k+1}\widehat{Q}|^2\big)|\nabla\phi|^2,
\end{eqnarray*}
which implies that
\begin{eqnarray}\label{ID5}
&&\frac{d}{dt}\int_{\mathbb R^3} (|\nabla^k\widehat{\bf u}|^2+{|\nabla^{k+1}\widehat{Q}|^2})\phi^2 
+\int_{\mathbb R^3} (|\nabla^{k+1}\widehat{\bf u}|^2+|\nabla^{k+2}\widehat{Q}|^2)\phi^2\nonumber\\
&&\le C \int_{\mathbb R^3} ({|\nabla^k\widehat{\bf u}|^2}+|\nabla^{k+1}\widehat{Q}|^2)(|\partial_t(\phi^2)|+
|\nabla^2(\phi^2)|)\nonumber\\
&&\ +C\int_{\mathbb R^3} |\nabla^{k-1}\widehat{P}| (|\nabla^{k+1}\widehat{\bf u}||\nabla (\phi^2)|
+|\nabla^{k}\widehat{\bf u}||\nabla^2(\phi^2)|)\nonumber\\
&&\ +C\int_{\mathbb R^3}\big(|\nabla^{k}\widehat{\bf u}|^2+|\nabla^{k+1}\widehat{Q}|^2\big)|\nabla\phi|^2.
\end{eqnarray}
It follows from the second equation of \eqref{limit-eqn} that  $\nabla^\beta\widehat{P}$ solves
\begin{equation}\label{P-equation4}
\Delta(\nabla^\beta \widehat{P})=-{\rm{div}}^2\big[ \overline{Q}\Delta(\nabla^\beta \widehat{Q})-\Delta(\nabla^\beta \widehat{Q})\overline{Q} \big]=0, \ {\rm{in}}\ B_\frac38(0),
\end{equation}
where we have applied Lemma \ref{freed2}. Hence by the standard theory of linear elliptic equations, 
\begin{equation}\label{P-est5}
\int_{\mathbb P_\frac14(0)}|\nabla^k \widehat{P}|^2
\le C\int_{\mathbb P_\frac13(0)}|\nabla^{k-1}\widehat{P}|^2.
\end{equation}
By choosing suitable test functions $\phi$, it is not hard to see that \eqref{P-est5} and \eqref{ID5} imply that for $k\ge 0$, 
\begin{eqnarray}\label{ID6}
&&\sup_{-\frac1{16}\le t\le 0}\int_{B_\frac14(0)}
(|\nabla^k\widehat{\bf u}|^2+|\nabla^{k+1}\widehat{Q}|^2)
+\int_{\mathbb P_\frac14(0)}(|\nabla^{k+1}\widehat{\bf u}|^2+|\nabla^{k+2}\widehat{Q}|^2+|\nabla^{k}\widehat{P}|^2)
\nonumber\\
&&\ \le C\int_{\mathbb P_\frac38(0)}(|\nabla^{k}\widehat{\bf u}|^2+|\nabla^{k+1}\widehat{Q}|^2+|\nabla^{k-1}\widehat{P}|^2)
\end{eqnarray}
It is clear that with suitable adjustment of radius, applying  (\ref{ID6} inductively on $k$ yields that
\begin{eqnarray}\label{ID7}
&&\sup_{-\frac1{16}\le t\le 0}\int_{B_\frac14(0)}
(|\nabla^k\widehat{\bf u}|^2+|\nabla^{k+1}\widehat{Q}|^2)
+\int_{\mathbb P_\frac14(0)}(|\nabla^{k+1}\widehat{\bf u}|^2+|\nabla^{k+2}\widehat{Q}|^2+|\nabla^{k}\widehat{P}|^2)
\nonumber\\
&&\ \le C\int_{\mathbb P_\frac38(0)}(|\nabla\widehat{\bf u}|^2+|\nabla^{2}\widehat{Q}|^2+|\nabla\widehat{P}|^2),
\ \forall k\ge 1.
\end{eqnarray}

With \eqref{ID7}, we can apply the regularity theory for both the linear Stokes equation and the linear parabolic equation
to conclude that $(\widehat{\bf u}, \widehat{Q})\in C^\infty(\mathbb P_\frac14(0))$.
Furthermore, applying the elliptic estimate for the pressure equation \eqref{P-eqn} we see that
$\nabla^k\widehat{P}\in C^0(\mathbb P_\frac14(0))$ for any $k\ge 1$. For $l\ge 1$, 
taking $t$-derivative $\partial_t^l$  of both sides of
\eqref{P-eqn}, we can also see that $\nabla^k\partial_t^l \widehat{P}\in C^0(\mathbb P_\frac14(0))$.
Therefore $(\widehat{\bf u}, \widehat{Q}, \widehat{P})\in C^\infty(\mathbb P_\frac14(0))$ and the estimate \eqref{decay} holds. This completes the proof of Lemma \ref{reg-limit-eqn}.
\end{proof}

Now we can iterate Lemma \ref{small-reg} and utilize the Reisz potential estimates in Morrey spaces
to obtain the following  $\varepsilon_0$-regularity.

\begin{Lemma}\label{small-reg1} 
For any $M>0$, there exists $\varepsilon_0>0$, depending on $M$,
such that if $({\bf u},Q,P)$ is a suitable weak solution of \eqref{e1}  in $\O\times (0,\infty)$, which satisfies,
for $z_0=(x_0,t_0)\in \O\times (r_0^2, \infty)$ and 
\begin{equation}\label{bound11}
\begin{cases} |Q|\le M & {\rm{if}}\ F_{\rm{bulk}}=F_{\rm{LdG}} \ {\rm{and}}\ \Omega=\mathbb R^3,\\
|G_{\rm{BM}}(Q)|\le M &  {\rm{if}}\ F_{\rm{bulk}}=F_{\rm{BM}}\ {\rm{and}}\ \Omega=\mathbb T^3,
\end{cases} \ \ {\rm{in}}\ \mathbb P_{r_0}(z_0),
\end{equation}
and 
\begin{equation}\label{small-cond11}
r_0^{-2}\int_{\mathbb P_{r_0}(z_0)} (|{\bf u}|^3+|\nabla Q|^3)\,dxdt +\Big(r_0^{-2}\int_{\mathbb P_{r_0}(z_0)} |P|^\frac32\,dxdt\Big)^2\le\varepsilon_0^3,
\end{equation}
then for any $1<p<\infty$, $(\u, P, \nabla Q)\in L^p(\mathbb P_{\frac{r_0}4}(z_0))$ and
\begin{equation}\label{lpestimate}
\big\|(\u, P, \nabla Q)\big\|_{L^p(\mathbb P_{\frac{r_0}4}(z_0))}\le C(p, \varepsilon_0, M).
\end{equation}
\end{Lemma}
\begin{proof} From \eqref{small-cond11}, we have
\begin{equation}\label{small-cond12}
\big(\frac{r_0}2\big)^{-2}\int_{\mathbb P_{\frac{r_0}2}(z)} (|{\bf u}|^3+|\nabla Q|^3)\,dxdt +\Big(\big(\frac{r_0}2\big)^{-2}\int_{\mathbb P_{\frac{r_0}2}(z)} |P|^\frac32\,dxdt\Big)^2\le 8\varepsilon_0^3
\end{equation}
holds for any $z\in \mathbb P_{\frac{r_0}2}(z_0)$.  By applying Lemma \ref{small-reg} repeatedly on $\mathbb P_{\frac{r_0}2}(z)$
for $z\in \mathbb P_{\frac{r_0}2}(z_0)$, there are $C_0>0$ and $\tau_0\in (0,\frac12)$ that 
for any  $k\ge 1$, 
\begin{eqnarray}\label{decay10}
&&(\tau_0^{k} r_0)^{-2}\int_{\mathbb P_{\tau_0^kr_0}(z)} (|{\bf u}|^3+|\nabla Q|^3)\,dxdt +\big((\tau_0^{k} r_0)^{-2}\int_{\mathbb P_{\tau_0^kr_0}(z)} |P|^\frac32\,dxdt)^2\\
&&\le 2^{-k}\max\Big\{ (\frac{r_0}2)^{-2}\int_{\mathbb P_{\frac{r_0}2}(z)} (|{\bf u}|^3+|\nabla Q|^3)\,dxdt
+\big((\frac{r_0}2)^{-2}\int_{\mathbb P_{\frac{r_0}2}(z)} |P|^\frac32\,dxdt\big)^2, \nonumber\\
&&\qquad \qquad \qquad \frac{C_0 r_0^3}{1-2\tau_0^3}\Big\}.\nonumber
\end{eqnarray}
Therefore for $\theta_0=\frac{\ln 2}{3|\ln \tau_0|}\in (0,\frac13)$, it holds that
for any $0<s<\frac{r_0}2$ and  $z\in \mathbb P_{\frac{r_0}2}(z_0)$
\begin{equation}\label{morrey1}
s^{-2}\int_{\mathbb P_{s}(z)} (|{\bf u}|^3+|\nabla Q|^3+ |P|^\frac32)\,dxdt
\le C(1+\varepsilon_0^3)\big(\frac{s}{r_0}\big)^{3\theta_0}.
\end{equation}
By \eqref{bound11} and Lemma \ref{max0}, there exists $C>0$, depending on $M$, such that
\begin{equation}\label{Qbound}
|Q|+|f_{\rm{bulk}}(Q)|+|\nabla_Qf_{\rm{bulk}}(Q)|\le C\  {\rm{in}}\ \mathbb P_{r_0}(z_0).
\end{equation}
Now we can apply the local energy inequality \eqref{lei} for $(\u, P, Q)$ on $\mathbb P_{\frac{r_0}2}(z)$, for $z\in \mathbb P_{\frac{r_0}2}(z_0)$,  to get
that for $0<s<\frac{r_0}2$,
\begin{equation}\label{morrey2}
\begin{split}
&s^{-1}\int_{\mathbb P_s(z)} (|\nabla u|^2+|\D Q|^2)\,dxdt \\
&\le C\Big[ (2s)^{-3}\int_{\mathbb P_{2s}(z)} (|u|^2+|\nabla Q|^2)
+(2s)^{-2}\int_{\mathbb P_{2s}(z)} (|u|^3+|\nabla Q|^3+|P|^\frac32)\\
&\qquad+(2s)^{-2}\int_{\mathbb P_{2s}(z)} |u|+(2s)^{-1}\int_{\mathbb P_{2s}(z)} |\nabla Q|^2\Big]\\
&\le C(1+\varepsilon_0^3)\big(\frac{s}{r_0}\big)^{2\theta_0}.
\end{split}
\end{equation}

Next we employ the estimate of Reisz potentials in Morrey spaces to prove the smoothness of $(\u, P, Q)$ near $z_0$, analogous to
that by Huang-Wang \cite{HW}, Hineman-Wang \cite{HLW1}, and Huang-Lin-Wang \cite{HLW}.  

For any open set $U\subset\mathbb R^3\times \mathbb R$, $1\le p<\infty$, and $0\le\lambda\le 5$, define
the Morrey space $M^{p,\lambda}(U)$ by
$$
M^{p,\lambda}(U)
:=\Big\{ f\in L^p_{\rm{loc}}(U):\ \big\|f\big\|_{M^{p,\lambda}(U)}^p 
=\displaystyle\sup_{z\in U, r>0} r^{\lambda-5}\int_{\mathbb P_r(z)} |f|^p\,dxdt<\infty\Big\}.
$$
It follows from \eqref{morrey1} and \eqref{morrey2} that there exists $\alpha\in (0,1)$ such that
$$(\u, \nabla Q)\in M^{3, 3(1-\alpha)}\big(\mathbb P_{\frac{r_0}2}(z_0)\big), 
\ P\in M^{\frac32, 3(1-\alpha)}\big(\mathbb P_{\frac{r_0}2}(z_0)\big), 
\ (\nabla \u, \nabla^2 Q)\in M^{2, 4-2\alpha}\big(\mathbb P_{\frac{r_0}2}(z_0)\big).
$$
Write \eqref{e2}$_1$ as
\begin{equation}\label{e5.1}
\partial_t Q-\D Q= f,  \  \ f\equiv-\u\cdot\nabla Q+\omega Q-Q\omega -f_{\rm{bulk}}(Q)\in M^{\frac32, 3(1-\alpha)}\big(\mathbb P_{\frac{r_0}2}(z_0)\big).
\end{equation}
Let $\eta\in C_0^\infty(\mathbb R^4)$ be a cut off function of $\mathbb P_{\frac{r_0}2}(z_0)$ such that 
$0\le\eta\le 1$, $\eta=1$ in $\mathbb P_{\frac{r_0}2}(z_0)$, $|\partial_t\eta|+|\nabla^2\eta|\le Cr_0^2$, 
Set $w=\eta^2 (Q-Q_{z_0, r_0})$, where $Q_{z_0,r_0}$ is the average of $Q$ over $\mathbb P_{\frac{r_0}2}(z_0)$. Then
\begin{equation}\label{e5.2}
\partial_t w-\D w= F, \ \ F:=\eta^2 f+(\partial_t\eta^2-\D\eta^2) (Q-Q_{z_0, r_0})- \nabla\eta^2 \cdot\nabla Q.
\end{equation}
We can check that $F\in M^{\frac32, 3(1-\alpha)}(\R^4)$ and satisfies
\begin{equation}\label{morrey4}
\big\|F\big\|_{M^{\frac32, 3(1-\alpha)}(\R^4)}\le C(1+\varepsilon_0).
\end{equation}
Let $\Gamma$ denote the heat kernel in $\R^3$. Then
$$|\nabla\Gamma|(x,t)\le C\delta^{-4}((x,t), (0,0)), \ \forall (x,t)\not=(0,0),$$
where $\delta(\cdot, \cdot)$ denotes the parabolic distance on $\R^4$.
By the Duhamel formula, we have that
\begin{equation}\label{duhamel1}
|w(x,t)|\le \int_0^t\int_{\R^3} |\nabla \Gamma(x-y, t-s)||F(y,s)|\,dyds\le C\mathcal{I}_1(|F|)(x,t),
\end{equation}
where $\mathcal{I}_\beta$ is the Reisz potential of order $\beta$ on $\R^4$, $\beta\in [0,4]$, defined by
$$
\mathcal{I}_\beta(g)(x,t)=\int_{\R^4}  \frac{|g(y,s)|}{\delta^{5-\beta}((x,t), (y,s))}\,dyds, \ \forall g\in L^1(\R^4).
$$
Applying the Reisz potential estimates (see \cite{HW}  Theorem 3.1), we conclude that $\nabla w\in M^{\frac{3(1-\alpha)}{1-2\alpha}, 3(1-\alpha)}(\R^4)$
and 
\begin{equation}\label{morrey5}
\Big\|\nabla w\Big\|_{M^{\frac{3(1-\alpha)}{1-2\alpha}, 3(1-\alpha)}(\R^4)} \le C \Big\|F\Big\|_{M^{\frac32, 3(1-\alpha)}(\R^4)}\le C(1+\varepsilon_0).
\end{equation}
Since $\displaystyle\lim_{\alpha\uparrow \frac12}\frac{3(1-\alpha)}{1-2\alpha}=\infty$, we conclude that for any $1<p<\infty$, 
$\nabla w\in L^p(\mathbb P_{r_0}(z_0))$ and
\begin{equation}\label{lpestimate1}
\big\|\nabla w\big\|_{L^p(\mathbb P_{r_0}(z_0))}\le C(p, r_0, \varepsilon_0). 
\end{equation}
Since $Q-w$ solves
$$\partial_t(Q-w)-\D(Q-w)=0 \ \ \ {\rm{in}}\ \ \ \mathbb P_{\frac{r_0}2}(z_0),$$
it follows from the theory of heat equations that for any $1<p<\infty$, $\nabla Q \in \mathbb P_{\frac{r_0}2}(z_0)$ and
\begin{equation}\label{lpestimate2}
\big\|\nabla Q\big\|_{L^p(\mathbb P_{\frac{r_0}2}(z_0))}\le C(p, r_0, \varepsilon_0). 
\end{equation} 

We now proceed with the estimation of $\u$. Let $\v:\R^3\times (0,\infty)\mapsto \R^3$ solve the Stokes equation:
\begin{equation}\label{stokes}
\begin{cases}
\partial_t\v -\Delta\v +\nabla P= - \Dv\big[\eta^2\big(\u\otimes\u +(\nabla Q\otimes\nabla Q-\frac12|\nabla Q|^2I_3)\big)\big]\\
\qquad+\Dv\big[\eta^2(Q(\D Q-f_{\rm{bulk}}(Q))-  (\D Q-f_{\rm{bulk}}(Q))Q)\big]   & {\rm{in}}\ \R^4_+,\\
\Dv\v=0 & {\rm{in}}\ \R^4_+,\\
\v(\cdot, 0)=0 & {\rm{in}}\ \R^3.
\end{cases}
\end{equation}
By using the Oseen kernel (see Leray \cite{Leray}), an estimate of $\v$ can be given by 
\begin{equation}\label{duhamel2}
|\v(x,t)|\le C\mathcal{I}_1(|X|)(x,t), \ \forall (x,t)\in\R^3\times (0,\infty),
\end{equation}
where 
$$X=\eta^2\big[\u\otimes\u +(\nabla Q\otimes\nabla Q-\frac12|\nabla Q|^2I_3)
+(Q(\D Q-f_{\rm{bulk}}(Q))-  (\D Q-f_{\rm{bulk}}(Q))Q)\big].
$$
As above, we can check that $X\in M^{\frac32, 3(1-\alpha)}(\R^4)$ and
\begin{equation*}
\begin{split}
\big\|X\big\|_{M^{\frac32, 3(1-\alpha)}(\R^4)}&\le C\Big[\|\u\|_{M^{3, 3(1-\alpha)}(\mathbb P_{\frac{r_0}2}(z_0))}^2+
\|\nabla Q\|_{M^{3, 3(1-\alpha)}(\mathbb P_{\frac{r_0}2}(z_0))}^2\\
&\ \ \ \ \ \ +\|\D Q-f_{\rm{bulk}}(Q)\|_{M^{3, 3(1-\alpha)}(\mathbb P_{\frac{r_0}2}(z_0))}\Big]\\
&\le C(1+\varepsilon_0).
\end{split}
\end{equation*}
Hence we conclude that  $\v\in M^{\frac{3(1-\alpha)}{1-2\alpha}, 3(1-\alpha)}(\R^4)$
and 
\begin{equation}\label{morrey6}
\Big\|\v\Big\|_{M^{\frac{3(1-\alpha)}{1-2\alpha}, 3(1-\alpha)}(\R^4)} \le C \Big\|X\Big\|_{M^{\frac32, 3(1-\alpha)}(\R^4)}\le C(1+\varepsilon_0).
\end{equation}
As $\alpha\uparrow \frac12$, we conclude that for any $1<p<\infty$, 
$\v\in L^p(\mathbb P_{r_0}(z_0))$ and
\begin{equation}\label{lpestimate3}
\big\|\v\big\|_{L^p(\mathbb P_{r_0}(z_0))}\le C(p, r_0, \varepsilon_0). 
\end{equation}
Note that $\u-\v$ solves the linear homogeneous Stokes equation in $\mathbb P_{\frac{r_0}2}(z_0)$:
$$\partial_t(\u-\v)-\D(\u-\v)+\nabla P=0, \ \Dv(\u-\v)=0 \ \ {\rm{in}}\ \ \mathbb P_{\frac{r_0}2}(z_0).$$
Then $\u-\v\in L^\infty(\mathbb P_{\frac{r_0}4}(z_0))$. Therefore for any $1<p<\infty$,
$\u\in L^p(\mathbb P_{\frac{r_0}4}(z_0))$ and
\begin{equation}\label{lpestimate4}
\big\|\u\big\|_{L^p(\mathbb P_{\frac{r_0}4}(z_0)}\le C(p, r_0, \varepsilon_0).
\end{equation}
For $P$, since it satisfies the Poisson equation: for $t_0-\frac{r_0^2}4\le t\le t_0$, 
\begin{equation}\label{P-equation5}
-\D P=\Dv^2\big[\u\otimes\u+(\nabla Q\otimes\nabla Q-\frac12|\nabla Q|^2I_3)\big] \ \ {\rm{in}}\ \ B_{\frac{r_0}2}(x_0).
\end{equation} 
Hence $P\in L^p(\mathbb P_{\frac{r_0}4}(z_0))$ and satisfies the \eqref{lpestimate}.  The proof is now complete. \end{proof}
 
The higher order regularity of \eqref{e2} does not follow from the standard theory, since the equation for $\u$ involves $\nabla^3Q$ and
 the equation for $Q$ involves $\nabla\u$. It turns out the higher order regularity of \eqref{e2} can be obtained through higher oder 
 energy methods. Roughly speaking, if $(\u, P, \nabla Q)$ is in $L^p$ for any $1<p<\infty$, then \eqref{e2} can be viewed as a perturbed version of the linear equation \eqref{limit-eqn} with controllable error terms. Here higher order
versions of the cancellation properties \eqref{can2} and \eqref{free-div2} in the local energy inequality \eqref{lei} also plays an important role.
This kind of idea has been previously employed by Huang-Lin-Wang (see \cite{HLW} Lemma 3.4) for general Ericksen-Leslie systems in dimension two. 
 More precisely, we have 
 \begin{Lemma} \label{small-reg2} Under the same assumptions as Lemma \ref{small-reg1}, we have that for any $k\ge 0$, 
 $(\nabla^k\u, \nabla^{k+1}Q)\in \big(L^\infty_tL^2_x\cap L^2_tH^1_x\big)(\mathbb P_{\frac{1+2^{-(k+1)}}{2}r_0}(z_0))$ and the following
 estimates hold
 \begin{equation}\label{higherenergy1}
 \begin{split}
& \sup_{t_0-\big(\frac{(1+2^{-(k+1)})}{2}r_0\big)^2\le t\le t_0}\int_{B_{\frac{1+2^{-(k+1)}}{2}r_0}(x_0)} (|\nabla^k\u|^2
+|\nabla^{k+1}Q|^2)\,dx \\
&\qquad +\int_{\mathbb P_{\frac{1+2^{-(k+1)}}{2}r_0}(z_0)} (|\nabla^{k+1}\u|^2
+|\nabla^{k+2}Q|^2+|\nabla^k P|^{\frac53})\,dxdt\\
&\qquad\le C(k, r_0,) \varepsilon_0.
\end{split}
\end{equation}
In particular, $(\u, Q)$ is smooth in $\mathbb P_{\frac{r_0}4}(z_0)$.
\end{Lemma}
\begin{proof} For simplicity, assume $z_0=(0,0)$ and $r_0=8$. \eqref{higherenergy1} can be proved by
an induction on $k$. It is clear that when $k=0$,   \eqref{higherenergy1} follows directly from the local energy 
inequality \eqref{lei}. Here we indicate how to prove \eqref{higherenergy1} for $k=1$.
First, recall from Lemma \ref{small-reg1} that for any $i\in \mathbb N^+$
and $1<p<\infty$, 
\begin{equation} \label{prebound}
\big\|Q\big\|_{L^\infty(\mathbb P_2)}+ \big\|\nabla^i f_{\rm{bulk}}(Q)\big\|_{L^\infty(\mathbb P_2)}\le C(i, \varepsilon_0), 
\  \big\|(\u,P, \nabla Q)\big\|_{L^p(\mathbb P_2)}\le C(p) \varepsilon_0. 
\end{equation}
Taking spatial derivative of \eqref{e1}\footnote{Strictly speaking, we need to take finite quotient $D_h^j$
of \eqref{e1} $(j=1,2,3$) and then sending $h\to 0$}, we have 
\begin{equation}\label{e6.1}
\begin{cases}
&\partial_tQ_\alpha+\u\cdot\nabla Q_\alpha +\u_\alpha\cdot\nabla Q-\omega_\alpha Q+Q\omega_\alpha-\omega Q_\alpha
+Q_\alpha\omega\\
&=\D Q_\alpha-(f_{\rm{bulk}}(Q))_\alpha,\\
&\partial_t \u_\alpha +\u\cdot\nabla \u_\alpha+\u_\alpha\cdot\nabla \u+\nabla P_\alpha
=\D \u_\alpha-\nabla Q\D Q_\alpha-\nabla Q_\alpha \Delta Q\\
&\qquad\qquad\qquad\quad+\Dv \big[Q(\Delta Q-f_{\rm{bulk}}(Q))-(\Delta Q-f_{\rm{bulk}}(Q))Q\big]_\alpha, \\
&\Dv\u_\alpha=0,
\end{cases} \ {\rm{in}}\  \mathbb P_1.
\end{equation}
Here $\omega_\alpha=\omega(\u_\alpha)$. Let $\eta\in C_0^\infty(B_2)$ be such that 
$$0\le\eta\le 1, \ \eta\equiv1 \ {\rm{in}}\ B_{1+2^{-2}},\  \eta\equiv 0 \ {\rm{out}}\ B_{1+2^{-1}}, \ |\nabla\eta|+|\nabla^2\eta|\le 16.$$
Taking $\nabla$ of \eqref{e6.1}$_1$ and multiplying it by $\nabla Q_\alpha\eta^2$, and multiplying \eqref{e6.1}$_2$ by $\nabla\u_\alpha \eta^2$,  and then integrating resulting equations over $B_2$\footnote{strictly speaking, we need to multiply $\Delta (D^j_h Q)\eta^2$ and $\nabla(D_h^j\u)\eta^2$
and then sending $h\to 0$}, we obtain that
\begin{eqnarray*}
&&\frac12\frac{d}{dt}\int_{\O} |\nabla^2 Q|^2\eta^2-\int_{\R^3}(\u_\alpha\cdot\nabla) Q\cdot \Delta Q_\alpha\eta^2
-\int_{\O} (\u\cdot\nabla) Q_\alpha \cdot(\D Q_\alpha\eta^2+\nabla Q_\alpha\nabla\eta^2)\\
&&-\int_{\O}(\u_\alpha\cdot\nabla) Q\cdot\nabla Q_\alpha \nabla \eta^2-\int_{\O} (-\omega_\alpha Q+Q\omega_\alpha)\cdot(\D Q_\alpha\eta^2+\nabla Q_\alpha\nabla\eta^2)\\
&&=\int_{\O} \big[(-\omega Q_\alpha +Q_\alpha \omega)
- (\D Q_\alpha-(f_{\rm{bulk}}(Q))_\alpha)\big]\cdot(\D Q_\alpha\eta^2+\nabla Q_\alpha\nabla\eta^2),
\end{eqnarray*} 
and
\begin{eqnarray*}
&&\frac12\frac{d}{dt}\int_{\O} |\nabla \u|^2\eta^2-\int_{\O}\frac{|\nabla\u|^2}2 \u\cdot\nabla\eta^2
+\int_{\O}(\u_\alpha\cdot\nabla) \u\cdot \u_\alpha\eta^2 -\int_{\O}P_\alpha\u_\alpha\cdot\nabla\eta^2\\
&&=-\int_{\O}(|\nabla^2 \u|^2\eta^2-\frac{|\nabla\u|^2}{2}\D \eta^2)-\int_{\O}((\u_\alpha\cdot\nabla) Q\cdot\D Q_\alpha\eta^2
+(\u_\alpha\cdot\nabla) Q_\alpha\cdot\Delta Q\eta^2)\\
&&\ \ \ -\int_{\O}[Q_\alpha (\Delta Q-f_{\rm{bulk}}(Q))
-(\Delta Q-f_{\rm{bulk}}(Q))Q_\alpha]\cdot(\nabla\u_\alpha\eta^2+\u_\alpha\otimes\nabla\eta^2)\\
&&\ \ \ -\int_{\O}[Q (\Delta Q-f_{\rm{bulk}}(Q))_\alpha-(\Delta Q-f_{\rm{bulk}}(Q))_\alpha Q]\cdot(\nabla\u_\alpha\eta^2+\u_\alpha\otimes\nabla\eta^2).
\end{eqnarray*} 
Adding these two equations together and regrouping terms, and using the cancellation identity
$$
\int_{\O}(-\omega_\alpha Q+Q\omega_\alpha)\cdot \D Q_\alpha \eta^2
=\int_{\O} (Q\D Q_\alpha-\D Q_\alpha Q)\cdot \nabla \u_\alpha\eta^2,
$$
we arrive at
\begin{eqnarray*}
&&\frac12\frac{d}{dt}\int_{\O}(|\nabla \u|^2+ |\nabla^2 Q|^2)\eta^2+\int_{\O} (|\nabla^2\u|^2+|\Delta\nabla Q|^2)\eta^2\\
&&=\int_{\O} [(\u\cdot\nabla) Q_\alpha\cdot (\D Q_\alpha\eta^2+\nabla Q_\alpha\nabla\eta^2)+(\u_\alpha\cdot)\nabla Q\cdot\nabla Q_\alpha \nabla \eta^2]\\
&&+\int_{\O} (-\omega_\alpha Q+Q\omega_\alpha-\D Q_\alpha):\nabla Q_\alpha\nabla\eta^2\\
&&+\int_{\O} \big[(-\omega Q_\alpha +Q_\alpha \omega)
+(f_{\rm{bulk}}(Q))_\alpha\big]:(\D Q_\alpha\eta^2+\nabla Q_\alpha\nabla\eta^2)\\
&&+ \int_{\O}[\frac{|\nabla\u|^2}2 (\D  \eta^2+\u\cdot\nabla\eta^2)
-\u_\alpha\cdot(\nabla \u\cdot \u_\alpha+\nabla Q_\alpha: \D Q)\eta^2 +P_\alpha \u_\alpha\cdot\nabla\eta^2]\\
&&-\int_{\O}[Q_\alpha (\Delta Q-f_{\rm{bulk}}(Q))-(\Delta Q-f_{\rm{bulk}}(Q))Q_\alpha]:(\nabla\u_\alpha\eta^2+\u_\alpha\otimes\nabla\eta^2)\\
&& -\int_{\O}[Q (\Delta Q-f_{\rm{bulk}}(Q))_\alpha-(\Delta Q-f_{\rm{bulk}}(Q))_\alpha Q]:\u_\alpha\otimes\nabla\eta^2\\
&& +\int_{\O}[Q f_{\rm{bulk}}(Q)_\alpha-f_{\rm{bulk}}(Q))_\alpha Q]:(\nabla\u_\alpha\eta^2+\u_\alpha\otimes\nabla\eta^2)
:=\sum_{i=1}^7 A_i.
\end{eqnarray*}
We can estimate $A_i$'s separately as follows.
$$
|A_7|\le \frac1{16} \int_{\O}|\nabla^2 \u|^2\eta^2+C\int_{\O}(|\nabla Q|^2(|\eta|^2+|\nabla\eta|^2) +|\nabla\u|^2\eta^2),
$$
$$|A_6|\le \frac1{16} \int_{\O}|\D\nabla Q|^2\eta^2+C\int_{\O}(|\nabla Q|^2\eta^2+|\nabla\u|^2(\eta^2+|\nabla\eta|^2),$$
$$|A_5|\le \frac1{16} \int_{\O}|\nabla^2 \u|^2\eta^2+C\int_{\O}|\nabla Q|^2|\D Q|^2\eta^2
+C\int_{\O}|\nabla \u|^2|\nabla\eta|^2,
$$
\begin{equation*}
\begin{split}
|A_4|&\le \frac18\int_{\O} (|\nabla^2 \u|^2+|\D\nabla Q|^2)\eta^2
+C\int_{\O}[|\nabla\u|^2|\D\eta^2|+|\u|^2(|\nabla\u|^2+|\D Q|^2)\eta^2]\\
&\quad+C\int_{\O} (|\nabla\u|^2+|\D Q|^2)|\nabla\eta|^2+C\int_{\O} (|P|^2|\nabla\eta|^2+|P||\nabla\u||\D\eta^2|),
\end{split}
\end{equation*}
\begin{equation*}
\begin{split}
|A_3|&\le \frac1{16}\int_{\O} |\D\nabla Q|^2\eta^2
+C\int_{\O}|\nabla Q|^2(|\nabla\u|^2+|\D Q|^2)\eta^2\\
&+C\int_{\O} (|\nabla Q|^2\eta^2+|\nabla\u|^2|\nabla\eta|^2),
\end{split}
\end{equation*}
\begin{equation*}
\begin{split}
|A_2|&\le \frac1{16}\int_{\O} |\D\nabla Q|^2\eta^2
+C\int_{\O}(|\nabla\u|^2+|\D Q|^2)|\nabla\eta|^2,
\end{split}
\end{equation*}
\begin{equation*}
\begin{split}
|A_1|&\le \frac1{16}\int_{\O} |\D\nabla Q|^2\eta^2
+C\int_{\O}[(|\u|^2|+|\nabla Q|^2)\D Q|^2\eta^2+(|\nabla\u|^2+|\D Q|^2)|\nabla\eta|^2].
\end{split}
\end{equation*}
Substituting these estimates on $A_i$'s into the above inequality, we obtain that
\begin{eqnarray*}
&&\frac{d}{dt}\int_{\O}(|\nabla \u|^2+ |\nabla^2 Q|^2)\eta^2+\int_{\O} (|\nabla^2\u|^2+|\Delta\nabla Q|^2)\eta^2\\
&&\le C \int_{B_{1+2^{-1}}}(|\u|^2+ |\nabla Q|^2+|\nabla\u|^2+|\D Q|^2+|P|^2)\\
&&\ \ \ +C\int_{\O} (|\u|^2|\nabla\u|^2+|\u|^2|\D Q|^2+|\nabla Q|^2|\D Q|^2+|\nabla Q|^2|\nabla\u|^2)\eta^2.
\end{eqnarray*}
Now we want to estimate the second term in the right hand side. By Sobolev-interpolation inequalities, we have
\begin{equation*}
\begin{split}
&\int_{\O} |\u|^2|\nabla\u|^2\eta^2\\
&\le \|\nabla\u \eta\|_{L^2(\O)} \|\nabla\u \eta\|_{L^3(\O)} \|\u\|_{L^{12}(B_{1+2^{-1}})}^2\\
&\le C \|\nabla\u \eta\|_{L^2(\O)}\|\nabla\u \eta\|_{L^2(\O)}^\frac12\|\nabla(\nabla\u \eta)\|_{L^2(\O)}^\frac12
 \|\u\|_{L^{12}(B_{1+2^{-1}})}^2\\
 &\le C \|\nabla\u \eta\|_{L^2(\O)}\|\nabla(\nabla\u \eta)\|_{L^2(\O)}\|\u\|_{L^{12}(B_{1+2^{-1}})}^2\\
 &\le  \frac18 \int_{\O} |\nabla^2 \u|^2\eta^2+C\int_{B_{1+2^{-1}}}|\nabla\u|^2
 + C\|\u\|_{L^{12}(B_{1+2^{-1}})}^4\int_{\O}|\nabla\u|^2\eta^2, 
\end{split}
\end{equation*}
\begin{equation*}
\begin{split}
\int_{\O} |\u|^2|\D Q|^2\eta^2
 &\le  \frac18 \int_{\O} |\D\nabla Q|^2\eta^2+C\int_{B_{1+2^{-1}}}|\D Q|^2\\
& + C\|\u\|_{L^{12}(B_{1+2^{-1}})}^4\int_{\O}|\D Q|^2\eta^2, 
\end{split}
\end{equation*}
\begin{equation*}
\begin{split}
\int_{\O} |\nabla Q|^2|\D Q|^2\eta^2
 &\le  \frac18 \int_{\O} |\D\nabla Q|^2\eta^2+C\int_{B_{1+2^{-1}}}|\D Q|^2\\
 &+ C\|\nabla Q\|_{L^{12}(B_{1+2^{-1}})}^4\int_{\O}|\D Q|^2\eta^2, 
\end{split}
\end{equation*}
and 
\begin{equation*}
\begin{split}
\int_{\O} |\nabla Q|^2|\nabla\u|^2\eta^2
 &\le  \frac18 \int_{\O} |\nabla\u|^2\eta^2+C\int_{B_{1+2^{-1}}}|\nabla\u|^2\\
 &+ C\|\nabla Q\|_{L^{12}(B_{1+2^{-1}})}^4\int_{\O}|\nabla\u|^2\eta^2.
\end{split}
\end{equation*}
Substituting these estimates into the above inequality, we would arrive at
\begin{eqnarray}\label{higherenergy2}
&&\frac{d}{dt}\int_{\O}(|\nabla \u|^2+ |\nabla^2 Q|^2)\eta^2+\int_{\O} (|\nabla^2\u|^2+|\Delta\nabla Q|^2)\eta^2\nonumber\\
&&\le C \int_{B_{1+2^{-1}}}(|\u|^2+ |\nabla Q|^2+|\nabla\u|^2+|\D Q|^2+|P|^2)\nonumber\\
&&\ \ \ +C(1+\|(\u, \nabla Q)\|_{L^{12}(B_{1+2^{-1}})}^{12}) \int_{\O}(|\nabla \u|^2+ |\nabla^2 Q|^2)\eta^2.
\end{eqnarray}
From \eqref{prebound}, we can apply Gronwall's inequality to \eqref{higherenergy2} to show that
\eqref{higherenergy1} holds for $k=1$. For $k\ge 2$, we can perform an induction argument as in
\cite{HLW} Lemma 3.4. We leave the details to interested readers. 

It is readily seen that by the Sobolev embedding theorem, Lemma \ref{small-reg1} implies that 
$(\nabla^k u, \nabla^{k+1}Q) \in L^\infty(\mathbb P_{\frac{r_0}4}(z_0))$ for any $k\ge 1$.
This, combined with the theory of linear
Stokes equation and heat equation, would imply the smoothness of $(\u, Q)$ in 
$\mathbb P_{\frac{r_0}4}(z_0)$. This completes the proof.
\end{proof}

Applying Lemma \ref{small-reg1}, we can prove a weaker version of Theorem \ref{main}.

\begin{Proposition}\label{weakmain} Under the same assumptions as in Theorem \ref{main},
there exists a closed subset $\Sigma\subset\Omega\times (0,\infty)$, with
$\mathcal{P}^\frac53(\Sigma)=0$, such that $(\u, Q)\in C^\infty(\O\times (0,\infty)\setminus\Sigma)$.
\end{Proposition}
\begin{proof} First it follows from Lemma \ref{max1} and Lemma \ref{max0} that for any $\delta>0$,
$Q$ and $f_{\rm{BM}}(Q)$ are bounded in $\Omega\times (\delta,\infty)$. Define
$$\Sigma_\delta=\Big\{z\in\Omega\times (\delta,\infty):
\ \liminf_{r\to 0} r^{-2}\int_{\mathbb P_{r}(z)} (|\u|^3+|\nabla Q|^3)\,dxdt
+\big(r^{-2}\int_{\mathbb P_{r}(z)} |P|^\frac32\,dxdt\big)^2>\varepsilon_0^3\Big\}.
$$
From Lemma \ref{small-reg1}, we know that $\Sigma_\delta$ is closed and
$(\u, Q)\in C^\infty(\Omega\times (\delta,\infty)\setminus \Sigma_\delta)$. 
Since $\delta>0$ is arbitrary, we have that $(\u, Q)\in C^\infty(\O\times(0,\infty)
\setminus \cup_{\delta>0} \Sigma_\delta)$.

Since $u \in L^\infty_tL^2_x\cap L^2_tH^1_x(\O\times (0,\infty))$ and
$\nabla Q \in L^\infty_tH^1_x\cap L^2_tH^2_x(\O\times (0,\infty))$, we see that
$(\u,\nabla Q)\in L^{\frac{10}{3}}(\O\times (0,\infty))$. Moreover, since $P$ solves
the Poisson equation \eqref{P-equation4} in $\Omega\times (0,\infty)$, we conclude
that $P\in L^\frac53(\O\times (0,\infty))$. By H\"older's inequality, we see that
$\Sigma_\delta$ is  a subset of
\begin{equation*}
\begin{split}
\mathcal{S}_\delta=\Big\{z\in \O\times (\delta, \infty):&
\ \ \liminf_{r\to 0} r^{-\frac53}\int_{\mathbb P_{r}(z)} (|\u|^{\frac{10}3}+|\nabla Q|^{\frac{10}3})\,dxdt\\
&+\big(r^{-\frac53}\int_{\mathbb P_{r}(z)} |P|^\frac53\,dxdt\big)^2>\varepsilon_0^{\frac{10}3}\Big\}.
\end{split}
\end{equation*}
A simple covering argument implies that $\mathcal{P}^\frac53(\mathcal{S}_\delta)=0$, see \cite{S}. Hence
$\Sigma=\cup_{\delta>0}\Sigma_\delta$ has $\mathcal{P}^\frac53(\Sigma)=0$. This completes
the proof.
\end{proof}

\section{Partial regularity, part II}

In this section, we will utilize the results from the previous section
and the Sobolev inequality to first show the so-called
A-B-C-D Lemmas (see \cite{CKN} and \cite{LIN})
and then establish an improved $\varepsilon_1$-regularity property
for suitable weak solutions to \eqref{e1}.

\begin{Theorem}\label{thm:strongregularity} Under the same assumptions as in Theorem \ref{main},
there exists $\varepsilon_1>0$ such that if
$({\bf u}, Q):\O\times (0,\infty) \mapsto \mathbb R^3\times \mathcal{S}_0^{(3)}$ is a suitable
weak solution of \eqref{e}, which satisfies, for  $z_0\in\Omega\times (0,\infty)$,
\begin{equation}\label{small_cond1}
\limsup_{r\to 0}\f{1}{r}\int_{\mathbb P_r(z_0)}\big(|\nabla\u|^2+|\nabla^2 Q|^2\big)dxdt<\varepsilon_1^2,
\end{equation}
then $({\bf u}, Q)$ is smooth near $z_0$.
\end{Theorem}

For simplicity, we assume $z_0=(0,0)\in\O\times (0,\infty)$.  To streamline the presentation, we introduce the following
dimensionless quantities:
\begin{align*}
  A(r)&:=\sup_{-r^2\leq t\leq 0}r^{-1}\int_{B_r(0)\times\left\{ t \right\}}(|{\bf u}|^2+|\nabla Q|^2)\,dx, \\
  B(r)&:=\frac{1}{r}\int_{\mathbb{P}_r(0,0)}(|\nabla {\bf u}|^2+|\nabla^2 Q|^2)\, dxdt,\\
  C(r)&:=\frac{1}{r^2}\int_{\mathbb{P}_r(0,0)}(|{\bf u}|^3+|\nabla Q|^3)\,dxdt, \\
  D(r)&:=r^{-2}\int_{\mathbb{P}_r(0,0)}|P|^{\frac{3}{2}}\, dxdt.
\end{align*}
Also set
$$({\bf u})_r(t):=\frac{1}{|B_r(0)|}\int_{B_r(0)}{\bf u}(x,t)\,dx, (\nabla Q)_r(t):=\frac{1}{|B_r(0)|}\int_{B_r(0)}\nabla Q(x, t)\,dx.$$ 

We recall the following interpolation Lemma, whose proof can be found in [1] and [3].
\begin{Lemma}
  For $v\in H^1(\R^3)$, 
  \begin{eqnarray}
    \int_{B_r(0)}|v|^q(x, t)\,dx&\lesssim& \big( \int_{B_r(0)}|\nabla v|^2(x, t)\, dx \big)^{\frac{q}{2}-a}
    \big( \int_{B_r(0)}|v|^2(x,t)\,dx \big)^a\nonumber\\
    &&+r^{3\big( 1-\frac{q}{2} \big)}\big( \int_{B_r(0)}|v|^2(x, t)\,dx \big)^{\frac{q}{2}}. 
    \label{}
  \end{eqnarray}
  for every $B_r(0)\subset\R^3$, $2\leq q\leq 6$, $a=\frac{3}{2}\big( 1-\frac{q}{6} \big).$
  \label{eqn:interpolation}
\end{Lemma}

Applying Lemma \ref{eqn:interpolation}, we can have 
\begin{Lemma}  \label{lem:CAB}
  For any ${\bf u}\in L^{\infty}([-\rho^2, 0], L^2(B_\rho(0)))
  \cap L^2([-\rho^2, 0], H^1(B_\rho(0)))$, and
  $Q\in L^\infty([-\rho^2, 0], H^1(B_\rho(0)))\cap L^2([-\rho^2, 0], H^2(B_\rho(0)))$,
   it holds that for any $0<r\leq \rho$,
  \begin{equation}
    C(r)\lesssim \big( \frac{r}{\rho} \big)^3A^{\frac{3}{2}}(\rho)+\big( \frac{\rho}{r} \big)^3A^{\frac{3}{4}}(\rho)B^{\frac{3}{4}}(\rho). 
    \label{eqn:CAB}
  \end{equation}
\end{Lemma}
\begin{proof}
  From \eqref{eqn:interpolation} with $q=3, a=\frac{3}{4}$, we obtain that for any $v\in H^1(B_\rho(0))$, 
  \begin{eqnarray}
    \int_{B_r(0)}|v|^3(x, t)\, dx&\lesssim&\big( \int_{B_r(0)}|\nabla v|^2(x, t)\,dx \big)^{\frac{3}{4}}
    \big( \int_{B_r(0)}|v|^2(x,t)\,dx \big)^{\frac{3}{4}}
    \nonumber\\
    &&+r^{-\frac{3}{2}}\big( \int_{B_r(0)}|v|^2(x, t)\,dx \big)^{3/2}.
    \label{}
  \end{eqnarray}  
Applying Poincar\'e's inequality, we obtain that for $0<r\le \rho$, 
\begin{align*}
 & \int_{B_r(0)}(|{\bf u}|^2+|\nabla Q|^2)\,dx\\
 &\lesssim \int_{B_r(0)}\left(\left||{\bf u}|^2-(|{\bf u}|^2)_{\rho}\right|+\left||\nabla Q|^2-(|\nabla Q|^2)_{\rho}\right|\right)\,dx+\big( \frac{r}{\rho} \big)^3\int_{B_\rho(0)}(|{\bf u}|^2+|\nabla Q|^2)\, dx\\
  &\lesssim \rho\int_{B_\rho(0)}(|{\bf u}||\nabla {\bf u}|+|\nabla Q||\nabla^2 Q|)\,dx
  +\big( \frac{r}{\rho} \big)^3\int_{B_\rho(0)}(|{\bf u}|^2+|\nabla Q|^2)\,dx\\
  &\lesssim\rho^{\frac{3}{2}}\big( \rho^{-1}\int_{B_\rho(0)}(|{\bf u}|^2+|\nabla Q|^2)\,dx \big)^{\frac{1}{2}}
  \big( \int_{B_\rho(0)}(|\nabla {\bf u}|^2
  +|\nabla^2 Q|^2)\,dx \big)^{\frac{1}{2}}\\
  &\quad +\big( \frac{r}{\rho} \big)^3\int_{B_\rho(0)}(|{\bf u}|^2+|\nabla Q|^2)\, dx\\
  &\lesssim \rho^{\frac{3}{2}}A^{\frac{1}{2}}(\rho)\big( \int_{B_\rho(0)}(|\nabla {\bf u}|^2+|\nabla^2 Q|^2)\,dx \big)^{\frac{1}{2}}
  +\big( \frac{r}{\rho} \big)^3\rho A(\rho). 
\end{align*}
Substituting this estimate into the second term of the right hand side of the previous inequality, we conclude that 
\begin{align*}
  &\int_{B_r(0)}\left( |{\bf u}|^3+|\nabla Q|^3 \right)\,dx\\
  &\lesssim \rho^{\frac{3}{4}}\big(\int_{B_r(0)}\left( |\nabla {\bf u}|^2+
  |\nabla^2 Q|^2\right) \,dx\big)^{\frac{3}{4}}\big(\rho^{-1} \int_{B_r(0)}(|{\bf u}|^2+|\nabla Q|^2)(x, t)\,dx \big)^{\frac{3}{4}}\\
  &\quad +r^{-\frac{3}{2}}\big( \int_{B_r(0)}(|{\bf u}|^2+|\nabla Q|^2)(x, t)\,dx \big)^{\frac{3}{2}}\\
  &\lesssim \rho^{\frac{3}{4}}A^{\frac{3}{4}}(\rho)\big( \int_{B_r(0)}(|\nabla {\bf u}|^2+|\nabla^2 Q|^2)(x, t)\,dx \big)^{\frac{3}{4}}\\
  &\quad +r^{-\frac{3}{2}}\big( \int_{B_r(0)}(|{\bf u}|^2+|\nabla Q|^2)(x, t)\,dx \big)^{\frac{3}{2}}\\
  &\lesssim\big( \rho^{\frac{3}{4}}+\frac{\rho^{\frac{9}{4}}}{r^{\frac{3}{2}}} \big)\big( \int_{B_r(0)}(|\nabla {\bf u}|^2+|\nabla^2 Q|^2)\,dx \big)^{\frac{3}{4}}A^{\frac{3}{4}}(\rho)+\big( \frac{r}{\rho} \big)^3 A^{\frac{3}{2}}(\rho). 
\end{align*}
Integrating this inequality over $[-r^2,0]$, by H\"{o}lder's inequality we have
\begin{align*}
  &C(r)=\frac{1}{r^2}\int_{\mathbb{P}_r(0,0)}(|{\bf u}|^3+|\nabla Q|^3)\,dx\\
  &\lesssim \big( \frac{r}{\rho} \big)^3 A^{\frac{3}{2}}(\rho)+\big( \rho^{\frac{3}{4}}+\frac{\rho^{\frac{9}{4}}}{r^{\frac{3}{2}}} \big)\int_{-r^2}^{0}\big( \int_{B_r(0)}(|\nabla {\bf u}|^2+|\nabla^2 Q|^2)\,dx \big)^{\frac{3}{4}}\,dt A^{\frac{3}{4}}(\rho)\\
  &\lesssim \big( \frac{r}{\rho} \big)^3 A^{\frac{3}{2}}(\rho)+r^{-\frac{3}{2}}\rho^{\frac{3}{4}}\big( \rho^{\frac{3}{4}}+\frac{\rho^{\frac{9}{4}}}{r^{\frac{3}{2}}} \big)A^{\frac{3}{4}}(\rho)B^{\frac{3}{4}}(\rho)\\
  &\lesssim\big( \frac{r}{\rho} \big)^3 A^{\frac{3}{2}}(\rho)+\big[\big( \frac{\rho}{r} \big)^{\frac{3}{2}}+\big( \frac{\rho}{r} \big)^3 \big]A^{\frac{3}{4}}(\rho)B^{\frac{3}{4}}(\rho)\\
  &\lesssim \big( \frac{r}{\rho} \big)^3 A^{\frac{3}{2}}(\rho)+\big( \frac{\rho}{r} \big)^3A^{\frac{3}{4}}(\rho)B^{\frac{3}{4}}(\rho).
\end{align*}
This completes the proof of (5.2). \end{proof}

Next we want to estimate the pressure function.
\begin{Lemma}
  Under the same assumption with Lemma \ref{lem:CAB}, it holds for any $0<r\leq\frac{\rho}{2} $
  \begin{equation}
 D(r)\lesssim \frac{r}{\rho}D(\rho)+\big( \frac{\rho}{r} \big)^2A^{\frac{3}{4}}(\rho)B^{\frac{3}{4}}(\rho).
 \label{eqn:DAB}
  \end{equation}
\label{lem:DAB}
\end{Lemma}
\begin{proof}
  From the scaling invariance of all quantities, we only need to consider the case $\rho=1$, $0<r\leq \frac{1}{2}$.   By taking divergence of the equation \eqref{e}$_1$, we obtain
  \begin{eqnarray}  \label{P-equation6}
    &&-\Delta P={\rm{div}}^2\left[ {\bf u}\otimes {\bf u}+\nabla Q\otimes \nabla Q \right]\nonumber\\
   &=&{\rm{div}}^2\left[ ({\bf u}-({\bf u})_1)\otimes({\bf u}-({\bf u})_1)+ \nabla Q\otimes\nabla Q \right]\nonumber\\
   &=&{\rm{div}}^2[({\bf u}-({\bf u})_1)\otimes({\bf u}-({\bf u})_1)+(\nabla Q-(\nabla Q)_1)\otimes(\nabla Q-(\nabla Q)_1)]\nonumber\\
   &&+{\rm{div}}^2[(\nabla Q)_1\otimes (\nabla Q-(\nabla Q)_1)+(\nabla Q-(\nabla Q)_1)\otimes (\nabla Q)_1].
  \end{eqnarray}
Let $\eta\in C_0^\infty(\R^3)$ be a cut off function of $B_{\frac{1}{2}}(0)$ such that 
  \begin{equation}
    \left\{
      \begin{array}{ll}
	\eta=1, &\text{ in }B_{\frac{1}{2}}(0), \\
	\eta=0, & \text{ in }\R^3\setminus B_1(0), \\
	0\leq \eta\leq 1, & |\nabla \eta|\leq 8. 
      \end{array}
      \right.
      \label{}
    \end{equation}
    Define the following auxillary function
    \begin{align*}
      P_1(x, t)&=-\int_{\R^3}\nabla_y^2 G(x-y):\eta^2(y)\big[ ({\bf u}-({\bf u})_1)\otimes ({\bf u}-({\bf u})_1)\\
      &\qquad +(\nabla Q-(\nabla Q)_1)\otimes (\nabla Q-(\nabla Q)_1)+(\nabla Q-(\nabla Q)_1)\otimes(\nabla Q)_1\\
      &\qquad +(\nabla Q)_1\otimes (\nabla Q-(\nabla Q)_1)\big](y, t)\,dy,
    \end{align*}
    Then we have
    $$-\Delta P_1={\rm{div}}^2\left[({\bf u}-({\bf u})_1)\otimes({\bf u}-({\bf u})_1)+\nabla Q\otimes \nabla Q\right] \text{ in }B_{\frac{1}{2}}(0),$$
    and
    $$-\Delta (P-P_1)=0 \text{ in }B_{\frac{1}{2}}(0).$$
    For $P_1$, we apply the Calderon-Zygmund theory to deduce
    \begin{eqnarray}
      \left\|P_1\right\|_{L^{\frac{3}{2}}(\R^3)}^{\frac{3}{2}}&\lesssim& \left\|\eta^2|{\bf u}-({\bf u})_1|^2\right\|_{L^{\frac{3}{2}}(\R^3)}^{\frac{3}{2}}+\left\|\eta^2|\nabla Q-(\nabla Q)_1|^2\right\|_{L^{\frac{3}{2}}(\R^3)}^{\frac{3}{2}}\nonumber\\
      &&+\left\|\eta^2|(\nabla Q)_1||\nabla Q-(\nabla Q)_1|\right\|_{L^{\frac{3}{2}}(\R^3)}^{\frac{3}{2}}\nonumber\\
     &\lesssim& \int_{B_1(0)}(|{\bf u}-({\bf u})_1|^3+|\nabla Q-(\nabla Q)_1|^3)\,dx\nonumber\\
     &&+ |(\nabla Q)_1|^\frac32\int_{B_1(0)}|\nabla Q-(\nabla Q)_1|^\frac32\,dx.
      \label{}
    \end{eqnarray}
Since $P-P_1$ is harmonic in $B_\frac12(0)$, we get
$$\frac{1}{r^2}\left\|P-P_1\right\|_{L^{\frac{3}{2}}(B_r(0))}^{\frac{3}{2}}\lesssim r\left\|P-P_1\right\|_{L^{\frac{3}{2}}(B_1(0))}^{\frac{3}{2}}\lesssim r\big( \left\|P\right\|_{L^{\frac{3}{2}}(B_1(0))}^{\frac{3}{2}}+\left\|P_1\right\|_{L^{\frac{3}{2}}(B_1(0))}^{\frac{3}{2}} \big).$$
Integrating it over $[-r^2, 0]$ and applying (5.8),  we can show that
 \begin{eqnarray*}
 &&\frac{1}{r^2}\int_{\PP_r(0,0)}|P|^{\frac{3}{2}}\,dxdt\\
 &\lesssim&
 r\int_{\PP_1(0,0)}|P|^{\frac{3}{2}}\,dxdt+\frac{1}{r^2}\int_{\PP_1(0,0)}(|{\bf u}-({\bf u})_1|^3+|\nabla Q-(\nabla Q)_1|^3)\,dxdt\\
 &&+ \frac{1}{r^2} \big(\sup_{-1\le t\le 0} |(\nabla Q)_1(t)|\big)^\frac32\int_{\PP_1(0,0)}|\nabla Q-(\nabla Q)_1|^\frac32\,dxdt\\
 &\lesssim&r\int_{\PP_1(0,0)}|P|^{\frac{3}{2}}\,dxdt+\frac{1}{r^2}\int_{\PP_1(0,0)}(|{\bf u}-({\bf u})_1|^3+|\nabla Q-(\nabla Q)_1|^3)\,dxdt\\
 &&+\frac{1}{r^2}A^\frac34(1)\int_{\PP_1(0,0)}|\nabla Q-(\nabla Q)_1|^\frac32\,dxdt.
 \end{eqnarray*}
 This, combined with the interpolation inequality
 \begin{eqnarray*}
  && \int_{\PP_1(0,0)}(|{\bf u}-({\bf u})_1|^3+|\nabla Q-(\nabla Q)_1|^3)\,dxdt\nonumber\\
  &&{\lesssim \sup_{-1\leq t\leq 0}\big( \int_{B_1(0)}(|{\bf u}|^2+|\nabla Q|^2)\,dx \big)^{\frac{3}{4}}
   \times \big( \int_{\PP_1(0,0)}(|\nabla {\bf u}|^2+|\nabla^2 Q|^2 )\,dxdt\big)^{\frac{3}{4}}, }
   \label{}
 \end{eqnarray*}
and H\"older's inequality
$$\int_{\PP_1(0,0)}|\nabla Q-(\nabla Q)_1|^\frac32\,dxdt\lesssim 
\big(\int_{\PP_1(0,0)}|\nabla Q-(\nabla Q)_1|^2\,dxdt\big)^\frac34,$$
implies that
 $$D(r)\lesssim r D(1)+\frac{1}{r^2} A^{\frac{3}{4}}(1)B^{\frac{3}{4}}(1).$$
This, after scaling back to $\rho$, yields \eqref{eqn:DAB}. The proof is now complete.
 \end{proof}
 \begin{proof}[Proof of Theorem \ref{thm:strongregularity}]
 For $\theta\in (0,\frac12)$ and $\rho\in (0,1)$,
 let $\varphi\in C_0^\infty(\PP_{\theta\rho}(0,0))$ be a function such that 
 $$\varphi=1 \ {\rm{in}}\ \PP_{\frac{\theta \rho}2}(0,0), \
 |\nabla \varphi|\lesssim \frac{1}{\theta\rho}, \ |\nabla ^2 \varphi|+|\varphi_t|\lesssim \frac{1}{\theta^2\rho^2}.
 $$
 Applying the local energy inequality in Lemma 2.2, the maximum principles Lemmas \ref{max1} and \ref{max0}, and the integration by parts,
 we obtain that
 {
 \begin{align*}
&\sup_{-(\theta\rho)^2\leq t\leq 0}\int_{\O}(|{\bf u}|^2+|\nabla Q|^2)\varphi^2\,dx
+\int_{\O\times [-(\theta\rho)^2,0]}(|\nabla {\bf u}|^2+|\nabla^2 Q|^2)\varphi^2\,dxdt\\
   &\lesssim\int_{\O\times [-(\theta\rho)^2,0]}(|{\bf u}|^2+|\nabla Q|^2)(|\varphi_t|+|\nabla\varphi|^2+|\nabla^2 \varphi|)\,dxdt\\
   &\quad+\int_{\O\times [-(\theta\rho)^2,0]}[(|{\bf u}|^2-(|{\bf u}|^2)_{\theta\rho})+(|\nabla Q|^2-|\nabla Q|^2)_{\theta\rho})
   +|P|]|{\bf u}||\nabla \varphi|\,dxdt\\
   &\quad+\int_{\O\times [-(\theta\rho)^2,0]}|\nabla Q|^2\varphi^2\,dxdt
  +\int_{\O\times [-(\theta\rho)^2,0]}(|\nabla {\bf u}||\nabla Q|+|{\bf u}||\Delta Q|)|\varphi||\nabla\varphi|\,dxdt.  
\end{align*}   
This, with the help of Young's inequality:
\begin{align*}
&\int_{\O\times [-(\theta\rho)^2,0]}(|\nabla {\bf u}||\nabla Q|+|{\bf u}||\Delta Q|)|\varphi||\nabla\varphi|\,dxdt\\
&\leq \frac12 \int_{\O\times [-(\theta\rho)^2,0]}(|\nabla {\bf u}|^2+|\nabla^2 Q|^2)\varphi^2\,dxdt\\
&\ \ +4\int_{\O\times [-(\theta\rho)^2,0]}(|{\bf u}|^2+|\nabla Q|^2)|\nabla\varphi|^2\,dxdt,
\end{align*}
implies that   
   \begin{align*}
   &A(\frac12\theta\rho)+B(\frac12\theta\rho)\\
   &=\sup_{-(\frac{\theta\rho}2)^2\leq t\leq 0}\frac{2}{\theta\rho}\int_{B_{\frac{\theta\rho}2}(0)}(|{\bf u}|^2+|\nabla Q|^2)\,dx
+\frac{2}{\theta\rho}\int_{\PP_{\frac{\theta\rho}2}(0,0)}(|\nabla {\bf u}|^2+|\nabla^2 Q|^2)\,dxdt\\
   &\lesssim\sup_{-(\theta\rho)^2\leq t\leq 0}\frac{1}{\theta\rho}\int_{\R^3}(|{\bf u}|^2+|\nabla Q|^2)\varphi^2\,dx
+\frac{1}{\theta\rho}\int_{\R^3\times [-(\theta\rho)^2,0]}(|\nabla {\bf u}|^2+|\nabla^2 Q|^2)\varphi^2\,dxdt\\
&\lesssim \frac{1}{\theta\rho}\int_{\R^3\times [-(\theta\rho)^2,0]}(|{\bf u}|^2+|\nabla Q|^2)(|\varphi_t|+|\nabla\varphi|^2+|\nabla^2 \varphi|)\,dxdt\\
   &\ \ \ +\frac{1}{\theta\rho}\int_{\R^3\times [-(\theta\rho)^2,0]}[(|{\bf u}|^2-(|{\bf u}|^2)_{\theta\rho})+
   (|\nabla Q|^2-(|\nabla Q|^2)_{\theta\rho})+|P|]|{\bf u}||\nabla \varphi|\,dxdt\\
   &\ \ \ +\frac{1}{\theta\rho}\int_{\R^3\times [-(\theta\rho)^2,0]}|\nabla Q|^2\varphi^2\,dxdt\\
   &\lesssim\frac{1}{(\theta\rho)^3}\int_{\PP_{\theta\rho}(0,0)}(|{\bf u}|^2+|\nabla Q|^2)\,dxdt
+\frac{1}{(\theta\rho)^2}\int_{\PP_{\theta\rho}(0,0)}|P| |{\bf u}|\,dxdt\\
   &\quad+\frac{1}{(\theta\rho)^2}\int_{\PP_{\theta\rho}(0,0)}\left( ||{\bf u}|^2-(|{\bf u}|^2)_{\theta\rho}|
   +||\nabla Q|^2-(|\nabla Q|^2)_{\theta\rho}| \right)|{\bf u}|\,dxdt\\
  &=I_1+I_2+I_3.
    \end{align*}
  It is not hard to see that
  $$|I_1|\lesssim \big(\frac{1}{(\theta\rho)^2}\int_{\PP_{\theta\rho}(0,0)}(|{\bf u}|^3+|\nabla Q|^3)\,dxdt\big)^\frac23\lesssim C^\frac23(\theta\rho),$$
  $$|I_2|\lesssim \big(\frac{1}{(\theta\rho)^2}\int_{\PP_{\theta\rho}(0,0)}|{\bf u}|^3\,dxdt\big)^\frac13
  \big(\frac{1}{(\theta\rho)^2}\int_{\PP_{\theta\rho}(0,0)}|P|^\frac32\,dxdt\big)^\frac23\lesssim C^\frac13(\theta\rho) D^\frac23(\theta\rho),
  $$
 while, by employing H\"older's and Poincar\'e's inequalities, 
 \begin{eqnarray*}
 |I_3|&\lesssim& \frac{1}{(\theta\rho)^2} \int_{-(\theta\rho)^2}^0\int_{B_{\theta\rho}(0)} (|{\bf u}||\nabla{\bf u}|+|\nabla Q||\nabla^2 Q|)
 \big(\int_{B_{\theta\rho}(0)}|{\bf u}|^3+|\nabla Q|^3\big)^\frac13\,dt\\
 &\lesssim & A^\frac12(\theta\rho)B^\frac12(\theta\rho)C^\frac13(\theta\rho).
 \end{eqnarray*}   
  }
Putting together all the estimates, we have
\begin{align*}
  A(\frac12{\theta \rho})+B(\frac12{\theta\rho})&\lesssim\big[C^{\frac{2}{3}}(\theta\rho)
  +A^{\frac{1}{2}}(\theta\rho)B^{\frac{1}{2}}(\theta\rho)C^{\frac{1}{3}}(\theta\rho)+C^{\frac{1}{3}}(\theta\rho)D^{\frac{2}{3}}(\theta\rho)\big]\\
  &\lesssim \big[C^{\frac{2}{3}}(\theta\rho)
  +A(\theta\rho)B(\theta\rho)+D^{\frac{4}{3}}(\theta\rho)\big]
  \end{align*}
  so that
  \begin{align*}
   A^\frac32(\frac12{\theta \rho})\lesssim\big[C(\theta\rho)
  +A^\frac32(\theta\rho)B^\frac32(\theta\rho)+D^2(\theta\rho)\big].
  \end{align*}
While
  \begin{align*}
  D^2(\theta\rho)&\lesssim \theta^2 \big[D^2(\rho)+\theta^{-6} A^{\frac{3}{2}}(\rho)B^{\frac{3}{2}}(\rho)\big],
  \end{align*}
  and
  \begin{align*}
  C(\theta \rho)&\lesssim \theta^3 A^{\frac{3}{2}}(\rho)+\theta^{-3} A^{\frac{3}{4}}(\rho)B^{\frac{3}{4}}(\rho). 
\end{align*}
Also note that
$$A^\frac32(\theta\rho) B^\frac32(\theta\rho)\le \theta^{-3} A^\frac32(\rho)B^\frac32(\rho).$$
Therefore we conclude that for $0<\theta_0<\frac12$,
\begin{align*}
  &A^\frac32(\frac12{\theta_0 \rho})+D^2(\frac12{\theta_0 \rho})\\
  &\leq c[\theta_0^2D^2(\rho)+(\theta_0^{-3}+\theta_0^{-4})A^\frac32(\rho)B^\frac32(\rho)+\theta_0^3A^\frac32(\rho)+\theta_0^{-3}A^\frac34(\rho) B^\frac34(\rho)]\\
  &\leq c[\theta_0^2(D^2(\rho)+A^\frac32(\rho))+\theta_0^{-8}A^\frac32(\rho)B^\frac32(\rho)+\theta_0^2]\\
  &\leq c(\theta_0^2+\theta_0^{-8} B^\frac32(\rho))(A^\frac32(\rho)+D^2(\rho))+c\theta_0^2.
  \end{align*}
For $\varepsilon_1>0$ given by Theorem 5.1, let $\theta_0\in (0,\frac12)$ such that
$$c\theta_0^2=\min\big\{\frac14, \frac12\varepsilon_1^2\big\}.$$
From \eqref{small_cond1}, we know that
$$\limsup_{\rho\to 0} B(\rho)\le \varepsilon_1^2,$$
hence there exists $\rho_0>0$ such that  
$$c\theta_0^{-8}B^\frac32(\rho)\le \frac14, \ \forall 0<\rho<\rho_0.$$
Therefore we conclude that there exist $\theta_0\in (0,\frac12)$ and $\rho_0>0$ such that
\begin{align*}
 A^\frac32(\frac12{\theta_0 \rho})+D^2(\frac12{\theta_0 \rho})\le \frac12 (A^\frac32(\rho)+D^2(\rho))
 +\frac12\varepsilon_1^2, \ \forall 0<\rho<\rho_0.
 \end{align*} 
 Iterating this inequality yields that
 \begin{align}\label{decay1}
 A^\frac32((\frac12{\theta_0)^k \rho})+D^2((\frac12{\theta_0)^k \rho})\le \frac{1}{2^k} (A^\frac32(\rho)+D^2(\rho))
 +\varepsilon_1^2
 \end{align} 
holds for all $0<\rho<\rho_0$ and $k\ge 1$.

Employing (5.2) and \eqref{decay1}, we obtain that
\begin{align}\label{decay2}
C((\frac12\theta_0)^{k}\rho)&\le c\big[(\frac12\theta_0)^3 A^\frac32((\frac12\theta_0)^{k-1}\rho)
+(\frac12\theta_0)^{-3} A^\frac34((\frac12\theta_0)^{k-1}\rho) B^\frac34((\frac12\theta_0)^{k-1}\rho)\big]\nonumber\\
&\le c\big[(\frac12\theta_0)^3 +(\frac12\theta_0)^{-3} \varepsilon_1^\frac32\big]
\big[\frac{1}{2^{k-1}} (A^\frac32(\rho)+D^2(\rho))
 +\varepsilon_1^2\big]
\end{align}
holds for all $0<\rho<\rho_0$ and $k\ge 1$.

Putting \eqref{decay1} and \eqref{decay2} together, we obtain that
\begin{align}\label{decay3}
\limsup_{k\to\infty} \big[C((\frac12\theta_0)^{k}\rho)+D^2((\frac12\theta_0)^{k}\rho)\big]
&\le c\big[1+(\frac12\theta_0)^3 +(\frac12\theta_0)^{-3} \varepsilon_1^\frac32]\varepsilon_1^2
\le \frac12\varepsilon_0^3,
\end{align}
holds for all $\rho\in (0,\rho_0)$, provided $\varepsilon_1=\varepsilon_1(\theta_0, \varepsilon_0)>0$ is chosen
sufficiently small. Therefore, by Lemma 5.4 
$({\bf u}, Q, P)$ is smooth near $(0,0)$.  This completes the proof.
\end{proof}
Theorem \ref{main} can be proved by the following covering argument. 
Let $\Sigma$ be the singular set of suitable weak solutions $({\bf u}, Q, P)$. 
If $(x,t)\in \Sigma$,  then by the theorem \ref{thm:strongregularity}, 
\begin{equation}
    \limsup_{r\rightarrow0}\frac{1}{r}\int_{\PP_r(x,t)}(|\nabla {\bf u}|^2+|\nabla^2 Q|^2)\,dxdt\geq \varepsilon_1.
\end{equation}
Let $V$ be a neighborhood of $\Sigma$ and $\delta>0$ such that for all $(x, t)\in \Sigma$, we can find $r<\delta$ such that $\PP_r(x, t)\subset V$ and
$$\frac{1}{r}\int_{\PP_r(x, t)}\left( |\nabla {\bf u}|^2+|\nabla^2 Q|^2 \right)\,dxdt\geq \varepsilon_1.$$
By Vitali's covering lemma, $\exists (x_i, t_i)\in V, 0<r_i<\delta$ such that $\left\{ \PP_{r_i}(x_i, t_i) \right\}_{i=1}^\infty$ are pairwise disjoint and 
$$\Sigma\subset \bigcup_{i=1}^\infty \PP_{5r_i}(x_i, t_i).$$
Hence
\begin{align*}
&\mathcal{P}^1_{5\delta}(\Sigma)\le 
  \sum_{i=1}^{\infty}5r_i\leq \frac{5}{\varepsilon_1}\sum_{i=1}^{\infty}\int_{\PP_{r_i}(x_i, t_i)}\left( |\nabla {\bf u}|^2+|\nabla^2 Q|^2 \right)
  \,dxdt\\
  &\leq \frac{5}{\varepsilon_1}\int_{\cup_i \PP_{r_i}(x_i, t_i)}\left( |\nabla {\bf u}|^2+|\nabla^2 Q|^2 \right)\,dxdt\\
  &\leq \frac{5}{\varepsilon_1}\int_{V}\left( |\nabla {\bf u}|^2+|\nabla^2 Q|^2 \right)\,dxdt<\infty. 
\end{align*}
We can conclude that $\Sigma$ is of zero Lesbegue measure. Then we can choose $|V|$ to be arbitrarily small, from the fact that
{$$\int_0^\infty\int_{\O}\left( |\nabla {\bf u}|^2+|\nabla^2 Q|^2 \right)\,dxdt
=\int_{0}^{\infty}\int_{\O}\left( |\nabla {\bf u}|^2+|\Delta Q|^2 \right)
\,dxdt<\infty$$}
and the absolute continuity of integral, we have
$$\lim_{|V|\rightarrow0}\int_{V}\left( |\nabla {\bf u}|^2+|\nabla^2 Q|^2 \right)\,dxdt\rightarrow0.$$
Hence
$$\mathcal{P}^1(\Sigma)=\lim_{\delta\rightarrow0}\mathcal{P}^1_{5\delta}(\Sigma)=0, $$
This completes the proof of Theorem \ref{main}. \qed

\bigskip\bigskip
\noindent{\bf Acknowledgements}.  Both the first and third authors are partially supported by NSF grant 1764417.
The second author is partially supported by the GRF grant (Project No. CityU 11332216). The third author wishes to express his gratitude to Professor Fanghua Lin for helpful discussions related to the blowing up argument in this paper.

\end{document}